\documentclass[11pt,a4paper]{amsart}
\usepackage{amsmath}
\usepackage{amssymb}
\usepackage{amsthm}
\usepackage{latexsym}
\usepackage{pstricks}
\usepackage{amsbsy}
\usepackage{tikz-cd}
\usepackage{mathtools}
\usepackage{faktor}  
\usepackage{soul}

\tikzset{curve/.style={settings={#1},to path={(\tikztostart)
    .. controls ($(\tikztostart)!\pv{pos}!(\tikztotarget)!\pv{height}!270:(\tikztotarget)$)
    and ($(\tikztostart)!1-\pv{pos}!(\tikztotarget)!\pv{height}!270:(\tikztotarget)$)
    .. (\tikztotarget)\tikztonodes}},
    settings/.code={\tikzset{quiver/.cd,#1}
        \def\pv##1{\pgfkeysvalueof{/tikz/quiver/##1}}},
    quiver/.cd,pos/.initial=0.35,height/.initial=0}
\usetikzlibrary{calc}
\tikzset{tail reversed/.code={\pgfsetarrowsstart{tikzcd to}}}
\tikzset{2tail/.code={\pgfsetarrowsstart{Implies[reversed]}}}
\tikzset{2tail reversed/.code={\pgfsetarrowsstart{Implies}}}

\usepackage{array}
\usepackage[all]{xy}
\usepackage{amscd}
\usepackage{mathrsfs}
\usepackage{txfonts}
\usepackage{amsfonts}
\usepackage{graphicx}
\usepackage{listings}
\usepackage{url}
\usepackage{bookmark}
\newtheorem{theorem}{Theorem}[section]
\newtheorem{lemma}[theorem]{Lemma}
\newtheorem{corollary}[theorem]{Corollary}
\newtheorem{proposition}[theorem]{Proposition}
\newtheorem{definition}[theorem]{Definition}

\newtheorem{remark}[theorem]{Remark}
\newtheorem{example}[theorem]{Example}
\newtheorem{conjecture}[theorem]{Conjecture}

\usepackage{marginnote}

\begin{document}
\newcommand{\abs}[1]{\left\lvert#1\right\rvert}
\newcommand{\news}[1]{{\color{green} #1}}
\newcommand{\new}[1]{{\color{blue} #1}}
\newcommand{\newp}[1]{{\color{purple} #1}}
\newcommand{\old}[1]{{\color{red} #1}}
\definecolor{darkgreen}{rgb}{0.33, 0.42, 0.18}
\newcommand{\rjm}[1]{{\color{darkgreen} #1}}

\newcommand{\proj}{\mathop{\rm proj}\nolimits}%
\newcommand{\inj}{\mathop{\rm inj}\nolimits}%
\providecommand{\Gen}{\mathop{\rm Gen}\nolimits}%
\providecommand{\Cogen}{\mathop{\rm Cogen}\nolimits}%
\providecommand{\Filt}{\mathop{\rm Filt}\nolimits}%
\providecommand{\cone}{\mathop{\rm cone}\nolimits}%
\providecommand{\Sub}{\mathop{\rm Sub}\nolimits}%
\providecommand{\rad}{\mathop{\rm rad}\nolimits}%
\providecommand{\coker}{\mathop{\rm coker}\nolimits}%
\providecommand{\im}{\mathop{\rm im}\nolimits}%
\providecommand{\btop}{\mathop{\rm top}\nolimits}%

\def\A{\mathcal{A}}
\def\C{\mathcal{C}}
\def\D{\mathcal{D}}
\def\P{\mathcal{P}}
\def\Y{\mathbb{Y}}
\def\Z{\mathcal{Z}}
\def\K{\mathcal{K}}
\def\F{\mathcal{F}}
\def\Eps{\mathcal{E}}
\def\T{\mathcal{T}}
\def\U{\mathcal{U}}
\def\V{\mathcal{V}}
\def\W{\mathcal{W}}
\def\E{\mathsf{E}}
\def\M{\mathsf{M}}
\def\N{\mathsf{N}}
\def\X{\mathsf{X}}
\def\L{\mathsf{L}}
\def\Mcluster{\mathfrak{M}}

\def\PP{{\mathbb P}}
\providecommand{\add}{\mathop{\rm add}\nolimits}%
\providecommand{\End}{\mathop{\rm End}\nolimits}%
\providecommand{\Ext}{\mathop{\rm Ext}\nolimits}%
\providecommand{\Yext}{\mathop{\rm Yext}\nolimits}%
\providecommand{\Hom}{\mathop{\rm Hom}\nolimits}%
\providecommand{\ind}{\mathop{\rm ind}\nolimits}%
\providecommand{\pd}{\mathop{\rm pd}\nolimits}%
\providecommand{\id}{\mathop{\rm id}\nolimits}%
\providecommand{\gldim}{\mathop{\rm gl.dim}\nolimits}%
\newcommand{\modd}{\mathop{\rm mod}\nolimits}%
\newcommand{\res}{\mathop{\mathsf{res}}\nolimits}%
\newcommand{\stautilt}{\textnormal{s$\tau$-tilt }}
\newcommand{\twosilt}{\textnormal{2-silt }}
\newcommand{\taugenmin}{\textnormal{$\tau$-gen-min }}
\newcommand{\ftors}{\textnormal{f-tors }}

\title{$\tau$-exceptional sequences for representations of quivers over local algebras}
\author[I. Nonis]{Iacopo Nonis}
\address{School of Mathematics \\ 
University of Leeds \\ 
Leeds, LS2 9JT \\ 
United Kingdom \\
}
\email{mmin@leeds.ac.uk}

\begin{abstract}
    Let $k$ be an algebraically closed field. Let $R$ be a finite dimensional commutative local $k$-algebra and let $Q$ be a quiver with no oriented cycles. In this paper, we study (signed) $\tau$-exceptional sequences over the algebra $\Lambda = R\otimes kQ$, which is isomorphic to  $RQ$. We show there is a bijection between the set of complete (signed) $\tau$-exceptional sequences in $\modd kQ$ and the set of complete (signed) $\tau$-exceptional sequences in $\modd\Lambda$. Moreover, we prove that every $\tau$-perpendicular subcategory of $\modd \Lambda$ is equivalent to the module category of $R\otimes kQ'$, for some quiver $Q'$. As a consequence, we prove that the $\tau$-cluster morphism categories of $kQ$ and $\Lambda$ are equivalent. 
\end{abstract}


\maketitle
\tableofcontents  

\section*{Introduction}

An object $M$ in an abelian or triangulated category is called \textit{exceptional} if its endomorphism algebra is a division ring and $\Ext^{\geq1}(M,M) = 0$. A sequence of indecomposable objects $\M = (M_1\cdots, M_t)$ is an \textit{exceptional sequence} if every object is exceptional and $\Hom(M_i, M_j) = 0 = \Ext^{\geq 1}(M_i,M_j)$ for $1\leq j < i \leq t$. Exceptional sequences were first introduced in algebraic geometry \cite{Bondal, Gorodentsev, GR}, and later considered in the setting of the representation theory of finite dimensional hereditary algebras by Crawley-Boevey in \cite{Boevey-ExcSeq} and Ringel in \cite{braid_group_action_on_exc_seq_Ringel}. 

If the length of an exceptional sequence of modules over a finite dimensional algebra equals the number of simple modules, the sequence is said to be \textit{complete}.

It was shown in \cite{braid_group_action_on_exc_seq_Ringel, Boevey-ExcSeq} that there is a transitive braid action on the set of complete exceptional sequences over a hereditary algebra. In this way, one can find all exceptional sequences and exceptional modules over a hereditary algebra. 

Complete sequences always exist over a hereditary algebra \cite{Boevey-ExcSeq}. However, this is not the case over an arbitrary finite dimensional algebra in general. 

Motivated by $\tau$-tilting theory introduced by Adachi, Iyama, and Reiten in \cite{tau-tiling-theory} and the notion of \textit{signed exceptional sequences} over a hereditary algebra $H$ introduced by Igusa and Todorov in \cite{signed_exc_seq}; Buan and Marsh recently defined the concept of (signed) $\tau$-\textit{exceptional sequnces} \cite{tauExcSeq_BM}. 

They proved that complete $\tau$-exceptional sequences always exist over an arbitrary finite dimensional algebra $\Lambda$ and that exceptional and $\tau$-exceptional sequences coincide when $\Lambda$ is hereditary. Moreover, they established a bijection between complete signed $\tau$-exceptional sequences and ordered support $\tau$-tilting objects, which generalizes the bijective correspondence between complete signed exceptional sequences and the ordered cluster-tilting objects in the cluster category (see \cite{Tilting_Theory_and_Cluster_Combinatorics}) corresponding to $H$ given by Igusa and Todorov. Signed exceptional sequences were originally defined to describe morphisms in the \textit{cluster morphism category} of $H$, whose objects are wide subcategories of $\modd H$. 

A key ingredient in the definition of a $\tau$-exceptional sequence is the $\tau$-\textit{perpendicular subcategory}, introduced by Jasso in \cite{Jasso_Reduction} as a generalization of the notion of the Geigle and Lenzing \textit{perpendicular subcategory} \cite{Geigle_Lenzing_Perp_Subcat}. 

Buan and Marsh extended the definition of the cluster morphism category to a $\tau$-tilting finite algebra in \cite{a_category_of_wide_subcategories}, which is an algebra with finitely many isomorphism classes of $\tau$-tilting modules, see \cite{tauTiltingFiniteAlgebrasAnd_g-vectors}. Later, Hanson and Igusa \cite{Hanson-Igusa_tau-cluster_morphism_categories_and_picture_groups} gave it the name of $\tau$-\textit{cluster morphism category}. They proved that the classifying space of this category is a cube complex, whose fundamental group is the picture group of the algebra, see also \cite{igusa2016picture}. 

Recently, Buan and Hanson \cite{tau-perpendicular_wide_subategories} generalized the definition of the $\tau$-cluster morphism category to an arbitrary finite dimensional algebra. Moreover, the $\tau$-cluster morphism category has been also studied from a geometric perspective \cite{a_geometric_perspective_on_the_tau_cluster_morphism_category, Kaipel-the_cagtegory_of_partition_fan}.

Representations of a quiver $Q$ over the algebra of the dual numbers have been considered for the first time by Ringel and Zhang in \cite{RepOfQuiversOverTheDualNumbers}. An $RQ$-module $M$, where $R$ denotes the algebra of the dual numbers, is a $kQ$-module equipped with an endomorphism $\epsilon: M\to M$ such that $\epsilon^2 = 0$. These modules are known as \textit{differential graded} modules. $RQ$ can be viewed as the path algebra of the quiver $Q$ over $R=k[x]/(x^2)$. This is equivalent to viewing $RQ$ as the tensor product of $R$ and $kQ$ over the field $k$. Subsequently, Geiss, Leclerc, and Schr\"{o}er \cite{Quivers_with_relations_for_symmetrizable_Cartan_matrices_I} studied the algebras $R\otimes_k kQ$, where $R= k[x]/(x^t)$, for $t\geq 2$.  

We propose a further generalization given by the algebras $\Lambda = R\otimes_k kQ\cong RQ$, where $R$ is a finite dimensional local commutative algebra. In particular, this paper aims to study (signed) $\tau$-exceptional sequences over this class of algebras and their associated $\tau$-cluster morphism categories.

\subsection*{Notation and main results} Let $k$ be an algebraically closed field and let $\Lambda$ be a basic finite dimensional algebra over $k$. Denote by $\modd\Lambda$ the category of finitely generated left $\Lambda$-modules. Let $\proj\Lambda$ be the full subcategory of projective $\Lambda$-modules. 

Given a subcategory $\mathcal{X}\subseteq\modd\Lambda$, we denote by $\ind(\mathcal{X})$ the set of isomorphism classes of indecomposable objects in $\mathcal{X}$, and by $\P(\mathcal{X})$ the full subcategory of $\mathcal{X}$ consisting of all $\Ext$-projective modules in $\mathcal{X}$, i.e. the modules $P$ in $\mathcal{X}$ such that $\Ext_\Lambda^1(P,\mathcal{X}) = 0$. 

We define $\mathcal{X}^{\perp} = \{Y\in\modd\Lambda \mid \Hom(\mathcal{X},Y) = 0\}$, and define $^{\perp}\mathcal{X}$ dually. 

For $M\in\modd\Lambda$, we denote by $\add M$ (resp. $\Gen M$) the full subcategory of $\modd{\Lambda}$ whose objects are direct summands (resp. factors) of finite direct sums of copies of $M$.

For any basic $\Lambda$-module $X$, let $\abs{X}$ denote the number of indecomposable direct summands of $X$ and let $\abs{\Lambda} = n$. All modules are taken to be basic where possible and considered up to isomorphism. We denote by $D^b(\modd\Lambda)$ the bounded derived category of $\modd\Lambda$, and by $(-)[1]$ the shift functor. For an arbitrary module category $\W$, let $\tau_\W$ (or simply $\tau$ if no confusion appears) denote the Auslander-Reiten translation in $\W$. 

The \textit{Bongartz completion} of $M$, denoted by $B_M$, is given by the direct sum of all indecomposable modules in $\P({^\perp}\tau M)$. 

The tensor product $\otimes_k$ will be denoted by $\otimes$. Throughout this paper, we always denote by $R$ a finite dimensional commutative local $k$-algebra.\\

The paper is organized as follows. Let $Q$ be a quiver with no oriented cycles and let $\Lambda = R\otimes kQ $. In Section \ref{Section1}, we first recall definitions and results from $\tau$-tilting theory and we establish a bijection between indecomposable $\tau$-rigid $kQ$-modules and indecomposable $\tau$-rigid $\Lambda$-modules. More precisely, every indecomposable $\tau$-rigid $\Lambda$-module is of the form $\Lambda\otimes_{kQ}M$, for an indecomposable $\tau$-rigid $kQ$-module $M$.  As a consequence, we establish a bijection between support $\tau$-tilting $kQ$-modules and support $\tau$-tilting $\Lambda$-modules, and between functorially finite torsion classes in $\modd kQ$ and functorially finite torsion classes in $\modd \Lambda$. Furthermore, we characterize when $\Lambda$ is $\tau$-tilting finite. 

In section \ref{induction-2-term-approx}, we discuss the connection between induction functor, 2-term rigid objects, and approximations.

In section \ref{pdSection}, we study the relationship between the projective dimension of $M$ as a $\Lambda$-module and the projective dimension of $M$ as an object in a functorially finite wide subcategory $\W\subseteq \modd{\Lambda}$, where $\Lambda$ is an arbitrary finite dimensional algebra. We denote by $\pd_\W$ the projective dimension in $\W$. We prove the following result. 

\begin{proposition}[Proposition \ref{pdIsPreservedInW}]
    Let $\Lambda$ be a finite dimensional algebra. Let $m\geq 0$ and let $M\in\modd\Lambda$ with $\pd M = m$. Suppose $M$ lies in a functorially finite wide subcategory $\mathcal{W}$. Then,  $\pd_\mathcal{W} M \leq m$. 
\end{proposition}

Let $\Lambda$ be a finite dimensional algebra and let $M$ be a basic $\tau$-rigid $\Lambda$-module. Jasso \cite{Jasso_Reduction} defined the \textit{$\tau$-perpendicular category} associated to $M$ as $J(M) = {^{\perp}(\tau M)} \cap M^\perp$, and proved that $J(M)$ is equivalent to the module category of a finite dimensional algebra $\Gamma_M$, where $\Gamma_M = \End_\Lambda(B_M)^\mathrm{op}/I$ and $I$ is the ideal generated by all maps factoring through $M$; see \cite[Thm. 3.8]{Jasso_Reduction}. In particular, if $\Lambda$ is hereditary, so is $\Gamma_M$  \cite[Cor. 3.19(a)]{Jasso_Reduction}.  In Section \ref{section2}, we give a precise description of the $\tau$-perpendicular subcategories of $\modd{\Lambda}$, where $\Lambda = R\otimes kQ$. The following is our first main result.

\begin{theorem}[Theorem \ref{J(Lambda_otimes_kQ)=mod(R_otimes_kQ')}]
    Let $\Lambda = R\otimes kQ$. Let $M$ be a basic $\tau$-rigid $kQ$-module and let $\Lambda\otimes_{kQ}M$ be the corresponding basic $\tau$-rigid $\Lambda$-module. Then, $J(\Lambda\otimes_{kQ}M)\simeq \modd\Gamma_{\Lambda\otimes_{kQ}M}$, with $$ \Gamma_{\Lambda\otimes_{kQ}M} \cong R\otimes \Gamma_M \quad \text{and} \quad \Gamma_M = \End_{kQ}(B_M)^{\mathrm{op}}/\langle e_M\rangle$$
    where $e_M$ is the idempotent corresponding to the projective $\End_{kQ}(B_M)^\mathrm{op}$-module $\Hom_{kQ}(B_M,M)$. In particular, $\Gamma_M$ is hereditary. 
\end{theorem}

In Section \ref{section3} we recall the definition of (signed) $\tau$-exceptional sequences and prove our second main result. 

\begin{theorem}[Theorem \ref{bijection_of_tau_exc_seq}, Corollary \ref{bijection_signed_tau_exceptional_sequences}]\label{Main_Thm_Sec2}
        Let $\Lambda = R\otimes kQ$ and let $t\in\{1,\cdots,n\}$. Then, the induction functor $\Lambda\otimes_{kQ}-: \modd kQ\to \modd\Lambda$ induces a bijection between the set of (signed) ($\tau$-)exceptional sequences of length $t$ in $\modd{kQ}$ and the set of (signed) $\tau$-exeptional sequences of length $t$ in $\modd{\Lambda}$. 
\end{theorem}

Let $\Lambda = R\otimes kQ$. Let $\Mcluster(kQ)$ and $\Mcluster(\Lambda)$ denote the $\tau$-cluster morphism categories of $kQ$ and $\Lambda$, respectively. Theorem \ref{Main_Thm_Sec2} motivates us to study the relation between $\Mcluster(kQ)$ and $\Mcluster(\Lambda)$ in Section \ref{section4}. In particular, we prove our third main theorem. 

\begin{theorem}[Theorem \ref{Mcluster(kQ)_equvalent_to_Mcluster(Lambda)}]
    There is an equivalence of categories between $\Mcluster(kQ)$ and $\Mcluster(\Lambda)$. 
\end{theorem}

In addition, we prove that a conjecture by Buan and Hanson \cite[Conjecture 6.8]{tau-perpendicular_wide_subategories} holds for $\Lambda$; that is, if $\W\subseteq\modd{\Lambda}$ is a $\tau$-perpendicular subcategory of $\modd\Lambda$ and $\V\subseteq \W$ is a wide subcategory of $\W$, then $\V$ is a $\tau$-perpendicular subcategory of $\modd\Lambda$ if and only if $\V$ is a $\tau$-perpendicular subcategory of $\W$.  

In the final section, we conclude by presenting a concrete example that illustrates the theory developed in this work.

\section{$\tau$-rigid modules}\label{Section1}

In this section, we recall some facts about the tensor product of algebras and $\tau$-tilting theory, and we establish a bijection between indecomposable $\tau$-rigid  $\Lambda$-modules and indecomposable $\tau$-rigid  $kQ$-modules, where $\Lambda = R\otimes kQ$. \\

We start with summarizing some properties of the tensor product of algebras. The next two results will be used multiple times in our poofs.

\begin{proposition}\label{HomExt_tensor_product}
        Let $\Lambda$ and $\Gamma$ be finite dimensional algebras. Let $M_i\in \modd\Lambda$ and $N_i \in \modd\Gamma$, for $i = 1,2$. Then we have the following isomorphisms: 
    \begin{equation}\label{Hom_Tensor}
        \Hom_{\Lambda\otimes \Gamma}(M_1 \otimes N_1, M_2\otimes N_2) \cong \Hom_{\Lambda}(M_1, M_2)\otimes \Hom_{\Gamma}(N_1,  N_2).
    \end{equation}
    \begin{equation}\label{Ext_Tensor}
       \Ext_{\Lambda\otimes\Gamma}^m (M_1 \otimes N_1, M_2\otimes N_2)\cong \bigoplus_{i+j = m} \Ext_{\Lambda}^i(M_1,M_2) \otimes \Ext_{\Gamma}^j(N_1, N_2),  
    \end{equation}
    for $m\geq 1$. Moreover, for $M\in\modd\Lambda$ and $N\in\modd\Gamma$, we have that 
    \begin{equation}\label{pd_tensor_product}
        \pd_{\Lambda\otimes \Gamma}(M\otimes N) = \pd_\Lambda M + \pd_{\Gamma} N.
    \end{equation}
    Furthermore, if $M$ and $N$ are indecomposable, so is $M\otimes N$. 
\end{proposition}

\begin{proof}
    For the isomorphisms \eqref{Hom_Tensor}-\eqref{Ext_Tensor}, and the equation \eqref{pd_tensor_product}, see {\cite[Chapter IX]{CartanEilenberg}}. For the last statement see for example {\cite[Prop. 3.2]{Gproj_tau_tilting_modules}}.
\end{proof}

\begin{lemma}\label{tensor_of_minimal_proj_presentation}
    Let $\Lambda$ and $\Gamma$ be finite dimensional algebras. Let $N\in\modd\Gamma$, and let $Q_\bullet: Q_1\xrightarrow{d_1^N} Q_0\xrightarrow{d_0^N} N\to 0$ be its minimal projective presentation. Let $P$ be a projective $\Lambda$-module. Then,  $P\otimes Q_1\xrightarrow{\id\otimes d_1^N} P\otimes Q_0\xrightarrow{\id\otimes d_0^N} P\otimes N\to 0$ is a minimal projective presentation of $P\otimes N$ in $\modd(\Lambda\otimes \Gamma)$
\end{lemma}

\begin{proof}
It is well-known that projective $(\Lambda\otimes\Gamma)$-modules are of the form $P\otimes Q$, for $P\in\proj\Lambda$ and $Q\in\proj\Gamma$. Moreover, since $P$ is a projective $\Lambda$-module, the functor $P\otimes-$ is exact and preserves projective covers. The claim follows. 
\end{proof}

\begin{definition}[{\cite[Def. 0.1]{tau-tiling-theory}}]
    Let $\Lambda$ be a finite dimensional algebra and let $M$ in $\modd \Lambda$:
    \begin{enumerate}
        \item[(a)] $M$ is called $\tau$-\emph{rigid} if $\Hom_\Lambda(M,\tau M) = 0$; 
        \item[(b)] $M$ is called $\tau$-\emph{tilting} if $M$ is $\tau$-rigid and $|\Lambda|= |M|$; 
        \item[(c)] $M$ is called  \emph{support $\tau$-tilting} if there exists an idempotent $e$ of $\Lambda$ such that $M$ is a $\tau$-tilting $(\Lambda/\langle e \rangle)$-module. 
    \end{enumerate}
\end{definition}

It is often convenient to view support $\tau$-tilting $\Lambda$-modules as certain pairs of modules in $\modd \Lambda$. 

\begin{definition}[{\cite[Def. 0.3]{tau-tiling-theory}}]
    Let $(M,P)$ be a pair with $M\in\modd\Lambda$ and $P\in\proj\Lambda$. 
    \begin{enumerate}
        \item[(a)] We call $(M,P)$ a \emph{$\tau$-rigid} pair if $M$ is $\tau$-rigid and $\Hom_\Lambda(P,M) = 0$; 
        \item[(b)] We call $(M,P)$ a \emph{support $\tau$-tilting} pair if $(M,P)$ is $\tau$-rigid and $\abs{M}+\abs{P} = \abs{\Lambda}$. 
    \end{enumerate}
\end{definition}

The notions of support $\tau$-tilting $\Lambda$-module and support $\tau$-tilting pair are compatible in the following way. 

\begin{proposition}[{\cite[Prop. 2.3]{tau-tiling-theory}}]
        Let $(M,P)$ be a pair with $M\in\modd\Lambda$ and $P\in \proj\Lambda$. Let $e$ be and idempotent of $\Lambda$ such that $\add P = \add \Lambda e$.
        \begin{enumerate}
            \item[(i)]  $(M,P)$ is a support $\tau$-tilting pair for $\Lambda$ if and only if $M$ is a support $\tau$-tilting $(\Lambda/\langle e \rangle)$-module; 
            \item[(ii)] If $(M,P)$ and $(M,Q)$ are support $\tau$-tilting pairs for $\Lambda$, then $\add P = \add Q$. In other words, $M$ determines $P$ and $e$ uniquely.  
        \end{enumerate}
\end{proposition}
 
We recall the following useful characterizations of Auslander and Smal{\o}.

\begin{lemma}[{\cite[Proposition 5.8]{AlmostSplitSeqInSubcategories}}]\label{Smalo's_result}
Let $M,N\in \modd\Lambda$. The following holds: 
\begin{enumerate}
    \item[(i)] $\Hom_\Lambda(X,\tau Y) = 0$ if and only if $\Ext^1_\Lambda(Y, \Gen X) = 0$;
    \item[(ii)] $X$ is $\tau$-rigid if and only if $X$ is $\Ext$-projective in $\Gen X$. 
\end{enumerate}
\end{lemma} 

Let $R$ be a finite dimensional local commutative algebra and let $\Lambda= R\otimes kQ$. Notice that $R\cong kQ_R/I_R$ where $Q_R$ is the quiver with one vertex and $t$ loops $x_i$, for $1\leq i\leq t$, and $I_R$ is an admissible ideal of $kQ_R$. It is easy to see that $\Lambda\cong kQ'/I'$, where the quiver $Q'$ is obtained from $Q$ by adding loops $x_{i,j}$ at each vertex $j\in Q_0$ for $1\leq i\leq t$ (i.e. a copy of $Q_R$ at every vertex $j$). The ideal $I'$ is generated by the relations $\alpha\cdot x_{i,u} = x_{i,v}\cdot \alpha$ for every arrow $\alpha: u\to v$ in $Q$, together with $I_{R,j}$, where $I_{R,j}$ denotes the inclusion of the admissible ideal $I_R$ for each copy of $Q_R$ at a vertex $j$; see \cite{On_the_representation_type_of_tensor_product_algebras} for the general construction of the tensor product of path algebras modulo admissible ideal. 
Observe that for $t=1$, our class of algebras coincides with the class studied in \cite{Quivers_with_relations_for_symmetrizable_Cartan_matrices_I}.

It is well-known there is a functor $$ \Lambda\otimes_{kQ}- : \modd{kQ}\to \modd\Lambda$$ known as the \textit{induction functor}. Notice that, for $M\in \modd{kQ}$, we have $$\Lambda\otimes_{kQ}M = R\otimes (kQ\otimes_{kQ} M) \cong R\otimes M.$$ Moreover, the induction functor has a right adjoint given by the \textit{restriction of scalars} $\res(-)$. The $\Lambda$-modules in the image of the induction functor are called \textit{induced modules}. 

\begin{remark}\label{res(ind)}
    Let $d:=\dim R$. Observe that for a $kQ$-module $M$, composing the induction functor with the restriction of scalars gives $d$ copies of $M$, that is $\res(\Lambda\otimes_{kQ}M)= M^d$. 
\end{remark}

The third statement of the next result was mentioned (without proof) in the proof of  \cite[Corollary 4.4]{super_Caldero-Chapoton_map_for_type_A}. We include the proof here for completeness. 

\begin{proposition}\label{tauCommutesWithInduction}
    Let $\Lambda = R\otimes kQ$. The following statements hold: 
    \begin{enumerate}
        \item[(i)] The induction functor is exact;
        \item[(ii)] The Nakayama functor commutes with the induction functor, that is, for $M\in\modd{kQ}$, there is a natural isomorphism in $\modd\Lambda$ $$ \nu_\Lambda(\Lambda\otimes_{kQ} M)\cong \Lambda\otimes_{kQ}\nu_{kQ}(M);$$ 
        \item[(iii)] The AR-translation commutes with the induction functor, that is, for $M\in\modd{kQ}$, there is a natural isomorphism in $\modd\Lambda$ $$ \tau_\Lambda(\Lambda\otimes_{kQ} M)\cong \Lambda\otimes_{kQ}\tau_{kQ}(M).$$ 
    \end{enumerate}
\end{proposition}

\begin{proof}
    \begin{enumerate}
        \item[(i)] We have that $_{kQ}\Lambda \cong (kQ)^d$ in $\modd{kQ}$. Hence, $\Lambda$ is a projective $kQ$-module and therefore the induction functor $\Lambda\otimes_{kQ}-$ is exact. 
        \item[(ii)]  First, observe that 
        \begin{align*}
            D(\Lambda) = D(R\otimes kQ) &= \Hom_k(R\otimes kQ, k)\\
            &\cong \Hom_{k\otimes k}(R\otimes kQ, k\otimes k)\\
            &\cong \Hom_k(R,k)\otimes\Hom_k(kQ,k)\\
            &= D(R)\otimes D(kQ)\\
            &\cong R\otimes D(kQ).
        \end{align*}
        Moreover, for a finite dimensional algebra $A$, we have a functorial isomorphism $\nu_A(-)\cong D(A)\otimes_A-$. It follows that, for $M\in\modd{kQ}$ 
        \begin{align*}
            \nu_\Lambda(\Lambda\otimes_{kQ} M)&\cong D(\Lambda)\otimes_\Lambda (\Lambda\otimes_{kQ}M) \\
            &\cong D(\Lambda)\otimes_{kQ} M\\
            &\cong R\otimes D(kQ) \otimes_{kQ} M \\
            &\cong R\otimes \nu_{kQ}(M) \\
            & \cong (R\otimes kQ)\otimes_{kQ} \nu_{kQ}(M) \cong \Lambda\otimes_{kQ}\nu_{kQ}(M).
        \end{align*}
    \end{enumerate}
    \item[(iii)] Let $M\in \modd{kQ}$ and let $P_1\to P_0\to M\to 0$ a minimal projective presentation. By Lemma \ref{tensor_of_minimal_proj_presentation}(i), applying the induction functor $\Lambda\otimes_{kQ}-$ we get a minimal projective presentation of $\Lambda\otimes_{kQ}M$ in $\modd{\Lambda}$ of the form 
    $$\Lambda\otimes_{kQ}P_1\to \Lambda\otimes_{kQ}P_0\to\Lambda\otimes_{kQ}M\to 0.$$
    Thus,
    \begin{equation}\label{seq1}
        0\to\tau_\Lambda(\Lambda\otimes_{kQ}M)\to \nu_\Lambda(\Lambda\otimes_{kQ}P_1)\to \nu_\Lambda(\Lambda\otimes_{kQ}P_0)\to \nu_\Lambda(\Lambda\otimes_{kQ}M)\to 0
    \end{equation}
    is an exact sequence in $\modd\Lambda$. Now consider the exact sequence $$ 0\to \tau_{kQ}(M)\to \nu_{kQ}(P_1)\to \nu_{kQ}(P_0)\to \nu_{kQ}(M)\to 0$$
    in $\modd{kQ}$. Then, the induction functor yields an exact sequence 
    \begin{equation}\label{seq2}
        0\to \Lambda\otimes_{kQ}\tau_{kQ}(M)\to \Lambda\otimes_{kQ}\nu_{kQ}(P_1)\to \Lambda\otimes_{kQ}\nu_{kQ}(P_0)\to \Lambda\otimes_{kQ}\nu_{kQ}(M)\to 0
    \end{equation}
    in $\modd{\Lambda}$. Comparing the sequences \eqref{seq1} and \eqref{seq2}, we obtain a commutative diagram 

    \[\begin{tikzcd}[ampersand replacement=\&,cramped,column sep=small]
	0 \& {\tau_{\Lambda}(\Lambda\otimes_{kQ}M)} \& {\nu_\Lambda(\Lambda\otimes_{kQ}P_1)} \& {\nu_\Lambda(\Lambda\otimes_{kQ}P_0)} \& {\nu_\Lambda(\Lambda\otimes_{kQ}M)} \& 0 \\
	0 \& {\Lambda\otimes_{kQ}\tau_{kQ}(M)} \& {\Lambda\otimes_{kQ}\nu_{kQ}(P_1)} \& {\Lambda\otimes_{kQ}\nu_{kQ}(P_0)} \& {\Lambda\otimes_{kQ}\nu_{kQ}(M)} \& 0
	\arrow[from=1-1, to=1-2]
	\arrow[from=1-2, to=1-3]
	\arrow[dashed, from=1-2, to=2-2]
	\arrow[from=1-3, to=1-4]
	\arrow["\cong", from=1-3, to=2-3]
	\arrow[from=1-4, to=1-5]
	\arrow["\cong", from=1-4, to=2-4]
	\arrow[from=1-5, to=1-6]
	\arrow["\cong", from=1-5, to=2-5]
	\arrow[from=2-1, to=2-2]
	\arrow[from=2-2, to=2-3]
	\arrow[from=2-3, to=2-4]
	\arrow[from=2-4, to=2-5]
	\arrow[from=2-5, to=2-6]
    \end{tikzcd}\]
    
in $\modd{\Lambda}$ where the last three maps are isomorphism by part (ii) and therefore the induced map $\tau_\Lambda(\Lambda\otimes_{kQ}M)\to \Lambda\otimes_{kQ}\tau_{kQ}(M)$ is also an isomorphism. This finishes the proof.
\end{proof}

As a consequence, the induction functor preserves $\tau$-rigid modules. 

\begin{corollary}\label{inductionPreservesTauRigidModules}
     Let $M\in\modd{kQ}$ be $\tau$-rigid. Then, $\Lambda\otimes_{kQ}M$ is a $\tau$-rigid $\Lambda$-module.
\end{corollary}

\begin{proof}
    Let $M\in\modd{kQ}$ be $\tau$-rigid. By Proposition \ref{tauCommutesWithInduction} (iii), we have that $\tau_\Lambda(\Lambda\otimes_{kQ} M)\cong \Lambda\otimes_{kQ}\tau_{kQ}(M)$. Hence, 
    \begin{align*}
        \Hom_\Lambda(\Lambda\otimes_{kQ} M,\tau_\Lambda(\Lambda\otimes_{kQ} M)) &\cong \Hom_\Lambda(\Lambda\otimes_{kQ}M, \Lambda\otimes_{kQ}\tau_{kQ}(M))\\
        &\cong \Hom_{R\otimes kQ}(R\otimes M, R\otimes\tau_{kQ}(M))\\
        &\cong \Hom_{R}(R, R)\otimes \Hom_{kQ}(M, \tau_{kQ}(M)) = 0
    \end{align*}
    This concludes the proof. 
\end{proof}

Before getting to the main result of this section we need the following.  

\begin{definition}
    Let $\Lambda$ be a finite dimensional algebra and let $M\in\modd\Lambda$. Let $P_1\to P_0\to M\to 0$ be a minimal projective presentation with $P_0 = \bigoplus_{i=1}^n P(i)^{\alpha_i}$ and $P_1 = \bigoplus_{i=1}^n P(i)^{\beta_i}$, where each $P(i)$ is an indecomposable projective $\Lambda$-module. Then, the \emph{g-vector} of $M$ is defined as $$g^M:=(\alpha_1-\beta_1,\cdots, \alpha_n-\beta_n).$$
\end{definition}

\begin{theorem}[{\cite{tauTiltingFiniteAlgebrasAnd_g-vectors}}]\label{M=Niffg^M=g^N}
    Let $\Lambda$ be a finite dimensional algebra and let $M$ and $N$ be $\tau$-rigid $\Lambda$-modules. Suppose that $(g^M)_i\leq (g^N)_i$ for $1\leq i\leq n$. Then $M$ is a quotient of $N$. In particular, $g^M=g^N$ if and only if $M\cong N$. 
\end{theorem}

\begin{theorem}[{\cite[Thm. 11]{ReductionTheoremForTauRigidModules}}]\label{EJR_Thm11}
    Let $\Lambda$ be a finite dimensional algebra. Let $I$ be an ideal generated by central elements and contained in the Jacobson radical of $\Lambda$. Then, the $g$-vectors of the indecomposable $\tau$-rigid $\Lambda$-modules coincide with the $g$-vectors of the indecomposable $\tau$-rigid $\Lambda/I$-modules.
\end{theorem}

We are now ready to state the main result of this section. We use a similar argument to the one in \cite[Cor. 4.4]{super_Caldero-Chapoton_map_for_type_A}. We include the proof for completeness.

\begin{proposition}\label{tau-rigid-Lambda1:1tau-rigid-kQ}
    Let $\Lambda = R\otimes kQ$. Then, the induction functor $\Lambda\otimes_{kQ}-$ induces a bijection between the set of isoclasses of indecomposable $\tau$-rigid $kQ$-modules and the set of isoclasses of indecomposable $\tau$-rigid $\Lambda$-modules. 
\end{proposition}

\begin{proof}
    By Proposition \ref{HomExt_tensor_product} and Corollary \ref{inductionPreservesTauRigidModules}, we observe that the induction functor maps indecomposable $\tau$-rigid $kQ$-modules to indecomposable $\tau$-rigid $\Lambda$-modules. Hence, the above map is well-defined.
    
    We show that $\Lambda\otimes_{kQ}-$ is injective. Let $X,Y$ be indecomposable $\tau$-rigid $kQ$-modules such that $\Lambda\otimes_{kQ}X\cong\Lambda\otimes_{kQ}Y$. Applying restriction of scalars we get $X^d \cong Y^d$ (see Remark \ref{res(ind)}), and therefore the induction functor is injective on indecomposable $\tau$-rigid modules.

    Notice that $I:= \langle x_i\otimes 1\mid 1\leq i\leq t\rangle$ is an ideal generated by central elements and contained in the Jacobson radical of $\Lambda$. Moreover, $\Lambda/I \cong kQ$. Hence, the $g$-vectors of the indecomposable $\tau$-rigid $kQ$-modules coincide with the $g$-vectors of the indecomposable $\tau$-rigid $\Lambda$-modules by Theorem \ref{EJR_Thm11}. To prove surjectivity, let $N$ be an indecomposable $\tau$-rigid $\Lambda$-module. By Theorem \ref{EJR_Thm11}, there exists an indecomposable $\tau$-rigid $kQ$-module $M$ with $g^M = g^N$. Let 
    $$\bigoplus_{i=1}^n P(i)^{\beta_i}\to \bigoplus_{i=1}^n P(i)^{\alpha_i}\to M\to 0$$
    be a minimal projective presentation of $M$ in $\modd kQ$. Since every indecomposable projective $\Lambda$-module is of the form $\Lambda\otimes_{kQ}P(i)$, where $P(i)$ is the indecomposable projective $kQ$-module corresponding to the vertex $i$, the induction functor $\Lambda\otimes_{kQ}-$ applied to the above sequence yields a minimal projective presentation of $\Lambda\otimes_{kQ}M$
    $$\bigoplus_{i=1}^n (\Lambda\otimes_{kQ}P(i))^{\beta_i}\to \bigoplus_{i=1}^n (\Lambda\otimes_{kQ}P(i))^{\alpha_i}\to \Lambda\otimes_{kQ}M\to 0$$
    in $\modd{\Lambda}$ by Lemma \ref{tensor_of_minimal_proj_presentation}. It follows that $g^M = g^{\Lambda\otimes_{kQ}M}$. But then $g^N = g^{\Lambda\otimes_{kQ}M}$ and therefore $N\cong \Lambda\otimes_{kQ}M$ by Theorem \ref{M=Niffg^M=g^N}. We conclude that $\Lambda\otimes_{kQ}-$ is surjective on indecomposable $\tau$-rigid objects. This finishes the proof. 
\end{proof}

For a finite dimensional algebra $\Lambda$, we denote by $\text{f-tors }\Lambda$ and $\text{s$\tau$-tilt }\Lambda$ the set of functorially finite torsion classes in $\modd \Lambda$ and the set of support $\tau$-tilting $\Lambda$-modules, respectively. There is a bijection \cite[Thm. 2.7]{tau-tiling-theory} between $\text{s$\tau$-tilt }\Lambda$ and $\text{f-tors }\Lambda$ given by $\text{s$\tau$-tilt }\Lambda\ni T\mapsto \Gen T\in \text{f-tors }\Lambda$ and $\text{f-tors }\Lambda\ni\T\mapsto \P(\T)\in\text{s$\tau$-tilt }\Lambda$.  As a first consequence of Proposition \ref{tau-rigid-Lambda1:1tau-rigid-kQ}, we establish a bijection between $\text{s$\tau$-tilt }kQ$ and $\text{s$\tau$-tilt }\Lambda$, and between $\text{f-tors }kQ$ and $\text{f-tors }\Lambda$, where $\Lambda = R\otimes kQ$. 

\begin{corollary}\label{stautitl(kQ)1:1stautilt(Lambda)}
    Let $\Lambda = R\otimes kQ$. Then, the following statements hold.
    \begin{enumerate}
        \item[(i)] The induction functor induces a bijection between the set of $\tau$-rigid pairs in $\modd kQ$ and and the set of $\tau$-rigid pairs in $\modd \Lambda$; 
        \item[(ii)] There is a commutative diagram of bijections 
    \[\begin{tikzcd}[ampersand replacement=\&,cramped]
	{\stautilt kQ} \&\& {\ftors kQ} \\
	{\stautilt \Lambda} \&\& {\ftors \Lambda}
	\arrow["{\Gen(-)}", from=1-1, to=1-3]
	\arrow["{\Lambda\otimes_{kQ}-}"', from=1-1, to=2-1]
	\arrow["{F'}", from=1-3, to=2-3]
	\arrow["{\Gen(-)}"', from=2-1, to=2-3]
\end{tikzcd}\]
where $F'$ is given by $\Gen(M)\mapsto \Gen(\Lambda\otimes_{kQ}M)$, for $M\in \stautilt kQ$.

    \end{enumerate}
\end{corollary}

\begin{proof}
   By Proposition \ref{tau-rigid-Lambda1:1tau-rigid-kQ}, the induction functor induces a bijection between the set of isoclasses of indecomposable $\tau$-rigid $kQ$-modules and the set of isoclasses of indecomposable $\tau$-rigid $\Lambda$-modules. Moreover, every indecomposable projective $\Lambda$-module is of the form $\Lambda\otimes_{kQ}(kQe_i)$ for a primitive idempotent $e_i$ in $kQ$. Hence, using the fact that the induction functor preserves indecomposable modules and commutes with direct sums, and $\Hom_\Lambda(\Lambda\otimes_{kQ}P,\Lambda\otimes_{kQ}M) = 0$ if and only if $\Hom_{kQ}(P,M) = 0$, we obtain that the assignment $(M, P)\mapsto (\Lambda\otimes_{kQ}M, \Lambda\otimes_{kQ}P)$ defines a bijection between $\tau$-rigid pairs in $\modd kQ$ and $\tau$-rigid pairs in $\modd \Lambda$. This proves (i).

   By \cite[Thm. 2.7]{tau-tiling-theory} the horizontal maps in the above diagram are bijections.  Using the fact that the induction functor preserves the number of direct summands together with the first part of the proof, we infer that $\Lambda\otimes_{kQ}-$ induces a bijection between $\stautilt kQ$ and $\stautilt\Lambda$. Since the diagram above is clearly commutative, we conclude that $F'$ induces a bijection between $\ftors kQ$ and $\ftors\Lambda$. This proves (ii) and concludes the proof. 
\end{proof}

Recall that a finite dimensional algebra $\Lambda$ is called $\tau$-\textit{tilting finite} if there are only finitely many isomorphism classes of basic $\tau$-tilting $\Lambda$-modules; see {\cite[Def. 1.1]{tauTiltingFiniteAlgebrasAnd_g-vectors}}. As a consequence of Proposition \ref{tau-rigid-Lambda1:1tau-rigid-kQ}, we characterize when $\Lambda=R\otimes kQ$ is $\tau$-tilting finite. In the following result, the equivalence between (ii) and (iii) is well-known. We provide a proof for the sake of completeness. 

\begin{corollary}\label{Lambda_tau-tilt_finite_iff_Q_Dynkin}
    Let $\Lambda=R\otimes kQ$. Then, the following statements are equivalent. 
    \begin{enumerate}
        \item[(i)] $\Lambda$ is $\tau$-tilting finite;
        \item[(ii)] $kQ$ is  $\tau$-tilting finite;
        \item[(iii)] $Q$ is of Dyinkin type.  
    \end{enumerate}
\end{corollary}

\begin{proof}
    By {\cite[Cor. 2.9]{tauTiltingFiniteAlgebrasAnd_g-vectors}}, we know that $\Lambda$ is $\tau$-tilting finite if and only if there are only finitely many isomorphism classes of indecomposable $\tau$-rigid $\Lambda$-modules. This is the case if and only if there are only finitely many isomorphism classes of indecomposable $\tau$-rigid $kQ$-modules by the bijection established in Proposition \ref{tau-rigid-Lambda1:1tau-rigid-kQ}. Using {\cite[Theorem 4.1]{tau-tiling-theory}}, we have that the isoclasses of indecomposable $\tau$-rigid $kQ$-modules are in bijection with indecomposable rigid objects in the cluster category $\C_Q$. But indecomposable rigid objects in $\mathcal{C}_Q$ are in one-to-one correspondence with the cluster variables of the cluster algebra $\A_Q$,  where the latter set is finite if and only if $Q$ is of Dynkin type. This finishes the proof.
\end{proof}

\begin{remark}
    Notice that if $R$ is representation infinite, then so is $\Lambda$. In this way, we can construct examples of algebras that are representation infinite but $\tau$-tilting finite, provided that $Q$ is of Dynkin type (for more examples of $\tau$-tilting finite algebras see \cite{Mizuno_tau-tilting-finite, Plamondon_tau-tilting-finite, Adachi_tau-tilting-finite,a-characterisation-of-tau-tilting-finite-algebras}).
\end{remark}

\section{Induction functor, 2-term objects, and approximations}\label{induction-2-term-approx}

In this section, we discuss the relationship between induction functor, 2-term rigid objects, and approximations. 

Let $\Lambda$ be a finite dimensional algebra. We consider the bounded homotopy category of projectives $\K := K^b(\proj\Lambda)$ as a full subcategory of $D^b(\modd \Lambda)$. Recall that an object $\PP\in \K$ is called \textit{rigid} (or \textit{presilting}) if $\Hom_K(\PP,\PP[i])=0$ for all $i>0$, and \textit{silting} if it is rigid and $\text{thick}(\PP) = \K$, where $\text{thick}(\PP)$ is the smallest full subcategory of $\K$ containing $\PP$ and is closed under cones,$[\pm1]$, direct summands and isomorphisms. An object $\PP = (P_i, d_i)$ in $\K$ is called \textit{2-term} if $P_{i} = 0$ for all $i \neq -1,0$. A 2-term object in $\K$ which is rigid (respectively, silting) is called \textit{2-term rigid} (respectively, \textit{2-term silting}). 

We denote by $\twosilt \Lambda$ the set of 2-term silting objects in $\K$. For a $\Lambda$-module $U$, we denote by $\PP_U$ its minimal projective presentation, viewed as a 2-term object in $\K$. 

For the rest of this section, let $\Lambda = R\otimes kQ$. 

\begin{lemma}\label{2-term_and_iduction}
      The following statements hold. 
      \begin{enumerate}
          \item[(i)] Let $M\in \modd kQ$. Then $\Lambda\otimes_{kQ}\PP_M \cong \PP_{\Lambda\otimes_{kQ}M}$ in $K^b(\proj\Lambda)$; 
          \item[(ii)] The induction functor induces a bijection between 2-term silting objects in $K^b(\proj kQ)$ and 2-terms silting objects in $K^b(\proj\Lambda)$. 
      \end{enumerate}
\end{lemma}

\begin{proof}\label{ind_and_two-term}
    Let $M\in \modd kQ$. Then $P_{-1}\xrightarrow{f}P_0\xrightarrow{g}M\to 0$ is a minimal projective presentation of $M$ in $\modd kQ$ if and only if $\Lambda\otimes_{kQ}P_{-1}\xrightarrow{\Lambda\otimes_{kQ}f}\Lambda\otimes_{kQ}P_0\xrightarrow{\Lambda\otimes_{kQ}g}\Lambda\otimes_{kQ}M\to 0$ is a minimal projective presentation of $\Lambda\otimes_{kQ}M$ in $\modd \Lambda$. Hence, $\Lambda\otimes_{kQ}\PP_M \cong \PP_{\Lambda\otimes_{kQ}M}$, giving (i). 

    For a finite dimensional algebra $\Lambda$, the assignment \cite[Thm. 3.2]{tau-tiling-theory} $\stautilt \Lambda \ni (M, P) \mapsto (P_{-1}\oplus P\xrightarrow{\begin{smallmatrix}(f & 0)\end{smallmatrix}} P_0)\in \twosilt \Lambda$, where  $P_{-1}\xrightarrow{f}P_0\to M\to 0$ is a minimal projective presentation of $M$, defines a bijection between support $\tau$-tilting pairs in $\modd\Lambda$ and 2-term silting objects in $K^b(\proj\Lambda)$. Now, let $\Lambda = R\otimes kQ$ and consider the diagram 
    \[\begin{tikzcd}[ampersand replacement=\&,cramped]
	{\stautilt kQ} \&\& {\twosilt kQ} \\
	{\stautilt \Lambda} \&\& {\twosilt\Lambda}
	\arrow[from=1-1, to=1-3]
	\arrow["{\Lambda\otimes_{kQ}-}"', from=1-1, to=2-1]
	\arrow["{\Lambda\otimes_{kQ}-}", from=1-3, to=2-3]
	\arrow[from=2-1, to=2-3]
    \end{tikzcd}\]
    where the horizontal maps are given by the bijection described above. By part (i) the diagram is commutative and since the left vertical arrow is a bijection by Corollary \ref{stautitl(kQ)1:1stautilt(Lambda)}(i), it follows that $\Lambda\otimes_{kQ}-: \twosilt kQ\to \twosilt \Lambda$ is also bijective. This gives (ii) and concludes the proof. 
\end{proof}

Let $\C$ be an additive category. A morphism $\alpha: X\to Y$ is called \textit{right minimal} if every endomorphism $\varphi: X\to X$ satisfying $\alpha\circ\varphi = \alpha$ is an isomorphism. We recall the following well-known characterization of a right minimal map. 

\begin{lemma}\label{right-min}
    Let $\C$ be an additive, Hom-finite, and Krull-Schmidt category. A morphism $\alpha: X\to Y$ in $\C$ is right minimal if and only if no direct summands of $X$ are sent to 0. 
\end{lemma}

\begin{proof}
    The only if part is straightforward. So, suppose that no direct summands of $X$ are sent to 0. Since $\C$ is idempotent complete, every idempotent in $\End(X)$ splits. By the dual of \cite[Cor. 1.4]{KrauseSaorin} we conclude that $\alpha$ is right minimal.
\end{proof}

The next result shows that the induction functor preserves right minimal approximations.  

\begin{lemma}\label{induction_preserves_approximations}
    The following statements hold. 
    \begin{enumerate}
        \item[(i)] Let $\mathcal{X}$ be a full subcategory of $\modd kQ$ and let $f:X\to N$ be a right $\mathcal{X}$-approximation in $\modd kQ$. Then, $\Lambda\otimes_{kQ}f: \Lambda\otimes_{kQ}X\to \Lambda\otimes_{kQ}N$ is a right $(\Lambda\otimes_{kQ}\mathcal{X})$-approximation in $\modd\Lambda$. Moreover, if $f$ is right minimal, then so is $\Lambda\otimes_{kQ}f$; 

        \item[(ii)] Let $U,V\in \modd kQ$ and let $\alpha: \PP_{U'}\to \PP_V$ be a right $\add\PP_U$-approximation in $K^b(\proj kQ)$. Then,  $\Lambda\otimes_{kQ}\alpha: \Lambda\otimes_{kQ}\PP_{U'}\to \Lambda\otimes_{kQ}\PP_V$ is a right $(\Lambda\otimes_{kQ}\add\PP_U)$-approximation in $K^b(\proj\Lambda)$. Moreover, if $\alpha$ is right minimal, then so is $\Lambda\otimes_{kQ}\alpha$. 
    \end{enumerate}
\end{lemma}

\begin{proof}
    Let $\mathcal{X}$ be a full subcategory of $\modd kQ$ and let $f: X\to N$ be a right $\mathcal{X}$-approximation. To prove (i), let $h\otimes g: \Lambda\otimes_{kQ}X'\to \Lambda\otimes_{kQ} N$ be a morphism in $\modd \Lambda$ with $\Lambda\otimes_{kQ}X'\in (\Lambda\otimes_{kQ}\mathcal{X})$. We want to show there exists $h'\otimes f': \Lambda\otimes_{kQ}X'\to \Lambda\otimes_{kQ}X$ such that $(\id\otimes f)\circ(h'\otimes f') = h\otimes g$. Since $f$ is a right $\mathcal{X}$-approximation, there exists $g': X'\to X$ such that $fg' = g$. Hence, taking $f' = g'$ and $h' = h$ we get $(\id\otimes f)\circ(h\otimes g')= h\otimes fg' = h\otimes g$. This shows that $\Lambda\otimes_{kQ}f: \Lambda\otimes_{kQ}X\to \Lambda\otimes_{kQ}N$ is a right $(\Lambda\otimes_{kQ}\mathcal{X})$-approximation. 

    Assume $f$ is right minimal. It remains to show that $\Lambda\otimes_{kQ}f$ is right minimal. Suppose there exists $\overline{h}\otimes\overline{g}\in \End(\Lambda\otimes_{kQ}X)$ such that $(\id\otimes f)\circ(\overline{h}\otimes\overline{g}) = (\id\otimes f)$, that is $\overline{h}\otimes f\overline{g} = \id\otimes f$. Then, $\overline{h} = \id$ and $f\overline{g} = f$. By the minimality of $f$, it follows that $\overline{g}$ is an isomorphism, and therefore $\id\otimes \overline{g}$ is also an isomorphism. This concludes the proof of (i). 

    We now prove (ii). Let $\PP_{U'} = (P_{-1}^{U'}\xrightarrow{u'}P_0^{U'})$, $\PP_V = (P_{-1}^V\xrightarrow{v}P_0^V)$ and let $\alpha: \PP_{U'}\to \PP_V$ b a right $\add \PP_U$-approximation in $K^b(\proj kQ)$. This means that if $\gamma: \PP_{U''}\to \PP_V$, with $\PP_{U''} = (P_{-1}^{U''}\xrightarrow{u''}P_0^{U''})\in \add \PP_U$, then there exists $\beta: \PP_{U''}\to \PP_{U'}$ and a degree $-1$ morphism $h:\PP_{U''}\to \PP_V$ such that 
    \begin{equation}\label{eq6}
        \alpha\beta -\gamma = hd^{U''}+d^Vh,
    \end{equation}
    where $d^{U''}, d^V$ denote the differentials on $\PP_{U''}$ and $\PP_V$, respectively. Let $f\otimes\gamma: \Lambda\otimes_{kQ}\PP_{U''}\to \Lambda\otimes_{kQ}\PP_V$ be a morphism in $K^b(\proj\Lambda)$ with $\Lambda\otimes_{kQ}\PP_{U''}\in (\Lambda\otimes_{kQ}\add\PP_U)$ and $f\in \End(\Lambda)$. We want to show the existence of $f'\otimes\gamma': \Lambda\otimes_{kQ}\PP_{U''}\to \Lambda\otimes_{kQ}\PP_{U'}$ with $f'\in \End(\Lambda)$ and a degree $-1$ morphism $h_1\otimes h_2: \Lambda\otimes_{kQ}\PP_{U''}\to \Lambda\otimes_{kQ}\PP_V$ such that $$ (\id\otimes\alpha)\circ (f'\otimes \gamma') - f\otimes\gamma = (h_1\otimes h_2)\circ(\id\otimes d^{U''})+(\id\otimes d^V)\circ(h_1\otimes h_2).$$ Using Equation \eqref{eq6} we get 
    \begin{align*}
        f\otimes(\alpha\beta - \gamma) &= f\otimes(hd^{U''}+d^Vh)\\
        f\otimes\alpha\beta -f\otimes\gamma &= f\otimes hd^{U''} + f\otimes d^Vh \\
        (\id\otimes\alpha)\circ(f\otimes\beta) -f\otimes\gamma &= (f\otimes h)\circ(\id\otimes d^{U''}) + (\id\otimes d^V)\circ(f\otimes h). 
    \end{align*}
    Then, the claim follows by taking $f'\otimes \gamma' = f\otimes \beta$ and $h_1\otimes h_2 = f\otimes h$. This shows that $\Lambda\otimes_{kQ}\alpha: \Lambda\otimes_{kQ}\PP_{U'}\to \Lambda\otimes_{kQ}\PP_V$ is a right $(\Lambda\otimes_{kQ}\add\PP_U)$-approximation in $K^b(\proj\Lambda)$.

    Assume $\alpha$ is right minimal. We are left to prove that $\id\otimes\alpha$ is also right minimal. Suppose this is not the case. Then, by Lemma \ref{right-min}, there exists a direct summand $\Lambda\otimes_{kQ}\overline{\PP_{U'}}$ of $\Lambda\otimes_{kQ}\PP_{U'}$ that is sent to 0. In other words, $a\otimes \overline{\alpha} = 0$ in $K^b(\proj\Lambda)$, where $a\otimes \overline{\alpha}$ is the restriction of $\id\otimes\alpha$ to $\Lambda\otimes_{kQ}\overline{\PP_{U'}}$. Hence, there exists a degree $-1$ morphism $g\otimes h_i: \Lambda\otimes_{kQ}\PP_{U'}\to \Lambda\otimes_{kQ}\PP_{V}$ with $g\in\End(\Lambda)$ such that 
    \begin{align*}
        a\otimes\overline{\alpha} &= (g\otimes h_i)\circ (\id\otimes d^{U'}) + (\id\otimes d^{V})\circ (g\otimes h_i)\\
        &= g\otimes h_id^{U'} + g\otimes d^{V}h_i\\
        &= g\otimes(h_id^{U'}+ d^{V}h_i). 
    \end{align*}
    where $d^{U'}$ denotes the differential on $\PP_{U'}$. It follows that $a = g$ and $\overline{\alpha} = h_id^{U'}+d^{V}h_i$. But then $\overline{\alpha}: \overline{\PP_{U'}}\to \PP_V$ is a zero component of $\alpha$ in $K^b(\proj kQ)$ and therefore $\alpha$ is not right minimal by Lemma \ref{right-min}, a contradiction. This proves (ii) and finishes the proof.
\end{proof}

\section{Projective dimension in wide subcategories}\label{pdSection}

Let $\Lambda$ be a finite dimensional algebra and let $\W\subseteq \modd\Lambda$ be a functorially finite wide subcategory of $\modd\Lambda$. Then, $\W$ is equivalent to the module category of a finite dimensional algebra by \cite[Prop. 4.12]{Rigid_ICE-subcategories}. Let $M\in\modd\Lambda$ and assume that $M$ lies in a functorially finite wide subcategory $\W$. This section aims to prove that if $M$ has finite projective dimension as a $\Lambda$-module, then $M$ also has finite projective dimension when considered as an object in $\W$.\\

We start with the following observation. 

\begin{lemma}\label{SimpleProjInjArePreservedInW}
    Let $M$ be an indecomposable $\Lambda$-module. Assume $M$ lies in a functorially finite wide subcategory $\mathcal{W}\subseteq\modd{\Lambda}$. The following statements hold.
    \begin{enumerate}
        \item[(i)] If $M$ is a simple $\Lambda$-module, then $M$ is simple in $\mathcal{W}$;
        \item[(ii)] If $M$ is a projective $\Lambda$-module, then $M$ is projective in $\mathcal{W}$;
        \item[(iii)] If $M$ is a injective $\Lambda$-module, then $M$ is injective in $\mathcal{W}$.
    \end{enumerate}
\end{lemma}

\begin{proof}
\begin{enumerate}
        \item[(i)] Let $S$ be a simple module in $\modd\Lambda$ and assume $S\in \mathcal{W}$. By assumption, we have that $\mathcal{W}$ is an exact full subcategory of $\modd\Lambda$, that is, the embedding $F:\mathcal{W}\to \modd\Lambda$ preserves short exact sequences. Assume $S$ is not a simple module in $\mathcal{W}$. Then, there exists a proper submodule $S'\subset S$ in $\mathcal{W}$ with $S'\neq 0$. Taking the cokernel of the inclusion of $S'$ into $S$, we get a short exact sequence $\eta: 0\to S'\xrightarrow{i} S \to \coker(i) \to 0$ in $\mathcal{W}$. Since $F$ is the identity on objects and morphisms, we obtain that $F(\eta)=\eta$ in $\modd\Lambda$. But this contradicts the fact that $S$ is a simple module in $\modd\Lambda$. Hence, $S$ is a simple module in $\mathcal{W}$.\\
        \item[(ii)] Similarly, let $P\in\proj\Lambda$ be indecomposable and assume $P\in \mathcal{W}$. Suppose $P$ is not a projective module in $\mathcal{W}$. Then, there exists a non-split short exact sequence $\eta': 0\to \ker(p) \to P' \xrightarrow{p} P \to 0$ in $\mathcal{W}$ for some projective module $P'\in\W$. After applying $F$ to $\eta'$, we get that $F(\eta')=\eta'$ is a non-split short exact sequence in $\modd\Lambda$ which contradicts the fact that $P\in\proj\Lambda$. Hence, $P$ is a projective module in $\mathcal{W}$.\\
        \item[(iii)] This is the dual statement of (ii).      
\end{enumerate}
\end{proof}

\begin{remark}
    The converse of the proposition above is not true in general. For example, consider $\Lambda = \begin{tikzcd}k( 1 \arrow[r, "a"] & 2 \arrow[r, "b"] & 3\end{tikzcd}) $. Then, $\W = \left\{ 3, \begin{smallmatrix} 1\\2\\3\end{smallmatrix}, \begin{smallmatrix} 1\\2 \end{smallmatrix} \right\}$ is a functorially finite wide subcategory of $\modd{\Lambda}$. Note that $\begin{smallmatrix} 1\\2 \end{smallmatrix}$ is simple in $\W$ but $\begin{smallmatrix} 1\\2 \end{smallmatrix}$ is not simple in $\modd\Lambda$.
\end{remark}

Let $M, N \in \modd\Lambda$. Following \cite[Chapter IV.9]{A_Course_In_Homological_Algebra}, recall that an
 $n$-\textit{extension} of $M$ by $N$ is an exact sequence of $\Lambda$-modules of the form $$ \E: 0\to N \to E_n\to \cdots \to E_1 \to M\to 0.$$ We write $\E \rightsquigarrow \E'$ if there exists a commutative diagram 

\[\begin{tikzcd}[ampersand replacement=\&,cramped,sep=scriptsize]
	{\E:} \& 0 \& N \& {E_n} \& \cdots \& {E_1} \& M \& 0 \\
	{\E':} \& 0 \& N \& {E_n'} \& \cdots \& {E_1'} \& M \& 0
	\arrow[from=1-2, to=1-3]
	\arrow[from=1-3, to=1-4]
	\arrow[from=1-4, to=1-5]
	\arrow[from=1-5, to=1-6]
	\arrow[from=1-6, to=1-7]
	\arrow[from=1-7, to=1-8]
	\arrow[from=2-2, to=2-3]
	\arrow[from=2-3, to=2-4]
	\arrow[from=2-4, to=2-5]
	\arrow[from=2-5, to=2-6]
	\arrow[from=2-6, to=2-7]
	\arrow[from=2-7, to=2-8]
	\arrow[Rightarrow, no head, from=1-3, to=2-3]
	\arrow[Rightarrow, no head, from=1-7, to=2-7]
	\arrow[from=1-6, to=2-6]
	\arrow[from=1-4, to=2-4]
\end{tikzcd}\]

We can define an equivalence relation given by $\E\sim \E'$ if and only if there exists a chain $\E =\E_0, \E_1,\cdots, \E_k = \E'$ $$\E_0 \rightsquigarrow \E_1 \leftsquigarrow \E_2 \rightsquigarrow \dots \leftsquigarrow \E_k.$$ We denote by $[\E]$ the equivalence class of the $n$-extension $\E$ and we define $$\Yext_\Lambda^n(M,N) := \left\{ [\E] \mid \E \text{ is an } n\text{-extension of } M \text{ by } N\right\}.$$ 
The following result characterizes $n$-extension groups in terms of $n$-extensions. 

\begin{theorem}[{\cite[Chapter IV.9, Theorem 9.1]{A_Course_In_Homological_Algebra}}]\label{Yext^n=Ext^n}
    Let $M,N\in\modd\Lambda$. For $n\geq 1$, there is a functorial isomorphism
    \begin{equation}
        \Yext_\Lambda^n(M,N)\cong \Ext_\Lambda^n(M,N).
    \end{equation}
\end{theorem}

\begin{definition}
    A $n$-extension of $M$ by $N$ $$ \E: 0\to N \xrightarrow{\iota} E_n\to \cdots \to E_1 \xrightarrow{\epsilon} M\to 0$$ is called \emph{split} if $\iota: N\to E_n$ is a split monomorphism.
\end{definition}

\begin{lemma}
    Let $$ \E: 0\to N \xrightarrow{\iota} E_n\to \cdots \to E_1 \xrightarrow{\epsilon} M\to 0$$ be a $n$-extension of $M$ by $N$. Then the following statements hold:
    \begin{enumerate}
        \item[(i)] $\iota$ is a split monomorphism if and only if $\epsilon$ is a split epimorphism;
        \item[(ii)]  if $\E$ is split, then so is any other representative in $[\E]$. 
    \end{enumerate}
\end{lemma}

\begin{proof}
    For statement (i) see \cite[Lemma 3.6, Cor. 3.7]{d-AR_sequences_in_subcategories}. We prove (ii). Assume $\E$ is a split $n$-extension where $\pi: E_n\to N$ is a section for $\iota$, and let $$ \E': 0\to N \xrightarrow{\iota'} E_n'\to \cdots \to E_1' \xrightarrow{\epsilon'} M\to 0$$ be a non-split $n$-extension, i.e. $\iota'$ (respectively, $\epsilon'$) is not a split monomorphism (respectively, not a split epimorphism). We show that $\E$ cannot be equivalent to $\E'$, that is there is no chain of the form $$\E \rightsquigarrow \E_1 \leftsquigarrow \E_2 \rightsquigarrow \dots \leftsquigarrow \E'$$
    To prove this, it suffices to show that neither $\E\rightsquigarrow\E'$ nor $\E'\rightsquigarrow \E$ can occur. Suppose we have a commutative diagram

\[\begin{tikzcd}[ampersand replacement=\&,cramped, sep=scriptsize]
	{\E':} \& 0 \& N \& {E_n'} \& \cdots \& {E_1'} \& M \& 0 \\
	{\E:} \& 0 \& N \& {E_n} \& \cdots \& {E_1} \& M \& 0
	\arrow["f", from=1-1, to=2-1]
	\arrow[from=1-2, to=1-3]
	\arrow["{\iota'}", from=1-3, to=1-4]
	\arrow[Rightarrow, no head, from=1-3, to=2-3]
	\arrow[from=1-4, to=1-5]
	\arrow["{f_n}", from=1-4, to=2-4]
	\arrow[from=1-5, to=1-6]
	\arrow["{\epsilon'}", from=1-6, to=1-7]
	\arrow["{f_1}", from=1-6, to=2-6]
	\arrow[from=1-7, to=1-8]
	\arrow[Rightarrow, no head, from=1-7, to=2-7]
	\arrow[from=2-2, to=2-3]
	\arrow["\iota", from=2-3, to=2-4]
	\arrow["\pi", curve={height=-12pt}, dashed, from=2-4, to=2-3]
	\arrow[from=2-4, to=2-5]
	\arrow[from=2-5, to=2-6]
	\arrow["\epsilon", from=2-6, to=2-7]
	\arrow[from=2-7, to=2-8]
\end{tikzcd}\]

Then $\pi f_n \iota ' = \pi \iota = \id_N$. In other words, $\iota'$ is a split monomorphism, a contradiction.  Dually, $\E\rightsquigarrow \E'$ contradicts the fact that $\epsilon': E_1'\to M$ is a non-split epimorphism.  This finishes the proof. 
\end{proof}

The next result can be deduced from Theorem \ref{Yext^n=Ext^n} and \cite[Lemma 1.6]{d-abelian_quotients_of_(d+2)-angulated_categories}, see also \cite[Exercise 9.4]{A_Course_In_Homological_Algebra}. It characterizes trivial $n$-extension groups in terms of split $n$-extensions.  

\begin{proposition}\label{Ext^n=0_and_n-extensions}
    Let $M, N\in \modd\Lambda$. The following statements are equivalent: 
    \begin{enumerate}
        \item[(i)] $\Ext_\Lambda^n(M,N) = 0$;
        \item[(ii)]  Every $n$-extension of $M$ by $N$ $$ \E: 0\to N \xrightarrow{\iota} E_n\to \cdots \to E_1 \xrightarrow{\epsilon} M\to 0$$ is split.
    \end{enumerate}
\end{proposition}

We are ready to state and prove the main result of this section. We denote by $\pd_\W$ the projective dimension in a functorially finite wide subcategory $\mathcal{W}\subseteq \modd\Lambda$. 

\begin{proposition}\label{pdIsPreservedInW}
        Let $m\geq 0$ and let $M\in\modd\Lambda$ with $\pd M = m$. Suppose $M$ lies in a functorially finite wide subcategory $\mathcal{W}$. Then,  $\pd_\mathcal{W} M \leq m$. 
\end{proposition}

\begin{proof}
    Let $M\in\modd\Lambda$ with $\pd M = m$. Suppose that $M$ lies in a functorially finite wide subcategory $\mathcal{W}\subseteq\modd\Lambda$. If $M\in\proj\Lambda$, then $\pd_\mathcal{W}M = 0$ by Proposition \ref{SimpleProjInjArePreservedInW} (ii). Thus, we can assume $\pd M = m >0$. Then, $\Ext_\Lambda^{m+i}(M,N) = 0$ for every $N\in \modd\Lambda$ and $i\geq 1$. If $M$ is a projective module in $\mathcal{W}$, then $\pd_\mathcal{W} M = 0$. 
    Now consider the case in which $M$ is not a projective module in $\mathcal{W}$, and suppose $\pd_\mathcal{W}M > m$. Then, there exists $N_\mathcal{W}\in \mathcal{W}$ such that $\Ext_\mathcal{W}^{m+1}(M,N_\mathcal{W}) \neq 0$. By Proposition \ref{Ext^n=0_and_n-extensions}, we obtain a non-split $(m+1)$-extension of $M$ by $N_\W$ in $\mathcal{W}$, that is an exact sequence in $\W$ of the form 

     \[\begin{tikzcd}[ampersand replacement=\&,cramped,sep=small]
	{} \& {\E: 0} \& {N_\mathcal{W}} \& {E_{m+1}} \&\& {E_m} \& \cdots \& {E_1} \& M \& 0 \\
	\&\&\&\& {R_{m+1}}
	\arrow["i", from=1-3, to=1-4]
	\arrow["{u_{m+1}}"', from=1-4, to=2-5]
	\arrow[from=2-5, to=1-6]
	\arrow[from=1-4, to=1-6]
	\arrow[from=1-6, to=1-7]
	\arrow[from=1-7, to=1-8]
	\arrow[from=1-8, to=1-9]
	\arrow[from=1-9, to=1-10]
	\arrow[from=1-2, to=1-3]
\end{tikzcd}\]

in which $i$ is a non-split monomorphism. This gives a non-split short exact sequence $$\eta : 0\to N_\mathcal{W}\xrightarrow{i} E_{m+1}\xrightarrow{u_{m+1}} R_{m+1}\to 0 $$ in $\mathcal{W}$. Since $F: \mathcal{W}\to \modd{\Lambda}$ is an exact embedding, we have that 
$F(\eta) = \eta$ is a non-split short exact sequence in $\modd\Lambda$. Hence, we get an exact sequence in $\modd\Lambda$ of the form $$ F(\E): 0\to N_\mathcal{W}\xrightarrow{F(i)} E_{m+1}\to \cdots \to E_1 \to M\to 0$$ where the monomorphism $F(i): N_\mathcal{W}\to E_{m+1}$ is non-split. But this implies that $\Ext_\Lambda^{m+1}(M,N_\mathcal{W}) \neq 0$, a contradiction. The claim follows. 
\end{proof}

As a first consequence of Proposition \ref{pdIsPreservedInW} we have the following. 

\begin{corollary}
    Let $\Lambda$ be a finite dimensional algebra with finite global dimension. Then every functorially finite wide subcategory $\W$ has finite global dimension.  
\end{corollary}

\begin{corollary}\label{chainOfW&pd}
 Let $\mathcal{W}_1\subset\cdots\subset\mathcal{W}_l\subset \mathcal{W}_{l+1} = \modd\Lambda$ be a chain of wide subcategories with $\W_i$ functorially finite in $\W_{i+1}$ for $1\leq i \leq l+1$. Let $M\in\modd\Lambda$ with $\pd M = m$ and assume $M$ lies in $\mathcal{W}_1$.  Then, $\pd_{\W_1}M\leq m$.
\end{corollary}

\begin{proof}
    The result follows applying Proposition \ref{pdIsPreservedInW} inductively. 
\end{proof}

\begin{corollary}\label{cor2.15}
    Let $M,N \in \modd\Lambda$ with $\Ext_\Lambda^n(M,N) = 0$. Let $\mathcal{W}$ be a functorially finite wide subcategory of $\modd\Lambda$ and assume $M, N\in \mathcal{W}$. Then, $\Ext^n_\mathcal{W}(M,N) = 0$. 
\end{corollary}

\begin{example}
    Let$\begin{tikzcd} \Lambda= k( 1 \arrow[r, "a"] & 2 \arrow[r, "b"] & 3\end{tikzcd})/\rad^2$. The AR quiver of $\Lambda$ can be depicted as follows 
    \[\begin{tikzcd}[ampersand replacement=\&,cramped,sep=scriptsize]
	\& {\begin{smallmatrix}2\\3\end{smallmatrix}} \&\& {\begin{smallmatrix}1\\2\end{smallmatrix}} \\
	{\begin{smallmatrix}3\end{smallmatrix}} \&\& {\begin{smallmatrix}2\end{smallmatrix}} \&\& {\begin{smallmatrix}1\end{smallmatrix}}.
	\arrow[from=2-1, to=1-2]
	\arrow[from=1-2, to=2-3]
	\arrow[from=2-3, to=1-4]
	\arrow[from=1-4, to=2-5]
	\arrow[dashed, from=2-5, to=2-3]
	\arrow[dashed, from=2-3, to=2-1]
\end{tikzcd}\]
Consider the functorially finite wide subcategory given by $$\mathcal{W} = \add(1\oplus 3) \simeq k(\bullet\quad \bullet).$$ Then, $\Ext_\mathcal{W}^2(1,3)=0$. On the other hand, there is a non-split 2-extension of 1 by 3 given by 
\[\begin{tikzcd}[ampersand replacement=\&,cramped,sep=scriptsize]
	0 \& {\begin{smallmatrix}3\end{smallmatrix}} \& {\begin{smallmatrix}2\\3\end{smallmatrix}} \&\& {\begin{smallmatrix}1\\2\end{smallmatrix}} \& {\begin{smallmatrix}1\end{smallmatrix}} \& 0. \\
	\&\&\& {\begin{smallmatrix}2\end{smallmatrix}}
	\arrow[from=1-3, to=2-4]
	\arrow[from=2-4, to=1-5]
	\arrow[from=1-2, to=1-3]
	\arrow[from=1-3, to=1-5]
	\arrow[from=1-5, to=1-6]
	\arrow[from=1-1, to=1-2]
	\arrow[from=1-6, to=1-7]
\end{tikzcd}\]
In other words, $\Ext_\Lambda^2(1,3)\neq 0$. This shows that the converse of Corollary \ref{cor2.15} does not hold in general. 

Finally (as pointed out by Eric Hanson), let $\begin{tikzcd} \W = \left\{3, \begin{smallmatrix}1\\2\end{smallmatrix}, 1\right\}\simeq  k( 1 \arrow[r] & 2 \end{tikzcd})$. Note that $S(1)$ lies in $\W$. Then, $\pd S(1) = 2$ but $\pd_\W S(1) = 1$. This shows that the inequality in Proposition \ref{pdIsPreservedInW} can be strict. 
\end{example}

\section{$\tau$-perpendicular subcategories}\label{section2}

Let $\Lambda$ be a finite dimenstional algebra. An important class of functorially finite wide subcategories of $\modd\Lambda$ is the $\tau$-\textit{perpendicular categories} introduced by Jasso in \cite{Jasso_Reduction} as a generalization of the Geigle-Lenzing perpendicular categories. This section aims to describe the $\tau$-perpendicular subcategories of $\modd\Lambda$, where $\Lambda = R\otimes kQ$, in terms of the $\tau$-perpendicular subcategories of $\modd kQ$.  

We start recalling the following definition. 

\begin{definition}[{\cite[Definition 3.3]{Jasso_Reduction}}]\label{tau-perp_Jasso}
    Let $\Lambda$ be a finite dimensional algebra and let $M\in\modd\Lambda$ be $\tau$-rigid. The \emph{$\tau$-perpendicular subcategory} associated to $M$ is the full subcategory of $\modd\Lambda$ given by $$J(M) = {^{\perp}(\tau M)} \cap M^\perp.$$
\end{definition}

Recall that $\mathcal{P}(^\perp\tau M)$ is the full subcategory of $^\perp\tau M$ consisting of $\Ext$-projective modules in $^\perp\tau M$, i.e. the modules $X$ in $^\perp\tau M$ such that $\Ext^1_\Lambda(X,{^\perp\tau}M) = 0$. The \textit{Bongartz completion} of $M$, denoted by $B_M$, is given by the direct sum of all indecomposable $\Ext$-projective modules in $^\perp\tau M$. The following summarizes some important facts about $J(M)$. 

\begin{proposition}\label{propertiesOfJ}
    Let $M\in\modd\Lambda$. Then we have: 
    \begin{enumerate} 
        \item[(i)] \cite[Thm. 3.8]{Jasso_Reduction} If $M$ is a basic $\tau$-rigid $\Lambda$-module, then $J(M)$ is equivalent to $\modd\Gamma_M$, where $\Gamma_M = \End_\Lambda(B_M)^\mathrm{op}/I$ and $I$ is the ideal generated by all maps factoring through $M$.
        \item[(ii)] \cite[Cor. 3.22]{Wall_Chamber_Structure} $J(M)$ is an exact abelian (wide) subcategory of $\modd\Lambda$;
        \item[(iii)] \cite[Thm. 3.8]{Jasso_Reduction} If $M$ is an indecomposable $\tau$-rigid $\Lambda$-module, then $J(M)$ has $\abs\Lambda -1$ simple modules up to isomorphism. 
    \end{enumerate}
\end{proposition}

\begin{remark}\label{Lambda_hereditary_J_hereditary}
    $\tau$-perpendicular subcategories are functorially finite wide; see for example \cite[Remark 4.10]{tau-perpendicular_wide_subategories}. Hence, if $\Lambda$ is hereditary, then so is any $\tau$-perpendicular subcategory of $\modd\Lambda$ by Proposition \ref{pdIsPreservedInW}. In this way, we recover \cite[Cor. 3.19(a)]{Jasso_Reduction} (see also \cite[Prop. 1.1]{Geigle_Lenzing_Perp_Subcat}).
\end{remark}

The next observation is useful.

\begin{lemma}[{\cite[Lemma 1.4]{Mutating-signed-tau-exceptional-sequences}}]\label{tauRigidInModLamTauRigidInJ}
    Let $\mathcal{W}$ be a $\tau$-perpendicular subcategory and let $X\in\modd\Lambda$ be indecomposable $\tau$-rigid. Suppose $X$ lies in $\mathcal{W}$. Then $X$ is $\tau$-rigid in $\mathcal{W}$. 
\end{lemma}

Recall that for $M, N\in\modd \Lambda$ there is a functorial isomorphism (AR duality)
\begin{equation*}
    \Ext_\Lambda^1(M,N) \cong D\overline{\Hom}_\Lambda(N,\tau M).
\end{equation*}

 In particular, if $\pd M \leq 1$ we have that 
 \begin{equation}\label{AR_Duality_pdM<=1}
      \Ext_\Lambda^1(M,N) \cong D\Hom_\Lambda(N,\tau M).
 \end{equation}
 
 We need the following useful lemma. 
\begin{lemma}\label{tau-rigid_in_J(X)_iff_taurigid_in_J(Lambda_otimes_X)}
    Let $X, Y, Z\in \modd kQ$, and suppose that $X$ is $\tau$-rigid. The following statements hold. 
    \begin{enumerate}
    \item[(i)]  $Y$ lies in $J(X)$ if and only if $\Lambda\otimes_{kQ}Y$ lies in $J(\Lambda\otimes_{kQ}X)$;
    \item[(ii)] Assume $Y,Z$ lie in $J(X)$. Then, $\Hom_{J(X)}(Z,\tau_{J(X)}Y) = 0$ if and only if $$\Hom_{J(\Lambda\otimes_{kQ} X)}(\Lambda\otimes_{kQ}Z,\tau_{J(\Lambda\otimes_{kQ}X)}(\Lambda\otimes_{kQ}Y)) = 0;$$
    \item[(iii)] Assume $Y$ lies in $J(X)$. Then, $Y$ is $\tau_{J(X)}$-rigid if and only if $\Lambda\otimes_{kQ}Y$ is $\tau_{J(\Lambda\otimes_{kQ}X)}$-rigid;
    \item[(iv)] If $\Lambda\otimes_{kQ}X$ is $\Ext$-projective in ${^\perp}\tau(\Lambda\otimes_{kQ}Y)$, then $X$ is $\Ext$-projective in ${^\perp}\tau Y$.
    \end{enumerate}

\end{lemma}

\begin{proof}
    We have that 
    \begin{align*}
        \Hom_\Lambda(\Lambda\otimes_{kQ}M, \Lambda\otimes_{kQ}N)&\cong \Hom_\Lambda(R\otimes M, R\otimes N))\\
        &\cong\Hom_{R}(R,R)\otimes\Hom_{kQ}(M,N) = 0 
    \end{align*}
    if and only if $N\in M^\perp$, and 
    \begin{align*}
        \Hom_\Lambda(\Lambda\otimes_{kQ}N, \tau(\Lambda\otimes_{kQ}M))&\cong \Hom_\Lambda(\Lambda\otimes_{kQ}N, \Lambda\otimes_{kQ}\tau M) \quad \text{ by Proposition \ref{tauCommutesWithInduction}}\\
        &\cong \Hom_\Lambda(R\otimes N, R\otimes  \tau M)\\
        &\cong \Hom_{R}(R,R)\otimes\Hom_{kQ}(N,\tau M) = 0
    \end{align*}
    if and only if $N\in {^\perp \tau M}$. This proves (i).
    
    Now assume that $Y,Z \in J(X)$. Using the fact that $Y\in \modd kQ$ together with Proposition \ref{HomExt_tensor_product} \eqref{pd_tensor_product}, we have that 
    \begin{align*}
        \pd_\Lambda(\Lambda\otimes_{kQ} Y) &= \pd_{R\otimes kQ}(R\otimes kQ\otimes_{kQ}Y)\\
                                           &= \pd_{R\otimes kQ}(R\otimes Y)\\
                                           &= \pd_R(R)+ \pd_{kQ}(Y) \leq 1.
    \end{align*}
    In particular, $\pd_{J(\Lambda\otimes_{kQ}X)}(\Lambda\otimes_{kQ}Y)\leq 1$ by Proposition \ref{pdIsPreservedInW}. Thus, we obtain
    \begin{align*}
    0&=\Hom_{J(\Lambda\otimes_{kQ}X)}(\Lambda\otimes_{kQ}Z, \tau_{J(\Lambda\otimes_{kQ}M_n)}(\Lambda\otimes_{kQ}Y))\\ 
    &\cong \Ext^1_{J(\Lambda\otimes_{kQ}X)}(\Lambda\otimes_{kQ}Y, \Lambda\otimes_{kQ}Z) && \text{by the AR-duality \eqref{AR_Duality_pdM<=1}}\\
    &\cong \Ext^1_\Lambda(\Lambda\otimes_{kQ}Y, \Lambda\otimes_{kQ}Z) && \text{$J(\Lambda\otimes_{kQ}X)$ is wide}\\
    &\cong \Hom_{R}(R,R)\otimes\Ext^1_{kQ}(Y, Z) && \text{by Proposition \ref{HomExt_tensor_product}}\\
    &\cong \Hom_{R}(R,R)\otimes\Ext^1_{J(X)}(Y,Z) && Y,Z\in J(X) \text{ and $J(X)$ is wide}\\
    &\cong \Hom_{R}(R,R)\otimes\Hom_{J(X)}(Y, \tau_{J(X)}(Z)) &&\text{by the AR-duality \eqref{AR_Duality_pdM<=1}}
    \end{align*}
    if and only if $\Hom_{J(X)}(Y, \tau_{J(X)}(Z)) = 0$. This shows (ii). 
    Now (iii) follows from (ii) taking $Y=Z$. 
    
    By the proof of part (i), if $\Lambda\otimes_{kQ}X \in {^\perp}\tau(\Lambda\otimes_{kQ}Y)$, then $X \in {^\perp}\tau Y$. Suppose $\Lambda\otimes_{kQ}X$ is $\Ext$-projective in ${^\perp}\tau(\Lambda\otimes_{kQ}Y)$.  By part (i), we have $\Lambda\otimes_{kQ}({^\perp}\tau Y)\subseteq {^{\perp}\tau(\Lambda\otimes_{kQ} Y)}$  and therefore
    $$ 0 = \Ext^1_\Lambda(\Lambda\otimes_{kQ}X, \Lambda\otimes_{kQ}({^\perp}\tau Y))\cong \Hom_{R}(R,R)\otimes\Ext^1_{kQ}(X, {{^\perp}\tau}Y).$$ This implies that $X$ is $\Ext$-projective in ${^\perp}\tau Y$. This shows (iv) and concludes the proof. 
\end{proof}

Recall that, given a finite dimensional algebra $\Lambda$, a pair of subcategories $(\mathcal{T}, \mathcal{F})$ is a \textit{torsion pair} in $\modd\Lambda$ if $\mathcal{T}={^\perp}\mathcal{F}$ and $\mathcal{F}=\mathcal{T}^\perp$. Given a torsion pair $(\mathcal{T}, \mathcal{F})$ and an arbitrary $\Lambda$-module $X$, there exists a unique short exact sequence (up to isomorphism) $$ 0\to tX\to X\to fX\to 0$$ where $tX\in \mathcal{T}$ and $fX\in \mathcal{F}$. This is called the \textit{canonical sequence} for $X$. For a $\tau$-rigid module $U$, we denote by $t_U$  (resp. $f_U$) the torsion (resp. torsion-free) functor associated with the torsion pair $(\Gen U, U^{\perp})$. \\

Let $M$ be a $\tau$-rigid $kQ$-module. The next key proposition gives a relationship between the torsion-free functors $f_M$ and $f_{\Lambda\otimes_{kQ}M}$, and the induction functor. As a consequence, the induction functor induces a bijection between indecomposable $\tau$-rigid modules in $J(M)$ and indecomposable $\tau$-rigid modules in $J(\Lambda\otimes_{kQ}M)$. 

\begin{proposition}\label{tau_rigid_in_J(M)_1:1_tau_rigid_in_J(Lambda_otimes_kQM)}
    Let $M$ be a $\tau$-rigid $kQ$-module. Then, for every $N\in \ind\modd kQ$ there is a functorial isomorphism 
    $$f_{\Lambda\otimes_{kQ}M}(\Lambda\otimes_{kQ}N)\cong \Lambda\otimes_{kQ}f_M N$$ in $\modd{\Lambda}$. In particular, we have a commutative diagram of bijections
\[
\begin{tikzcd}[ampersand replacement=\&, cramped, row sep=large, column sep=scriptsize, 
every matrix/.append style={font=\small}]
{\left\{ N \in\ind\modd kQ \;\middle|\; \begin{tabular}{@{}l@{}} $N\oplus M$ is $\tau_{kQ}$-rigid, \\ $N\notin \mathrm{Gen} M$ \end{tabular} \right\}} \&\& {\left\{X\in \mathrm{ind}J(M)\;\middle|\; \begin{tabular}{@{}l@{}} $X$ is $\tau_{J(M)}$-rigid \end{tabular} \right\}} \\
{\left\{ Y \in \ind\modd\Lambda \;\middle|\; \begin{tabular}{@{}l@{}} $Y\oplus(\Lambda\otimes_{kQ}M)$ is $\tau_{\Lambda}$-rigid, \\ $Y\notin \mathrm{Gen}(\Lambda\otimes_{kQ}M)$ \end{tabular} \right\}} \&\& {\left\{Z\in \mathrm{ind}J(\Lambda\otimes_{kQ}M)\;\middle|\; \begin{tabular}{@{}l@{}} $Z$ is $\tau_{J(\Lambda\otimes_{kQ}M)}$-rigid \end{tabular} \right\}}.
\arrow["{f_M}", from=1-1, to=1-3]
\arrow["{\Lambda\otimes_{kQ}-}"', from=1-1, to=2-1]
\arrow["{\Lambda\otimes_{kQ}-}", from=1-3, to=2-3]
\arrow["{f_{\Lambda\otimes_{kQ}M}}", from=2-1, to=2-3]
\end{tikzcd}
\]
\end{proposition}

\begin{proof}
    Let $M\in \modd kQ$ be $\tau$-rigid and let $N\in \ind\modd kQ$. Then, there exists a unique short exact sequence (up to isomorphism) in $\modd kQ$ \begin{equation}\label{canonical_ses_of_N}
        0\to t_MN \to N \to f_MN \to 0
    \end{equation}
    with $t_M N \in \Gen M$ and $f_M N\in M^\perp$. Similarly, consider the canonical short exact sequence in $\modd\Lambda$ 
    \begin{equation}\label{canonical_ses_of_inducedN}
        0\to t_{\Lambda\otimes_{kQ}M}(\Lambda\otimes_{kQ}N) \to \Lambda\otimes_{kQ}N\to f_{\Lambda\otimes_{kQ}M}(\Lambda\otimes_{kQ}N)\to 0 
    \end{equation}
    with $t_{\Lambda\otimes_{kQ}M}(\Lambda\otimes_{kQ}N)\in \Gen(\Lambda\otimes_{kQ}M)$ and $f_{\Lambda\otimes_{kQ}M}(\Lambda\otimes_{kQ}N)\in (\Lambda\otimes_{kQ}M)^\perp$. Applying the induction functor $\Lambda\otimes_{kQ}-$ to \eqref{canonical_ses_of_N}, we get an exact sequence in $\modd\Lambda$ 
    
    \begin{equation}\label{induced13}
        0\to \Lambda\otimes_{kQ}t_MN\to \Lambda\otimes_{kQ}N\to \Lambda\otimes_{kQ}f_MN\to 0.
    \end{equation}

    Since $\Lambda\otimes_{kQ}\Gen M \subseteq \Gen(\Lambda\otimes_{kQ}M)$, it follows that $\Lambda\otimes_{kQ}t_MN \in \Gen(\Lambda\otimes_{kQ}M)$. Moreover,
    since $f_MN\in M^\perp$, we infer that $\Lambda\otimes_{kQ}f_MN\in (\Lambda\otimes_{kQ}M)^\perp$ by Lemma \ref{tau-rigid_in_J(X)_iff_taurigid_in_J(Lambda_otimes_X)} (i). Hence, since the canonical sequence for $\Lambda\otimes_{kQ}N$ is unique (up to isomorphism), comparing the short exact sequences \eqref{canonical_ses_of_inducedN} and \eqref{induced13}, we get 
    \[\begin{tikzcd}[ampersand replacement=\&,cramped,column sep=small]
	0 \& {t_{\Lambda\otimes_{kQ}M}(\Lambda\otimes_{kQ}N)} \& {\Lambda\otimes_{kQ}N} \& {f_{\Lambda\otimes_{kQ}M}(\Lambda\otimes_{kQ}N)} \& 0 \\
	0 \& {\Lambda\otimes_{kQ}t_MN} \& {\Lambda\otimes_{kQ}N} \& {\Lambda\otimes_{kQ}f_MN} \& 0.
	\arrow[from=1-1, to=1-2]
	\arrow[from=1-2, to=1-3]
	\arrow["\cong"', from=1-2, to=2-2]
	\arrow[from=1-3, to=1-4]
	\arrow[Rightarrow, no head, from=1-3, to=2-3]
	\arrow[from=1-4, to=1-5]
	\arrow["\cong", from=1-4, to=2-4]
	\arrow[from=2-1, to=2-2]
	\arrow[from=2-2, to=2-3]
	\arrow[from=2-3, to=2-4]
	\arrow[from=2-4, to=2-5]
\end{tikzcd}\]
By {\cite[Prop. 4.5]{tauExcSeq_BM}} (see also \cite[Thm. 3.16]{Jasso_Reduction}), we have that the horizontal maps $f_M$ and $f_{\Lambda\otimes_{kQ}M}$ are bijections. Now we show that 
$$ {\left\{ N \in \ind\modd kQ \;\middle|\; \begin{tabular}{@{}l@{}} $N\oplus M$ is $\tau_{kQ}$-rigid, \\ $N\notin \mathrm{Gen} M$ \end{tabular} \right\}} \xrightarrow{\Lambda\otimes_{kQ}-} {\left\{ Y \in \ind\modd\Lambda \;\middle|\; \begin{tabular}{@{}l@{}} $Y\oplus(\Lambda\otimes_{kQ}M)$ is $\tau_{\Lambda}$-rigid, \\ $Y\notin \mathrm{Gen}(\Lambda\otimes_{kQ}M)$ \end{tabular} \right\}}$$

is bijective. Since $N$ is $\tau_{kQ}$-rigid and by Proposition \ref{tau-rigid-Lambda1:1tau-rigid-kQ} $\Lambda\otimes_{kQ}-$ induces a bijection between indecomposable $\tau$-rigid $kQ$-modules and indecomposable $\tau$-rigid $\Lambda$-modules, we have that $(\Lambda\otimes_{kQ}N)\oplus(\Lambda\otimes_{kQ}M)$ is $\tau_\Lambda$-rigid. It remains to show that $\Lambda\otimes_{kQ}N\notin \Gen(\Lambda\otimes_{kQ}M)$. Suppose this was not the case. Then we would have an epimorphism $(\Lambda\otimes_{kQ}M)^r\to (\Lambda\otimes_{kQ}N)$, for some $r\geq 1$. Applying the restriction of scalars we get an epimorphism $M^{dr}\to N^d$ in $\modd kQ$; see Remark \ref{res(ind)}. But this contradicts the fact that $N\notin\Gen M$. The claim follows. 

It remains to prove that 
$${\left\{X\in \mathrm{ind}J(M)\;\middle|\; \begin{tabular}{@{}l@{}} $X$ is $\tau_{J(M)}$-rigid \end{tabular} \right\}} \xrightarrow{\Lambda\otimes_{kQ}-} {\left\{Z\in \mathrm{ind}J(\Lambda\otimes_{kQ}M)\;\middle|\; \begin{tabular}{@{}l@{}} $X$ is $\tau_{J(\Lambda\otimes_{kQ}M)}$-rigid \end{tabular} \right\}}$$
is also bijective. Since $X\cong f_M N \in J(M)$ is $\tau$-rigid in $J(M)$,  it follows from Lemma \ref{tau-rigid_in_J(X)_iff_taurigid_in_J(Lambda_otimes_X)} (iii) that $\Lambda\otimes_{kQ}f_MN$ is $\tau$-rigid in $J(\Lambda\otimes_{kQ}M)$. This shows that the above map is well-defined and injective. Using the fact that $f_{\Lambda\otimes_{kQ}M}(\Lambda\otimes_{kQ}N)\cong \Lambda\otimes_{kQ}f_M N$ for every $N\in \ind \modd kQ$, we obtain that the diagram in the statement commutes. In other words, we have $$\Lambda\otimes_{kQ}f_M(-) = f_{\Lambda\otimes_{kQ}M}(\Lambda\otimes_{kQ}-).$$ Hence, we conclude that the above map between indecomposable $\tau$-rigid modules in $J(M)$ and indecomposable $\tau$-rigid modules in $J(\Lambda\otimes_{kQ}M)$ is also surjective and therefore bijective. This finishes the proof. 
\end{proof}

We have the following observation. 

\begin{lemma}\label{B_M-tau-rigid}
    Let $M\in \modd\Lambda$ be a $\tau$-rigid module. Then, the Bongartz completion $B_M$ is an induced module.  
\end{lemma}

\begin{proof}
   By {\cite[Thm. 2.10]{tau-tiling-theory}} we have that $B_M$ is a $\tau$-tilting module in $\modd\Lambda$. In particular, $B_M$ is $\tau$-rigid and therefore is an induced module by Proposition \ref{tau-rigid-Lambda1:1tau-rigid-kQ}.
\end{proof}

The next proposition is a crucial step toward proving the main result of this section. 

\begin{proposition}\label{indprojJ(M)1-1indprojJ(Lambda_otimes_M)}
    The bijection 
    $$ \Lambda\otimes_{kQ}- : {\left\{X\in \mathrm{ind}J(M)\;\middle|\; \begin{tabular}{@{}l@{}} $X$ is $\tau_{J(M)}$-rigid \end{tabular} \right\}} \to {\left\{Y\in \mathrm{ind}J(\Lambda\otimes_{kQ}M)\;\middle|\; \begin{tabular}{@{}l@{}} $Y$ is $\tau_{J(\Lambda\otimes_{kQ}M)}$-rigid \end{tabular} \right\}} $$
    restricts to a bijection $$ \Lambda\otimes_{kQ}- : \mathrm{ind.proj}J(M)\to \mathrm{ind.proj}J(\Lambda\otimes_{kQ}M).$$
\end{proposition}

\begin{proof}
    By {\cite[Lemma 4.9]{tauExcSeq_BM}}, the torsion-free functor $f_M$ induces a bijection 
    $$\mathcal{P}(^\perp\tau M)\setminus \mathrm{ind.}\add M \to \mathrm{ind.proj}J(M)$$
    in $\modd kQ$, and similarly there is a bijection 
    $$\mathcal{P}(^\perp\tau(\Lambda\otimes_{kQ}M))\setminus \mathrm{ind.}\add(\Lambda\otimes_{kQ}M) \to\mathrm{ind.proj}J(\Lambda\otimes_{kQ}M)$$
    in $\modd\Lambda$ induced by $f_{\Lambda\otimes_{kQ}M}$. Hence, using Proposition \ref{tau_rigid_in_J(M)_1:1_tau_rigid_in_J(Lambda_otimes_kQM)}, proving the bijection between indecomposable projective modules in $J(M)$ and indecomposable projective modules in $J(\Lambda\otimes_{kQ}M)$ is equivalent to showing that the induction functor $\Lambda\otimes_{kQ}-$ induces a bijection 
    $$\mathcal{P}(^\perp\tau M)\setminus \mathrm{ind.}\add M \to \mathcal{P}(^\perp\tau(\Lambda\otimes_{kQ}M))\setminus \mathrm{ind.}\add(\Lambda\otimes_{kQ}M).$$
    Since $B_{\Lambda\otimes_{kQ}M}$ is a $\tau$-tilting $\Lambda$-module and $|kQ| = |\Lambda|$, we obtain that
    $$ |\mathcal{P}(^\perp\tau M)\setminus \mathrm{ind.}\add M| = |\mathcal{P}(^\perp\tau(\Lambda\otimes_{kQ}M))\setminus \mathrm{ind.}\add(\Lambda\otimes_{kQ}M)|.$$
    Moreover, using Lemma \ref{B_M-tau-rigid} we get that 
    $$\mathcal{P}(^\perp\tau(\Lambda\otimes_{kQ}M))\setminus \mathrm{ind.}\add(\Lambda\otimes_{kQ}M) = {\left\{ \Lambda\otimes_{kQ}X \;\middle|\; \begin{tabular}{@{}l@{}} $\Ext_\Lambda^1(\Lambda\otimes_{kQ}X, {^\perp\tau(\Lambda\otimes_{kQ}M)}) = 0$, \\ $X\in \ind\modd kQ$ \end{tabular} \right\}}.$$
    Observe that, if $X$ is indecomposable in $\modd kQ$, so is $\Lambda\otimes_{kQ}X$ in $\modd\Lambda$; see Proposition \ref{HomExt_tensor_product}. So, let $\Lambda\otimes_{kQ}X \in \mathcal{P}(^\perp\tau(\Lambda\otimes_{kQ}M))\setminus \mathrm{ind.}\add(\Lambda\otimes_{kQ}M)$. Then, by Lemma \ref{tau-rigid_in_J(X)_iff_taurigid_in_J(Lambda_otimes_X)} (iv), $X$ is in $\mathcal{P}(^\perp\tau M)\setminus \mathrm{ind.}\add M$. We conclude that the induction functor induces the desired bijection. The claim follows. 
\end{proof}

We have the following immediate corollary. 

\begin{corollary}\label{Bongartz_commutes_with_induction}
        Let $M$ be a $\tau$-rigid $kQ$-module. Then, $B_{\Lambda\otimes_{kQ}M}\cong \Lambda\otimes_{kQ} B_M$ in $\modd\Lambda$. 
\end{corollary}

We are now prepared to state and prove the main result of this section. 

\begin{theorem}\label{J(Lambda_otimes_kQ)=mod(R_otimes_kQ')}
      Let $\Lambda = R\otimes kQ$. Let $M$ be a basic $\tau$-rigid $kQ$-module and let $\Lambda\otimes_{kQ}M$ be the corresponding basic $\tau$-rigid $\Lambda$-module. Then, $J(\Lambda\otimes_{kQ}M)\simeq \modd\Gamma_{\Lambda\otimes_{kQ}M}$, with $$ \Gamma_{\Lambda\otimes_{kQ}M} \cong R\otimes \Gamma_M \quad \text{and} \quad \Gamma_M = \End_{kQ}(B_M)^{\mathrm{op}}/\langle e_M\rangle$$
      where $e_M$ is the idempotent corresponding to the projective $\End_{kQ}(B_M)^\mathrm{op}$-module $\Hom_{kQ}(B_M,M)$. In particular, $\Gamma_M$ is hereditary. 
\end{theorem}

\begin{proof}
    Let $\Lambda\otimes_{kQ}M$ be a basic $\tau$-rigid $\Lambda$-module. We know from {\cite[Theorem 3.8]{Jasso_Reduction}} that the $\tau$-perpendicular subcategory $J(\Lambda\otimes_{kQ}M)$ is equivalent to $\modd\Gamma_{\Lambda\otimes_{kQ}M}$, where 
    \begin{align*}
        \Gamma_{\Lambda\otimes_{kQ}M} &= \End_{\Lambda}(B_{\Lambda\otimes_{kQ}M})^{\mathrm{op}}/\langle e_{\Lambda\otimes_{kQ}M}\rangle\\ &= \End_{kQ}(B_{\Lambda\otimes_{kQ}M})^{\mathrm{op}}/\left\{f \mid f \text{ factors through } \Lambda\otimes_{kQ}M\right\}.
    \end{align*}
    Using Corollary \ref{Bongartz_commutes_with_induction}, we get that 
    $$ \End_\Lambda(B_{\Lambda\otimes_{kQ}M})^{\mathrm{op}} \cong \End_\Lambda(\Lambda\otimes_{kQ}B_M)^{\mathrm{op}}\cong R\otimes \End_{kQ}(B_M)^{\mathrm{op}}. $$
    Moreover, the idempotent $e_{\Lambda\otimes_{kQ}M}$ corresponding to the projective $\End_{\Lambda}(B_{\Lambda\otimes_{kQ}M})^\mathrm{op}$-module\\ $\Hom_\Lambda(B_{\Lambda\otimes_{kQ}M},\Lambda\otimes_{kQ}M)$ gets sent to the idempotent $1\otimes e_M$ under the isomorphism 
    $$ \Hom_\Lambda(B_{\Lambda\otimes_{kQ}M}, \Lambda\otimes_{kQ}M) \cong R\otimes \Hom_{kQ}(B_M,M).$$
    Hence, we obtain that 
    \begin{align*}
        \Gamma_{\Lambda\otimes_{kQ}M} = \End_{\Lambda}(B_{\Lambda\otimes_{kQ}M})^{\mathrm{op}}/\langle e_{\Lambda\otimes_{kQ}M}\rangle 
        &\cong (R\otimes \End({B_M})^{\mathrm{op}})/\langle 1\otimes e_M\rangle\\
        &\cong R\otimes(\End_{kQ}(B_M)^{\mathrm{op}}/\langle e_M \rangle) = R \otimes\Gamma_M.
    \end{align*}
    In particular, $\Gamma_M$ is hereditary by \cite[Cor. 3.18 (a)]{Jasso_Reduction}; see also Remark \ref{Lambda_hereditary_J_hereditary}. This finishes the proof. 
\end{proof}

Let $M$ be a basic $\tau$-rigid $kQ$-module and let $\Lambda\otimes_{kQ}M$ be the corresponding basic $\tau$-rigid $\Lambda$-module.  By \cite[Thm. 3.8]{Jasso_Reduction} we have equivalences of categories $$G_M := \Hom_{kQ}(B_M,-): J(M)\to \modd \Gamma_M$$ and $$G_{\Lambda\otimes_{kQ}M} := \Hom_{\Lambda}(B_{\Lambda\otimes_{kQ}M},-): J(\Lambda\otimes_{kQ}M)\to \modd \Gamma_{\Lambda\otimes_{kQ}M}.$$ We denote by $F_M$ and $F_{\Lambda\otimes_{kQ}M}$ the inverses of $G_M$ and $G_{\Lambda\otimes_{kQ}M}$, respectively. Using this notation and the one from Theorem \ref{J(Lambda_otimes_kQ)=mod(R_otimes_kQ')} we can deduce the following. 

\begin{corollary}\label{transport_of_structure}
    There is a commutative diagram of the form 
\[
\begin{tikzcd}[ampersand replacement=\&, column sep=normal, row sep=normal]
	J(M) 
	\arrow[rr, bend left=15, "{G_M}"] 
	\arrow[dd, "{{\Lambda\otimes_{kQ}-}}"'] 
	\&\& \modd \Gamma_M 
	\arrow[ll, bend left=15, "{F_M}"] 
	\arrow[dd, "{\Gamma_{\Lambda\otimes_{kQ}M}\otimes_{\Gamma_M}-}"] \\
	\\
	J({\Lambda\otimes_{kQ}M}) 
	\arrow[rr, bend left=15, "{G_{{\Lambda\otimes_{kQ}M}}}"] 
	\&\& \modd \Gamma_{\Lambda\otimes_{kQ}M} 
	\arrow[ll, bend left=15, "{F_{{\Lambda\otimes_{kQ}M}}}"', swap]
\end{tikzcd}
\]
\end{corollary}

\begin{proof}
    Let $N\in J(M)$. Then, 
    \begin{align*}
        \Gamma_{\Lambda\otimes_{kQ}M}\otimes_{\Gamma_M}G_M(N) &= \Gamma_{\Lambda\otimes_{kQ}M}\otimes_{\Gamma_M}\Hom_{kQ}(B_M,N)\\
        &\cong R\otimes \Gamma_M \otimes_{\Gamma_M}\Hom_{kQ}(B_M,N) &&\text{by Theorem \ref{J(Lambda_otimes_kQ)=mod(R_otimes_kQ')}}\\
        &\cong R\otimes \Hom_{kQ}(B_M,N)\\
        &\cong \Hom_R(R,R)\otimes \Hom_{kQ}(B_M,N)\\
        &\cong \Hom_{R\otimes kQ}(R\otimes B_M, R\otimes N)\\
        &\cong \Hom_\Lambda(\Lambda\otimes_{kQ}B_M, \Lambda\otimes_{kQ}N)\\
        &\cong \Hom_\Lambda(B_{\Lambda\otimes_{kQ}M}, \Lambda\otimes_{kQ}N) &&\text{by Corollary \ref{Bongartz_commutes_with_induction}}\\
        &= G_{\Lambda\otimes_{kQ}M}(\Lambda\otimes_{kQ}N). 
    \end{align*}
    The claim follows. 
\end{proof}

\begin{corollary}
    Let $M\in \modd kQ$ and let $N\in J(M)$. Then, $$\tau_{J(\Lambda\otimes_{kQ}M)}(\Lambda\otimes_{kQ}N)\cong \Lambda\otimes_{kQ}\tau_{J(M)}N$$ in $J(\Lambda\otimes_{kQ}M)$.
\end{corollary}

\begin{proof}
    By Theorem \ref{J(Lambda_otimes_kQ)=mod(R_otimes_kQ')} we have that $J(\Lambda\otimes_{kQ}M)\simeq \modd(R\otimes \Gamma_M)$, for a hereditary alebra $\Gamma_M$. The result follows from Proposition \ref{tauCommutesWithInduction} (iii). 
\end{proof}

\section{(signed) $\tau$-exceptional sequences}\label{section3}

In this section, let $\Lambda = R\otimes kQ$ unless stated otherwise, where $R$ is a local commutative finite dimensional algebra. Our goal is to establish a bijection between complete (signed) ($\tau$-)exceptional sequences in $\modd kQ$ and complete (signed) $\tau$-exceptional sequences in $\modd\Lambda$. $\tau$-exceptional sequences are a generalization of exceptional sequences introduced by Crawley-Boevey in \cite{Boevey-ExcSeq}. We start by recalling the definition of an exceptional sequence.

\begin{definition}\label{def_exc_seq}
    Let $\Lambda$ be a finite dimensional algebra. For a positive integer $t$, an ordered $t$-tuple of indecomposable $\Lambda$-modules $(M_1,\cdots, M_t)$ in $\modd\Lambda$ is called \emph{exceptional} if the following conditions hold. 
    \begin{enumerate}
        \item[(a)] $\End_\Lambda(M_i)\cong k$ for $1\leq i\leq t$;
        \item[(b)]  $\Ext_\Lambda^{\geq 1}(M_i,M_i) = 0$ for $1\leq i \leq t$; 
        \item[(c)]  $\Hom_\Lambda(M_i,M_j) = 0 = \Ext_\Lambda^{\geq 1}(M_i,M_j)$ for $1\leq j < i \leq t$.  
    \end{enumerate}
    If $t = \abs\Lambda$ the sequence is said to be \emph{complete}.
\end{definition}

Let $\Lambda$ be a finite dimensional algebra. Recall that a pair of $\Lambda$-modules $(M,P)$ is called $\tau$-rigid if $M$ is $\tau$-rigid, $P\in\proj\Lambda$, and $\Hom_\Lambda(P,M) = 0$. Buan and Marsh considered the corresponding object $M\oplus P[1]$ in the full subcategory $\mathcal{C}(\Lambda) = \modd\Lambda \oplus \modd\Lambda[1]$ of $D^b(\modd\Lambda)$. 

\begin{definition}[{\cite[Def. 1.1]{tauExcSeq_BM}}]
    Let $\Lambda$ be a finite dimensional algebra. An object $M\oplus P[1]$ in $\C(\Lambda)$ is called \emph{support $\tau$-rigid} if
    \begin{enumerate}
        \item[(a)] $M$ is a $\tau$-rigid $\Lambda$ module;
        \item[(b)] $P$ is a projective $\Lambda$-module and satisfies $\Hom_\Lambda(P,M) = 0$.
    \end{enumerate}
    An object $M\oplus P[1]$ is called \emph{support $\tau$-tilting} if $M\oplus P[1]$ is support $\tau$-rigid and $\abs{M}+\abs{P} = \abs{\Lambda}$.
\end{definition}

Let $M$ be an indecomposable $\tau$-rigid $\Lambda$-module. Recall that the $\tau$-perpendicular subcategory $J(M)$ is equivalent to the module category of a finite dimensional algebra $\Gamma_M$ with $\abs{\Lambda}-1$ simple modules. If $P\in \proj\Lambda$, define $J(P[1]) = J(P)$ and set $\Gamma_{P[1]} = \Gamma_P$. Hence, for an indecomposable object $X\in \C(\Lambda)$,  we have an equivalence $J(X)\simeq \modd\Gamma_X$. For a full subcategory $\mathcal{X}$ of $\modd\Lambda$, define $\C(\mathcal{X})$ to be the full subcategory $\mathcal{X}\oplus \mathcal{X}[1]$ of $\C(\Lambda)$. Hence, for $X$ as above, using the fact that $J(X)$ is a wide subcategory of $\modd\Lambda$, one can see that $\C(J(X))\simeq \C(\Gamma_X)$.

An object $M\oplus P[1]\in \C(J(X))$ is \textit{support $\tau$-rigid} in $\C(J(X))$ if the corresponding object in $\C(\Gamma_X)$ is support $\tau$-rigid, that is $M$ is $\tau_{J(X)}$-rigid, $P$ lies in $\proj J(X)$, and $\Hom_{\Gamma_X}(P,M) = 0$. We recall the following definitions. 

\begin{definition}[{\cite[Def. 1.2]{tauExcSeq_BM}}]
    Let $\Lambda$ be a finite dimensional algebra. For a positive integer $t$, an ordered $t$-tuple of indecomposable objects $(\T_1,\cdots, \T_t)$ in $\C(\Lambda)$ is called an \emph{ordered support $\tau$-rigid object} if $\bigoplus_{i=1}^t\T_i$ is a basic support $\tau$-rigid object. If, in addition, $t=n$, then $(\T_1,\cdots, \T_t)$ is called an ordered support $\tau$-tilting object. 
\end{definition}

\begin{definition}[{\cite[Def. 1.3]{tauExcSeq_BM}}]
    Let $\Lambda$ be a finite dimensional algebra and let $t$ be a positive integer. An ordered $t$-tuple of indecomposable objects $(\U_1,\cdots, \U_t)$ in $\C(\Lambda)$ is a \emph{signed $\tau$-exceptional sequence} if 
    \begin{enumerate}
        \item[(a)] $\U_t$ is support $\tau$-rigid in $\C(\Lambda)$, and
        \item[(b)] $(\U_1,\cdots, \U_{t-1})$ is a signed $\tau$-exceptional sequence in $\C(J(M_t))$.
    \end{enumerate}
     A \emph{$\tau$-exceptional sequence} $(\U_1,\cdots, \U_t)$ is a signed $\tau$-exceptional sequence in which every indecomposable object lies in $\modd\Lambda$. If $t=n$, the sequence is said to be \emph{complete}.  
\end{definition}  

\begin{theorem}[{\cite[Thm. 5.4]{tauExcSeq_BM}}]\label{main_BM}
    Let $\Lambda$ be a finite dimensional algebra. For each $t\in \{1,\cdots, n\}$ there is a bijection $\varphi_t$ from the set of signed $\tau$-exceptional sequences of length $t$ in $\C(\Lambda)$ to the set of ordered support $\tau$-rigid objects of length $t$ in $\C(\Lambda)$. 
\end{theorem}

Given a $\tau$-exceptional sequence $\N = (N_1,\cdots, N_t)$ we write $$J(N_i,\cdots, N_t) = J_{J(N_{i+1},\cdots, N_t)}(N_i)$$ for the iterated $\tau$-perpendicular subcategory where $J(N_{t-1},N_t) = J_{J(N_t)}(N_{t-1})$.\\

Let $\Lambda = R\otimes kQ$ for the remainder of this section. Corollary \ref{stautitl(kQ)1:1stautilt(Lambda)} and Proposition \ref{indprojJ(M)1-1indprojJ(Lambda_otimes_M)} allow us to extend the definition of the induction functor on support $\tau$-rigid objects in the following way. 

\begin{definition}
    Let $X$ be a basic $\tau$-rigid $kQ$-module and let $J(X)\subseteq \modd kQ$ be a $\tau$-perpendicular subcategory. Let $U = M\oplus P[1]$ be a support $\tau$-rigid object in $\C(J(X))$.  We define $$\Lambda\otimes_{kQ}U = (\Lambda\otimes_{kQ}M)\oplus (\Lambda\otimes_{kQ}P[1]) := (\Lambda\otimes_{kQ}M)\oplus (\Lambda\otimes_{kQ}P)[1].$$
\end{definition}

 Let $\E = (M_1,\cdots, M_t)$ be an exceptional sequence in $\modd kQ$. We define $$\Lambda\otimes_{kQ}\E := (\Lambda\otimes_{kQ}M_1, \cdots, \Lambda\otimes_{kQ}M_t),$$ a sequence of $\Lambda$-modules.
 
\begin{proposition}\label{PTensorTauExcIs_TauExc}
    Let $t\in\{1,\cdots,n\}$ and let $\N = (N_1,\cdots, N_t)$ be a $\tau$-exceptional sequence in $\modd kQ$. Then, $\Lambda\otimes_{kQ}\N = (\Lambda\otimes_{kQ}N_1,\cdots, \Lambda\otimes_{kQ}N_t)$ is a $\tau$-exceptional sequence in $\modd\Lambda$. 
\end{proposition}

\begin{proof}
    We proceed by induction on the length of $\N$. By Corollary \ref{inductionPreservesTauRigidModules}, we have that $\Lambda\otimes_{kQ}N_t$ is $\tau$-rigid in $\modd\Lambda$. Since $N_{t-1}$ is $\tau$-rigid in $J(N_t)$, if follows that $\Lambda\otimes_{kQ}N_{t-1}$ is $\tau$-rigid in $J(\Lambda\otimes_{kQ}N_t)$ by Proposition \ref{tau_rigid_in_J(M)_1:1_tau_rigid_in_J(Lambda_otimes_kQM)}. This proves the base case.
    
    Now let $i\leq t-2$. Assume the claim holds for $i+1$, that is $(\Lambda\otimes_{kQ}N_{i+1},\cdots, \Lambda\otimes_{kQ}N_{t-1})$ is a $\tau$-exceptional sequence in $J(\Lambda\otimes_{kQ}N_t)$. We want to show that the claim holds for $i$. In other words, we want to show that $\Lambda\otimes_{kQ}N_i$ is $\tau$-rigid in 
    ${J(\Lambda\otimes_{kQ} N_{i+1},\cdots,\Lambda\otimes_{kQ}N_{t})}$. Applying Theorem \ref{J(Lambda_otimes_kQ)=mod(R_otimes_kQ')} inductively, we infer that 
    $$ J(\Lambda\otimes_{kQ} N_{i+1},\cdots,\Lambda\otimes_{kQ}N_{t}) \simeq \modd{\left( R\otimes \Gamma_{N_{i+1}} \right)}$$
    where $J(N_{i+1},\cdots, N_t)\simeq \modd{\Gamma_{N_{i+1}}}$  for a hereditary algebra $\Gamma_{N_{i+1}}$. Since $N_i$ is $\tau$-rigid in $J(N_{i+1},\cdots, N_t)$ by assumption, the claim follows from Lemma \ref{tau-rigid_in_J(X)_iff_taurigid_in_J(Lambda_otimes_X)} (iii). This finishes the proof.
\end{proof}

A key statement for the proof of the main result of this section is the following case of a theorem by Buan and Hanson. 

\begin{theorem}[{\cite[Theorem 6.4]{tau-perpendicular_wide_subategories}}]\label{Buan-Hanson-Theorem}
    Let $\Lambda$ be a finite dimensional algebra and let $M\oplus N$ be a basic $\tau$-rigid object in $\modd\Lambda$. Suppose that $N\notin \Gen M$. Then, $$J(M\oplus N) = J_{J(M)}(f_MN).$$
\end{theorem}

We are now ready to prove the main result of this section.

\begin{theorem}\label{bijection_of_tau_exc_seq}
    Let $t\in\{1,\cdots,n\}$. Then, the induction functor induces a bijection between the set of ($\tau$-)exceptional sequences of length $t$ in $\modd{kQ}$ and the set of $\tau$-exceptional sequences of length $t$ in $\modd{\Lambda}$. 
\end{theorem}

\begin{proof}
    Let $\M=(M_1,\cdots, M_t)$ be a $\tau$-exceptional sequence in $\modd{kQ}$. Then, $$\Lambda\otimes_{kQ}\M = (\Lambda\otimes_{kQ} M_1,\cdots, \Lambda\otimes_{kQ}M_t)$$ is a $\tau$-exceptional sequence in $\modd{\Lambda}$ by Proposition \ref{PTensorTauExcIs_TauExc}. In particular, $\Lambda\otimes_{kQ}-$ is injective on ($\tau$-)exceptional sequences.
    
    In order to prove surjectivity, let $\X = (X_1,\cdots, X_t)$ be a $\tau$-exceptional sequence in $\modd\Lambda$. We want to show there exists a $\tau$-exceptional $\M = (M_1,\cdots, M_t)$ in $\modd kQ$ such that $\Lambda\otimes_{kQ}\M = \X$.  Since $X_t$ is indecomposable $\tau$-rigid in $\modd\Lambda$, we get that $X_t\cong\Lambda\otimes_{kQ}M_t$ for an indecomposable $\tau$-rigid $kQ$-module $M_t$ by Proposition \ref{tau-rigid-Lambda1:1tau-rigid-kQ}. By assumption, we have that $X_{t-1}$ is indecomposable $\tau$-rigid in $J(\Lambda\otimes_{kQ}M_t)$. Using the commutativity of the diagram in Proposition \ref{tau_rigid_in_J(M)_1:1_tau_rigid_in_J(Lambda_otimes_kQM)}, we infer that $X_{t-1}\cong \Lambda\otimes_{kQ}M_{t-1}$, where $M_{t-1}$ is an indecomposable $\tau$-rigid module in $J(M_t)$. In particular, $M_{t-1}$ is of the form $f_{M_t}N$ for an indecomposable $\tau$-rigid $kQ$-module $N$. 
    
    Now consider $X_{t-2}\in J_{J(\Lambda\otimes_{kQ}M_t)}(\Lambda\otimes_{kQ}M_{t-1})$. By assumption, we have that $X_{t-2}$ is $\tau$-rigid in this subcategory. Since $M_{t-1}$ is of the form $f_{M_t}N$ for an indecomposable $\tau$-rigid $kQ$-module $N$, combining Proposition \ref{tau_rigid_in_J(M)_1:1_tau_rigid_in_J(Lambda_otimes_kQM)} and Theorem \ref{Buan-Hanson-Theorem} we get that
    \begin{align*}
        J_{J(\Lambda\otimes_{kQ}M_t)}(\Lambda\otimes_{kQ}M_{t-1}) &= J_{J(\Lambda\otimes_{kQ}M_t)}(\Lambda\otimes_{kQ}f_{M_t}N)\\
        &= J_{J(\Lambda\otimes_{kQ}M_t)}(f_{\Lambda\otimes_{kQ}M_{t}}(\Lambda\otimes_{kQ}N))\\
        &= J((\Lambda\otimes_{kQ}M_{t}) \oplus (\Lambda\otimes_{kQ}N))\\
        &= J(\Lambda\otimes_{kQ}(M_t\oplus N)).
    \end{align*}
    Using again the commutativity of the diagram in Proposition \ref{tau_rigid_in_J(M)_1:1_tau_rigid_in_J(Lambda_otimes_kQM)}, we infer that $X_{t-2}\cong \Lambda\otimes_{kQ}M_{t-2}$ for an indecomposable $\tau_{J(M_t\oplus N)}$-rigid module $M_{t-2}$, where$$J(M_t\oplus N) = J_{J(M_t)}(f_{M_t}N) = J_{J(M_t)}(M_{t-1})$$ by Theorem \ref{Buan-Hanson-Theorem}.
    Iterating this argument inductively on the length of the sequence, we can construct a $\tau$-exceptional sequence $\M = (M_1,\cdots, M_t)$ in $\modd kQ$ such that $\Lambda\otimes_{kQ}\M = \X$. This proves surjectivity. The claim follows. 
\end{proof}

\begin{remark}\label{central_element_rmk}
     We observed that $I:=\langle x_i\otimes 1 \mid 1\leq i\leq t \rangle$ is an ideal generated by central elements and contained in the Jacobson radical of $\Lambda$ and $kQ \cong \Lambda/I$. Hence, if the lattice of torsion classes of $\Lambda$ is finite, it is isomorphic to the lattice of torsion classes of $kQ$, as shown in \cite{ReductionTheoremForTauRigidModules}. Since $\tau$-exceptional sequences can be determined from the underlying structure of the lattice of torsion classes \cite[Thm 8.10, Rmk 8.11]{Barnard_Hanson}, one can recover an alternative bijection between complete ($\tau$-)exceptional sequences in $\modd kQ$ and complete $\tau$-exceptional sequences in $\modd\Lambda$, provided that $\Lambda$ is $\tau$-tilting finite. However, we remark that no $\tau$-tilting finiteness assumption is required in Theorem \ref{bijection_of_tau_exc_seq}.   
\end{remark}

The following is a significant consequence of Theorem \ref{bijection_of_tau_exc_seq}. 

\begin{corollary}\label{bijection_signed_tau_exceptional_sequences}
     Let $t\in\{1,\cdots,n\}$. Then, the induction functor induces a bijection between the set of signed ($\tau$-)exceptional sequences of length $t$ in $\C({kQ})$ and the set of signed $\tau$-exceptional sequences of length $t$ in $\C(\Lambda)$.
\end{corollary}

\begin{proof}
    The claim follows combining Proposition \ref{indprojJ(M)1-1indprojJ(Lambda_otimes_M)} and Theorem \ref{bijection_of_tau_exc_seq}. 
\end{proof}

\section{$\tau$-cluster morphism categories}\label{section4}

Let $\Lambda$ be a finite dimensional algebra. The $\tau$-\textit{cluster morphism category} of $\Lambda$ is a small category $\Mcluster(\Lambda)$ whose objects are the $\tau$-perpendicular subcategories of $\modd\Lambda$ and whose morphisms are indexed by support $\tau$-rigid pairs in these subcategories. This category was first defined for hereditary algebras by Igusa and Todorov as the "cluster morphism category" in \cite{signed_exc_seq}. Later, the definition was extended to the $\tau$-tilting finite case by Buan and Marsh in \cite{a_category_of_wide_subcategories}, and given the name of "$\tau$-cluster morphism category" in \cite{Hanson-Igusa_tau-cluster_morphism_categories_and_picture_groups}. Finally, Buan and Hanson defined the $\tau$-cluster morphism category for an arbitrary finite dimensional algebra in \cite{tau-perpendicular_wide_subategories}.

Every morphism in $\Mcluster(\Lambda)$ factorizes uniquely into the composition of irreducible morphisms. It was shown in \cite[Prop. 11.8]{a_category_of_wide_subcategories} that the composition of irreducible morphisms in $\Mcluster(\Lambda)$ corresponds to signed $\tau$-exceptional sequences in $\C(\Lambda)$. Motivated by this result and Corollary \ref{bijection_signed_tau_exceptional_sequences}, this section aims to prove an equivalence of categories between $\Mcluster(kQ)$ and $\Mcluster(\Lambda)$, where $\Lambda = R\otimes kQ$.\\

This section uses the following definition of a $\tau$-perpendicular subcategory. 

\begin{definition}[{\cite[Def. 3.2]{tau-perpendicular_wide_subategories}}]\label{tau-perp_subcat_BH}
    Let $\Lambda$ be a finite dimensional algebra. A full subcategory $\W\subseteq\modd\Lambda$ is called a \emph{$\tau$-perpendicular subcategory} if there exists a support $\tau$-rigid object $U = M\oplus P[1]\in \C(\Lambda)$ such that $\W = J(U)$, where $$J(U) := {^\perp}\tau M\cap (M\oplus P)^{\perp}.$$
\end{definition}

Recall that an object $N\oplus Q[1]\in \C(J(U))\subseteq \C(\Lambda)$ is \textit{support $\tau$-rigid} in $\C(J(U))$ if the corresponding object in $\C(\Gamma_U)$ is support $\tau$-rigid, that is $N$ is $\tau_{J(U)}$-rigid, $Q$ lies in $\proj J(U)$, and $\Hom_{J(U)}(Q,N) = 0$.

\begin{theorem}[{\cite[Thm. 2.8]{BHM_mutation}}]\label{BM_bijection}
    Let $\Lambda$ be a finite dimensional algebra and let $U = M\oplus P[1]$ be a support $\tau$-rigid object in $\C(\Lambda)$. Then there is a bijection 
            \[
            \begin{tikzcd}
            {\left\{V\in \ind\C(\Lambda)\mid V\oplus U \text{ support }\tau\text{-rigid} \right\}} 
            \arrow[d, "\mathcal{E}_U"] \\
            {\left\{ W\in \ind\C(J(U)) \mid W \text{ support }\tau_{J(U)}\text{-rigid}\right\}}.
            \end{tikzcd}
            \]
        \begin{enumerate}
            \item[(a)] For $V\in\ind\C(\Lambda)$ with $V\oplus U$ support $\tau$-rigid, we have $\mathcal{E}_U(V)\in (\proj J(U))[1]$ if and only if $V\in \Gen M$ or $V\in(\proj\Lambda)[1]$; 
            \item[(b)] If $U\in \proj\Lambda[1]$ and $V\in \ind\C(\Lambda)$ with $V\oplus U$ support $\tau$-rigid, we have $\mathcal{E}_U(V) = V$. Equivalently, if $W$ is a $\tau_{J(P[1])}$-rigid module, then $\mathcal{E}^{-1}_{P[1]}(W) = W$;  
            \item[(c)] If $V\in (\ind\bmod{\Lambda})\setminus \Gen M$ with $V\oplus U$ support $\tau$-rigid, then $\mathcal{E}_U(V) = \mathcal{E}_M(V) = f_MV$. 
        \end{enumerate}
\end{theorem}

The following result will be used multiple times in this section. 

\begin{theorem}[{\cite[Thm. 6.12]{tau-perpendicular_wide_subategories}}]\label{Eps_{U+V}}
    Let $\Lambda$ be a finite dimensional algebra and let $\W\subseteq\modd\Lambda$ be a $\tau$-perpendicular subcategory. Let $U\oplus V\in\C(\W)$ be support $\tau$-rigid an basic. Then, $$\Eps^{\W}_{U\oplus V} = \Eps_{\Eps_U(V)}^{J_\W(U)}\circ \Eps_U^{\W}.$$
\end{theorem}

\begin{theorem}[{\cite[Theorem 6.4]{tau-perpendicular_wide_subategories}}]\label{Buan-Hanson-Theorem-General}
    Let $\Lambda$ be a finite dimensional algebra and let $\W\subseteq \modd \Lambda$ be a $\tau$-perpendicular subcategory. Let $U\oplus V$ be a basic and support $\tau$-rigid object in $\C(\Lambda)$. Then, $$J_\W(U\oplus V) = J_{J_\W(U)}(\Eps_U^{\W}(V)).$$
\end{theorem}

The following observation reformulates Definition \ref{tau-perp_subcat_BH} in terms of Definition \ref{tau-perp_Jasso}. 

\begin{lemma}\label{two_def_of_J}
    Let $\Lambda$ be a finite dimensional algebra. Let $U = M \oplus P[1]\in \C(\Lambda)$ be  support $\tau$-rigid and let $J(U)$ be a $\tau$-perpendicular subcategory. Then $$J(U) = J_{J(P)}(M) = J(\overline{U}),$$ where $\overline{U} = P\oplus \Eps^{-1}_P(M)$. 
\end{lemma}

\begin{proof}
    By Theorem \ref{BM_bijection}(b) we have that $\Eps_{P[1]}(M) = M$. Moreover, by definition $J(P[1]) = J(P)$. Hence, Theorem \ref{Buan-Hanson-Theorem-General} gives \begin{align*}
        J(U) &= J(M\oplus P[1]) \\
             &= J_{J(P[1])}(\Eps_{P[1]}(M))\\
             &= J_{J(P)}(M)\\
    \end{align*}
    (see also Proof of \cite[Lemma 2.17(a)]{BHM_mutation}). Applying Theorem \ref{Buan-Hanson-Theorem-General} again, we obtain that $J_{J(P)}(M) = J(P\oplus \Eps^{-1}_P(M)).$ The claim follows. 
\end{proof}

We are ready to recall the definition of the $\tau$-cluster morphism category from {\cite{tau-perpendicular_wide_subategories}}. 

\begin{definition}[{\cite[Def. 6.1]{tau-perpendicular_wide_subategories}}]\label{def_tau-cluster_morph_cat}
    Let $\Lambda$ be a finite dimensional algebra. The \emph{$\tau$-cluster morphism category} of $\Lambda$, denoted as $\Mcluster(\Lambda)$, consists of the following data. 
    \begin{enumerate}
        \item[(a)] The objects of $\Mcluster(\Lambda)$ are the $\tau$-perpendicular subcategories of $\modd\Lambda$. 
        \item[(b)] For a $\tau$-perpendicular subcategory $\W\subseteq \modd{\Lambda}$ and $U\in \C(\W)$ support $\tau$-rigid and basic, define a formal symbol $g_U^\W$. 
        \item[(c)]  Given $\W_1, \W_2$ two $\tau$-perpendicular subcategories of $\modd{\Lambda}$, we define 
        $$\Hom_{\Mcluster(\Lambda)}(\W_1,\W_2) = {\left\{ g_U^{\W_1} \;\middle|\; \begin{tabular}{@{}l@{}} U \text{ is a basic support $\tau$-rigid object in } $\C(\W_1)$ \\ \text{and} $\W_2 = J_{\W_1}(U)$ \end{tabular} \right\}}.$$
        In particular: 
        \begin{enumerate}
            \item[(i)] If $\W_1 \not\supseteq \W_2$, then $\Hom_{\Mcluster(\Lambda)}(\W_1,\W_2) = \emptyset$;
            \item[(ii)] $\Hom_{\Mcluster(\Lambda)}(\W_1,\W_1) = g_0^{\W_1}$. 
        \end{enumerate}
        \item[(d)]  Given $g_U^{\W_1}:\W_1\to \W_2$ and $g_V^{\W_2}: \W_2\to \W_3$ in $\Mcluster(\Lambda)$, denote $\widetilde{V}:= (\mathcal{E}_U^{\W_1})^{-1}(V)$. We define the composition to be 
        $$ g_V^{\W_2}\circ g_U^{\W_1}:= g_{U\oplus\widetilde{V}}^{\W_1}.$$
    \end{enumerate}
\end{definition}

A morphism $g$ in $\Mcluster(\Lambda)$ is called \textit{irreducible} \cite[Def. 10.2]{a_category_of_wide_subcategories} if, whenever $g$ is expressed as a composition $g_1\circ g_2$, we have that either $g_1$ or $g_2$ is the identity map. Moreover, recall that for every $t\in\{1,\cdots, n\}$, there is a bijection $\varphi_t$ between the set of signed $\tau$-exceptional sequences of length $t$ in $\C(\Lambda)$ and ordered support $\tau$-rigid objects of length $t$ in $\C(\Lambda)$ (see Theorem \ref{main_BM}). 

\begin{remark}\label{BM_inverse_bijection}
    Let $\Lambda$ be a finite dimensional algebra and let $\W$ be a $\tau$-perpendicular subcategory of $\modd \Lambda$. The bijection $\varphi_t$ from Theorem \ref{main_BM} has an inverse bijection $\psi_t$ \cite[Remark 5.13]{tauExcSeq_BM} from the set of ordered support $\tau$-rigid objects of length $t$ in $\C(\W)$ to the set of signed $\tau$-exceptional sequences of length $t$ in $\C(\W)$ that we now describe. Let $(\T_1,\cdots,\T_t)$ be an ordered support $\tau$-rigid object in $\C(\W)$. Let 
    \begin{alignat*}{2}
    \W_t      &= \W              &\qquad  \U_t      &= \T_t \\
    \W_{t-1}  &= J_{W_t}(\U_t)   &\qquad  \U_{t-1}  &= \Eps_{\U_t}^{\W_t}(\T_{t-1})\\
    &\vdots &\qquad &\vdots \\
    \W_i &= J_{\W_{i+1}}(\U_{i+1}) &\qquad \U_i &= \Eps_{\U_{i+1}}^{\W_{i+1}}\cdots\Eps_{\U_{t-1}}^{\W_{t-1}}\Eps_{\U_t}^{\W_t}(\T_i)\\
     &\vdots &\qquad &\vdots \\
    \W_1 &=J_{\W_2}(\U_2) &\qquad \U_1 &= \Eps_{\U_{2}}^{\W_{2}}\cdots\Eps_{\U_{t-1}}^{\W_{t-1}}\Eps_{\U_t}^{\W_t}(\T_1). 
\end{alignat*}
Then $$\psi_t(\T_1,\cdots,\T_t) = (\U_1\,\cdots,\U_t).$$
\end{remark}

Let $\W\subseteq \modd{\Lambda}$ be a $\tau$-perpendicular subcategory. For a signed $\tau$-exceptional sequence $(\U_1,\cdots,\U_t)$ in $\W$, we denote by $\overline{\varphi_t}(\U_1,\cdots,\U_t)$ the direct sum of the entries in ${\varphi_t}(\U_1,\cdots,\U_t)$. 

The following result generalizes \cite[Prop. 11.8]{a_category_of_wide_subcategories} from the $\tau$-titling finite case to the case where $\Lambda$ is an arbitrary finite dimensional algebra, and the proof follows a similar argument to that of \cite[Prop. 11.8]{a_category_of_wide_subcategories}.   

\begin{proposition}\label{irr_morph_in_Mcluster}
    Let $\Lambda$ be a finite dimensional algebra and let $\W\subseteq \modd{\Lambda}$ be a $\tau$-perpendicular subcategory. Suppose that ${V} = M\oplus P[1]$ is a support $\tau$-rigid object in $\C(\W)$ with $t$ indecomposable direct summands. Then there is a bijection between
    \begin{enumerate}
        \item[(i)] The set of signed $\tau$-exceptional sequences $(\U_1,\cdots, \U_t)$ in $\W$ such that $\overline{\varphi_t}(\U_1,\cdots,\U_t) = V$;
        \item[(ii)] The set of factorizations of $g_V^\W$ into composition of irreducible maps in $\Mcluster(\Lambda)$.  
    \end{enumerate}
\end{proposition}

For the remainder of this section, let $\Lambda = R\otimes kQ$ unless stated otherwise. Our goal is to establish an equivalence of categories between $\Mcluster(kQ)$ and $\Mcluster(\Lambda)$. The first step in this process will be to formulate an "$\Eps$-version" of Proposition \ref{tau_rigid_in_J(M)_1:1_tau_rigid_in_J(Lambda_otimes_kQM)}. To do that we need the following results. 

\begin{lemma}\label{J(LambdaxU)=J(LambdaxUbar)}
    Let $U = M\oplus P[1]$ be a support $\tau$-rigid object in $\C(kQ)$. Then, $$J(\Lambda\otimes_{kQ}U) = J(\Lambda\otimes_{kQ}\overline{U}),$$ where $\overline{U} = P\oplus \Eps^{-1}_P(M)$.
\end{lemma}

\begin{proof}
    Let $\overline{U} = P\oplus \Eps_P^{-1}(M)$. By Lemma \ref{two_def_of_J}, we have $J(U) = J(\overline{U})$. Since $\Eps_P(\Eps_P^{-1}(M)) = M \notin (\proj J(U))[1]$, Theorem \ref{BM_bijection}(a) implies that $\Eps_P^{-1}(M)\notin\Gen P$, and therefore $\Eps_P^{-1}(M) = f_P^{-1}(M)$ by Theorem \ref{BM_bijection}(c). Similarly, $\Eps_{\Lambda\otimes_{kQ}P}^{-1}(\Lambda\otimes_{kQ}M) = f_{\Lambda\otimes_{kQ}P}^{-1}(\Lambda\otimes_{kQ}M)$. Then 
    \begin{align*}
        J(\Lambda\otimes_{kQ}U) &= J((\Lambda\otimes_{kQ}M)\oplus (\Lambda\otimes_{kQ}P)[1])\\
        &= J((\Lambda\otimes_{kQ}P)\oplus \Eps_{\Lambda\otimes_{kQ}P}^{-1}(\Lambda\otimes_{kQ}M)) &&\text{ by Lemma \ref{two_def_of_J}}\\
        &= J((\Lambda\otimes_{kQ}P)\oplus f_{\Lambda\otimes_{kQ}P}^{-1}(\Lambda\otimes_{kQ}M))\\
        &= J((\Lambda\otimes_{kQ}P)\oplus \Lambda\otimes_{kQ}f_P^{-1}(M)) &&\text{ by Proposition \ref{tau_rigid_in_J(M)_1:1_tau_rigid_in_J(Lambda_otimes_kQM)}}\\
        &= J(\Lambda\otimes_{kQ}(P\oplus f_P^{-1}(M)))\\
        &= J(\Lambda\otimes_{kQ}(P\oplus \Eps_P^{-1}(M)))\\
        &= J(\Lambda\otimes_{kQ}\overline{U}).
    \end{align*}
    This finishes the proof. 
\end{proof}

\begin{proposition}\label{supp_tau_J(U)-rigid1:1supp_tau_J(Lambda_x_U)-rigid}
    Let $U = M\oplus P[1]\in \C(kQ)$ be support $\tau$-rigid. Then there is a bijection
        \[
            \begin{tikzcd}
            {\left\{W\in \ind\C(J(U))\mid  W\text{ support }\tau_{J(U)}\text{-rigid} \right\}} 
            \arrow[d, "\Lambda\otimes_{kQ}-"] \\
            {\left\{ Z\in \ind\C(J(\Lambda\otimes_{kQ}U)) \mid Z\text{ support }\tau_{J(\Lambda\otimes_{kQ}U)}\text{-rigid}\right\}}.
            \end{tikzcd}
        \]
\end{proposition}

\begin{proof}
    We start by showing that the above map is well-defined. So, let $W\in \ind\C(J(U))$ be support $\tau_{J(U)}$-rigid. We distinguish two cases. First, assume that $W$ is an indecomposable $\tau$-rigid module in $J(U)$. By Lemma \ref{two_def_of_J} and Lemma \ref{J(LambdaxU)=J(LambdaxUbar)}, $J(U)=J(\overline{U})$ and $J(\Lambda\otimes_{kQ}U) = J(\Lambda\otimes_{kQ}\overline{U})$, respectively, where $\overline{U} = P\oplus \Eps_P^{-1}(M)$.
    Hence, since $W$ is $\tau$-rigid in $J(\overline{U})$, it follows that  $\Lambda\otimes_{kQ}W$ is $\tau$-rigid in $J(\Lambda\otimes_{kQ}\overline{U})$ by Lemma \ref{tau-rigid_in_J(X)_iff_taurigid_in_J(Lambda_otimes_X)} (i), (iii).

    Now suppose $W\in (\ind \proj J(U))[1]$, i.e. $W=Q[1]$ for $Q\in \ind \proj J(U)$. By Proposition \ref{indprojJ(M)1-1indprojJ(Lambda_otimes_M)}, the induction functor induces a bijection between $\ind\proj J(\overline{U})$ and $\ind\proj J(\Lambda\otimes_{kQ}\overline{U})$. Since $J(U) = J(\overline{U}) $ and $J(\Lambda\otimes_{kQ}U)=J(\Lambda\otimes_{kQ}\overline{U})$, we conclude that $(\Lambda\otimes_{kQ}Q)[1]\in (\ind\proj J(\Lambda\otimes_{kQ}U))[1]$. This proves that the above map is well-defined. 

    Notice that 
    \begin{align*}
        {\left\{W\in \ind\C(J(U))\mid W \text{ supp. }\tau_{J(U)}\text{-rigid} \right\}} =\\{\left\{X\in \ind(J(\overline{U}))\mid X \text{ is }\tau_{J(\overline{U})}\text{-rigid} \right\}} 
        &\cup (\ind \proj J(\overline{U}))[1],
    \end{align*}
    and similarly 
    \begin{align*}
        {\left\{Z\in \ind\C(J(\Lambda\otimes_{kQ}U))\mid Z \text{ supp. }\tau_{J(\Lambda\otimes_{kQ}U)}\text{-rigid} \right\}} \\=  {\left\{Y\in \ind(J(\Lambda\otimes_{kQ}\overline{U}))\mid Y \text{ is }\tau_{J(\Lambda\otimes_{kQ}\overline{U})}\text{-rigid} \right\}} &\cup (\ind \proj J(\Lambda\otimes_{kQ}\overline{U}))[1].
    \end{align*}
    Hence, $\Lambda\otimes_{kQ}-$ induces the desired bijection by Proposition\ref{tau_rigid_in_J(M)_1:1_tau_rigid_in_J(Lambda_otimes_kQM)} and Proposition \ref{indprojJ(M)1-1indprojJ(Lambda_otimes_M)}. 
\end{proof}

Let $\Lambda$ be a finite dimensional algebra and let $U$ be a $\tau$-rigid $\Lambda$-module. Recall from \cite[Section 2.3]{tau-tiling-theory} that, up to isomorphism, there is a unique basic module $C_U$ and a basic projective module $Q$ such that $C_U\oplus U \oplus Q[1]$ is a support $\tau$-tilting object and $\add(C_U\oplus U) = \P(\Gen U)$. In particular, $\add Q = \proj kQ \cap {^\perp U}$. We refer to $C_U\oplus Q[1]$ as the \textit{co-Bongartz complement} of $U$. 

A key ingredient towards the proof of an "$\Eps$-version" of Proposition \ref{tau_rigid_in_J(M)_1:1_tau_rigid_in_J(Lambda_otimes_kQM)} is to establish a connection between the co-Bongartz complement of a $\tau$-rigid $kQ$-module $U$ and the co-Bongartz complement of the corresponding $\tau$-rigid $\Lambda$-module $\Lambda\otimes_{kQ}U$. We need the following lemma. 

\begin{lemma}\label{Ext-proj_in_Gen(LambdaxU)}
   Let $U$ be a $\tau$-rigid $kQ$-module. Then $\Lambda\otimes_{kQ}X\in \P(\Gen(\Lambda\otimes_{kQ}U))$ if and only if $X\in\P(\Gen U)$.
\end{lemma}

\begin{proof}
    Observe that $\Lambda\otimes_{kQ}X\in \Gen(\Lambda\otimes_{kQ}U)$ if and only if $X\in \Gen U$. 
    
    Let $X\in \P(\Gen U)$, i.e. $\Ext_{kQ}^1(X,\Gen U) = 0$. Suppose $\Lambda\otimes_{kQ}X$ was not Ext-projective in $\P(\Gen(\Lambda\otimes_{kQ}U))$, that is there exists $Y\in\Gen(\Lambda\otimes_{kQ}U)$ such that $\Ext_{\Lambda}^1(\Lambda\otimes_{kQ}X,Y)\neq 0$. Then there exists a non-split short exact sequence in $\modd\Lambda$ of the form $0\to Y\to E\to \Lambda\otimes_{kQ}X\to 0$. Applying the restriction of scalars we get a non-split short exact sequence 
    $$0\to \res Y\to \res E\to \res(\Lambda\otimes_{kQ}X) = X^d\to 0$$
    in $\modd kQ$ with $\res Y\in \Gen U$ (this follows from the fact that $Y\in \Gen(\Lambda\otimes_{kQ}U)$ and $\res$ is exact). Hence, $\Ext^1_{kQ}(X, \res Y) \neq 0$, a contradiction. 

    Conversely, assume $\Lambda\otimes_{kQ}X\in\P(\Gen(\Lambda\otimes_{kQ}U))$ and suppose there exists $V\in\Gen U$ such that $\Ext_{kQ}^1(X,V)\neq 0$. It follows there exists a non-split short exact sequence in $\modd kQ$ of the form $0\to V\to E'\to X\to 0$. Applying the induction functor we obtain a non-split short exact sequence $$0\to \Lambda\otimes_{kQ}V\to \Lambda\otimes_{kQ}E'\to \Lambda\otimes_{kQ}X\to 0$$ with $\Lambda\otimes_{kQ}V\in \Gen(\Lambda\otimes_{kQ}U)$. Hence, $\Ext_\Lambda^1(\Lambda\otimes_{kQ}X,\Lambda\otimes_{kQ}V)\neq 0$, a contradiction. This concludes the proof.   
\end{proof}

\begin{proposition}\label{co-Bongartz_complement}
    Let $M$ be a $\tau$-rigid $kQ$-module and let $\Lambda\otimes_{kQ}M$ be the corresponding $\tau$-rigid $\Lambda$-module. Let $C_{\Lambda\otimes_{kQ}M}\oplus Q_{\Lambda\otimes_{kQ}M}[1]$ be the co-Bongartz complement of $\Lambda\otimes_{kQ}M$. Then, 
    $$C_{\Lambda\otimes_{kQ}M} \oplus Q_{\Lambda\otimes_{kQ}M}[1] \cong \Lambda\otimes_{kQ}(C_M\oplus Q_M[1])$$
    where $C_M\oplus Q_M[1]$ is the co-Bongartz complement of $M$.
\end{proposition}

\begin{proof}
    Let $C_{\Lambda\otimes_{kQ}M}\oplus Q_{\Lambda\otimes_{kQ}M}[1]$ be the co-Bongartz complement of $\Lambda\otimes_{kQ}M$. Then $C_{\Lambda\otimes_{kQ}M}\oplus \Lambda\otimes_{kQ}M\oplus Q_{\Lambda\otimes_{kQ}M}[1]$ is the unique (up to isomorphism) support $\tau$-tilting object in $\C(\Lambda)$ such that $\add(C_{\Lambda\otimes_{kQ}M}\oplus \Lambda\otimes_{kQ}M) = \P(\Gen(\Lambda\otimes_{kQ}M))$ and $\add(Q_{\Lambda\otimes_{kQ}M}) = \proj\Lambda \cap {^\perp(\Lambda\otimes_{kQ}M)}$. In particular, by Corollary \ref{stautitl(kQ)1:1stautilt(Lambda)} $C_{\Lambda\otimes_{kQ}M}\cong \Lambda\otimes_{kQ}C$ for a $\tau$-rigid $kQ$-module $C$ and $Q_{\Lambda\otimes_{kQ}M}\cong \Lambda\otimes_{kQ}Q$ for a projective $kQ$-module $Q$. Let $C_M\oplus Q_M[1]$ be the co-Bongartz complement of $M$. We want to show that $C\cong C_M$ and $Q\cong Q_M$. We have that $\add Q_M = \proj kQ \cap {^\perp M}$. Moreover, $\add(\Lambda\otimes_{kQ}Q) = \Lambda\otimes_{kQ}(\add Q)$. Since $\proj\Lambda = \Lambda\otimes_{kQ}(\proj kQ)$ and $\Lambda\otimes_{kQ}X\in {^\perp(\Lambda\otimes_{kQ}M)}$ if and only if $X\in {^\perp M}$ we get that 
    \begin{align*}
        \Lambda\otimes_{kQ}\add Q &= \add(\Lambda\otimes_{kQ}Q)\\ 
                                    &= \proj\Lambda \cap {^\perp(\Lambda\otimes_{kQ}M)}\\
                                    &= (\Lambda\otimes_{kQ}\proj kQ) \cap (\Lambda\otimes_{kQ}{^\perp M}).  
    \end{align*}
    Hence,
    $$\Lambda\otimes_{kQ}\add Q = (\Lambda\otimes_{kQ}\proj kQ) \cap (\Lambda\otimes_{kQ}{^\perp M}) = \Lambda\otimes_{kQ}(\proj kQ \cap {^\perp M}) = \Lambda\otimes_{kQ}\add Q_M.$$
    Thus $\add Q = \add Q_M$, and therefore $Q\cong Q_M$.

    It remains to show that $C\cong C_M$. By the definition of the co-Bongartz complement of $\Lambda\otimes_{kQ}M$, we have 
    \begin{align*}
        \P(\Gen(\Lambda\otimes_{kQ}M)) &= \add(C_{\Lambda\otimes_{kQ}M}\oplus(\Lambda\otimes_{kQ}M))\\
                                       &= \add(\Lambda\otimes_{kQ}(C\oplus M))\\
                                       &= \Lambda\otimes_{kQ}\add(C\oplus M).
    \end{align*} 
    Thus, every module in $\P(\Gen(\Lambda\otimes_{kQ}M))$ is an induced module. Hence, by Lemma \ref{Ext-proj_in_Gen(LambdaxU)} $\P(\Gen(\Lambda\otimes_{kQ}M)) = \Lambda\otimes_{kQ}\P(\Gen M)$, and therefore $\P(\Gen M) = \add(C\oplus M)$. Observe that, since $\Lambda\otimes_{kQ}(C\oplus M)\oplus (\Lambda\otimes_{kQ}Q_M)[1]$ is support $\tau$-tilting in $\C(\Lambda)$, it follows that $C\oplus M \oplus Q_M[1]$ is support $\tau$-tilting in $\C(kQ)$ by Corollary \ref{stautitl(kQ)1:1stautilt(Lambda)}. But $C_M$ is the unique (up to isomorphism) basic $\tau$-rigid $kQ$-module such that $C_M\oplus M \oplus Q_M[1]$ is support $\tau$-tilting in $\C(kQ)$ and $\P(\Gen M) = \add(C_M\oplus M)$. We conclude that $C\cong C_M$ and the claim follows. 
\end{proof}

Let $\Lambda$ be a finite dimensional algebra. Let $M\in \modd\Lambda$ be $\tau$-rigid and let $V\in \ind\C(\Lambda)$ be such that $M\oplus V$ is support $\tau$-rigid in $\C(\Lambda)$. If $V\in \Gen M$ or $V\in(\proj\Lambda)[1]$, Buan and Marsh \cite[Section 3]{tauExcSeq_BM} constructed an indecomposable direct summand $B_M^V$ of the Bongartz complement of $M$ as follows. 
\begin{enumerate}
    \item[(a)] If $V\in \Gen M$, take a minimal right $\add \PP_M$-approximation $\alpha: \PP_{M'}\to \PP_V$ in $K^b(\proj \Lambda)$ and complete to a triangle $\PP_{B_M^V}\to \PP_{M'}\xrightarrow{\alpha} \PP_V\to .$ Then, $B_M^V = H^0\Big(\PP_{B_M^V}\Big)$.
    \item[(b)] If $V = Q[1]$, let $a_Q: Q\to B_Q$ be a minimal left $\P({^\perp\tau}M)$-approximation. Then $B_M^V = B_Q$. 
\end{enumerate}

In particular, this assignment defines a bijection between the indecomposable direct summands of the co-Bongartz complement $C_M\oplus Q[1]$ of $M$ and the indecomposable direct summands of the Bongartz complement $B_M$ of $M$ \cite[Prop. 3.7, Def. 4.10, Rmk. 4.11]{tauExcSeq_BM}.  

\begin{lemma}\label{direct_summand_of_co-Bongartz_commutes_with_induction}
    Let $M$ be a $\tau$-rigid $kQ$-module. Let $C_M\oplus Q_M[1]$ and $C_{\Lambda\otimes_{kQ}M}\oplus Q_{\Lambda\otimes_{kQ}M}[1]$ be the co-Bongartz complements of $M$ and $\Lambda\otimes_{kQ}M$, respectively. Let $V$ be an indecomposable direct summand of $C_M\oplus Q_M[1]$. Then, with the notation above, we have $\Lambda\otimes_{kQ}B_M^V\cong B_{\Lambda\otimes_{kQ}M}^{\Lambda\otimes_{kQ}V}$.
\end{lemma}

\begin{proof}
    By Proposition \ref{co-Bongartz_complement}, $V$ is an indecomposable direct summand of $C_M$ (respectively, $Q_M[1]$) if and only if $\Lambda\otimes_{kQ}V$ is an indecomposable direct summand of $C_{\Lambda\otimes_{kQ}M}$ (respectively, $Q_{\Lambda\otimes_{kQ}M}[1]$). We distinguish two cases. 

    First assume $\Lambda\otimes_{kQ}V\in \Gen(\Lambda\otimes_{kQ}M)$ and let $\alpha': \PP_{M'}\to \PP_V$ be a right minimal $\add\PP_M$-approximation in $K^b(\proj kQ)$. Then, by Lemma \ref{induction_preserves_approximations}(ii), $\alpha := \Lambda\otimes_{kQ}\alpha': \Lambda\otimes_{kQ}\PP_{M'}\to \Lambda\otimes_{kQ}\PP_V$ is a right minimal $(\Lambda\otimes_{kQ}\add\PP_M) = \add(\Lambda\otimes_{kQ}\PP_M)$-approximation in $K^b(\proj\Lambda)$. By Lemma \ref{2-term_and_iduction}(i) this is equivalent to $\alpha: \PP_{\Lambda\otimes_{kQ}M'}\to \PP_{\Lambda\otimes_{kQ}V}$ being a right minimal $\add\PP_{\Lambda\otimes_{kQ}M}$-approximation in $K^b(\proj\Lambda)$. Complete $\alpha$ to a triangle 
    $$\PP_{B_{\Lambda\otimes_{kQ}M}^{\Lambda\otimes_{kQ}V}}\to \PP_{\Lambda\otimes_{kQ}M'}\xrightarrow{\alpha}\PP_{\Lambda\otimes_{kQ}V}\to .$$
    Then, $$\PP_{B_{\Lambda\otimes_{kQ}M}^{\Lambda\otimes_{kQ}V}} = \cone(\alpha)[-1] = \cone(\Lambda\otimes_{kQ}\alpha')[-1] = (\Lambda\otimes_{kQ}\cone(\alpha'))[-1].$$
    Hence, we obtain
    \begin{align*}
        B_{\Lambda\otimes_{kQ}M}^{\Lambda\otimes_{kQ}V} &= H^0(\cone(\alpha)[-1]) &&\text{by \cite[Prop. 3.7(a)]{tauExcSeq_BM}}\\
        &= H^0(\Lambda\otimes_{kQ}\cone(\alpha')[-1])\\
        &= \Lambda\otimes_{kQ}H^0(\cone(\alpha')[-1]) && \text{$H^0$ commutes with exact functors}\\ 
        &= \Lambda\otimes_{kQ}H^0\Big(\PP_{B_{M}^{V}}\Big)\\
        &= \Lambda\otimes_{kQ}B_M^V &&\text{by \cite[Prop. 3.7(a)]{tauExcSeq_BM}}.
    \end{align*}   
    This proves the first case. 

    Now suppose $\Lambda\otimes_{kQ}V = (\Lambda\otimes_{kQ}Q)[1]$, where $Q$ is an indecomposable direct summand of $Q_M$. Let $a_{\Lambda\otimes_{kQ}Q}: \Lambda\otimes_{kQ}Q\to B_{\Lambda\otimes_{kQ}Q}$ be a minimal left $\P({^\perp\tau}(\Lambda\otimes_{kQ}M))$-approximation. Then,  $B_{\Lambda\otimes_{kQ}M}^{\Lambda\otimes_{kQ}V}\cong B_{\Lambda\otimes_{kQ}Q}$ by \cite[Def. 4.10(ii), Rmk. 4.11]{tauExcSeq_BM}. Notice that Corollary \ref{Bongartz_commutes_with_induction} implies $\P({^\perp\tau}(\Lambda\otimes_{kQ}M)) = \Lambda\otimes_{kQ}\P({^\perp\tau}M)$. Let $a_Q : Q\to B_Q$ be a minimal left $\P({^\perp\tau M})$-approximation. Using \cite[Def. 4.10(ii), Rmk 4.11]{tauExcSeq_BM} again, we have that $B_M^V\cong B_Q$. Because $a_Q : Q\to B_Q$ is a minimal left $\P({^\perp\tau M})$-approximation, it follows that $\Lambda\otimes_{kQ}a_Q: \Lambda\otimes_{kQ}Q\to\Lambda\otimes_{kQ}B_M^V$ is a minimal left $\Lambda\otimes_{kQ}\P({^\perp\tau M})=\P({^\perp\tau}(\Lambda\otimes_{kQ}M))$-approximation by the dual of Lemma \ref{induction_preserves_approximations}(i). By the minimality of $a_{\Lambda\otimes_{kQ}Q}$ we conclude $\Lambda\otimes_{kQ}B_M^Q\cong B_{\Lambda\otimes_{kQ}M}^{\Lambda\otimes_{kQ}Q}$. This proves the second case and finishes the proof. 
    \end{proof}

We are now ready to state and prove an "$\Eps$-version" of Proposition \ref{tau_rigid_in_J(M)_1:1_tau_rigid_in_J(Lambda_otimes_kQM)}. 

\begin{proposition}\label{E-version3.7}
    Let $U = M\oplus P[1]\in \C(kQ)$ be support $\tau$-rigid and let $V\in \ind \C(kQ)$ be such that $U\oplus V$ is also support $\tau$-rigid in $\C(kQ)$. Then, $\Lambda\otimes_{kQ}(U\oplus V)$ is support $\tau$-rigid in $\C(\Lambda)$ and there is an isomorphism
    \begin{equation}\label{E-map}
         \Eps_{\Lambda\otimes_{kQ}U}(\Lambda\otimes_{kQ}V)\cong \Lambda\otimes_{kQ}\Eps_U(V)
    \end{equation}
    in $\C(\Lambda)$. In other words, there is a commutative diagram of bijections 
    \[
\begin{tikzcd}[ampersand replacement=\&, cramped, row sep=large, column sep=scriptsize, 
every matrix/.append style={font=\small}]
{\left\{ V \in\ind\C(kQ)\;\middle|\; \begin{tabular}{@{}l@{}} $V\oplus U$ is support \\ $\tau_{kQ}$-rigid \end{tabular} \right\}} \&\& {\left\{W\in \ind \C(J(U))\;\middle|\; \begin{tabular}{@{}l@{}} $W$ is support \\$\tau_{J(U)}$-rigid \end{tabular} \right\}} \\
{\left\{ Y \in \ind\C(\Lambda) \;\middle|\; \begin{tabular}{@{}l@{}} $Y\oplus(\Lambda\otimes_{kQ}U)$ is \\ support $\tau_{\Lambda}$-rigid \end{tabular} \right\}} \&\& {\left\{Z\in \ind \C(J(\Lambda\otimes_{kQ}U))\;\middle|\; \begin{tabular}{@{}l@{}} $Z$ is support \\$\tau_{J(\Lambda\otimes_{kQ}U)}$-rigid \end{tabular} \right\}}.
\arrow["{\mathcal{E}_U}", from=1-1, to=1-3]
\arrow["{\Lambda\otimes_{kQ}-}"', from=1-1, to=2-1]
\arrow["{\Lambda\otimes_{kQ}-}", from=1-3, to=2-3]
\arrow["{\mathcal{E}_{\Lambda\otimes_{kQ}U}}", from=2-1, to=2-3]
\end{tikzcd}
\]
\end{proposition}

\begin{proof}
    Let $U = M\oplus P[1]\in \C(kQ)$ be support $\tau$-rigid and let $V\in \ind \C(kQ)$ be such that $U\oplus V$ is also support $\tau$-rigid in $\C(kQ)$. Then, $\Lambda\otimes_{kQ}(U\oplus V)$ is support $\tau$-rigid in $\C(\Lambda)$ by Corollary \ref{stautitl(kQ)1:1stautilt(Lambda)}. We distinguish two cases.
    
   First, suppose $V\in \ind\modd kQ$ and $V\notin \Gen M$. Then $\Lambda\otimes_{kQ}V\in \ind\modd kQ$ and $\Lambda\otimes_{kQ}V\notin \Gen(\Lambda\otimes_{kQ}M)$. Hence, combining Proposition \ref{tau_rigid_in_J(M)_1:1_tau_rigid_in_J(Lambda_otimes_kQM)} and Theorem \ref{BM_bijection}(c), we obtain 
        \begin{align*}
            \Eps_{\Lambda\otimes_{kQ}U}(\Lambda\otimes_{kQ}V) &= \Eps_{\Lambda\otimes_{kQ}M}(\Lambda\otimes_{kQ}V)\\
            &= f_{\Lambda\otimes_{kQ}M}(\Lambda\otimes_{kQ}V)\\
            &\cong \Lambda\otimes_{kQ}f_M V\\
            &= \Lambda\otimes_{kQ}\Eps_M(V)\\
            &= \Lambda\otimes_{kQ}\Eps_U(V). 
        \end{align*}
    Now, consider the case where $V\in \Gen M$ or $V\in (\proj kQ)[1]$. Then $\Lambda\otimes_{kQ}V\in \Gen(\Lambda\otimes_{kQ}M)$ or $\Lambda\otimes_{kQ}V\in (\proj\Lambda)[1]$. 
    
    By \cite[Prop. 5.6]{tauExcSeq_BM} (see also \cite[Thm. 5.1(1)(b)-(c)]{tau-perpendicular_wide_subategories}) there exists an indecomposable direct summand 
    $B_{M}^{V}$ of the Bongartz complement of $M$ in $J(P)$ such that
    \begin{equation}\label{EpsEquation1}
        \Eps_{M}^{J(P)}(V) = f_{M}^{J(P)}\Big(B_{M}^{V}\Big)[1],
    \end{equation}
    and similarly, there exists an indecomposable direct summand
    $B_{\Lambda\otimes_{kQ}M}^{\Lambda\otimes_{kQ}V}$ of the Bongartz complement of $\Lambda\otimes_{kQ}M$ in $J(\Lambda\otimes_{kQ}P)$ such that
    \begin{equation}\label{EpsEquation2}
        \Eps_{\Lambda\otimes_{kQ}M}^{J(\Lambda\otimes_{kQ}P)}(\Lambda\otimes_{kQ}V) = f_{\Lambda\otimes_{kQ}M}^{J(\Lambda\otimes_{kQ}P)}\Big(B_{\Lambda\otimes_{kQ}M}^{\Lambda\otimes_{kQ}V}\Big)[1],
    \end{equation}
    where $f_M^{J(P)}$ and $f_{\Lambda\otimes_{kQ}M}^{J(\Lambda\otimes_{kQ}P)}$ denote the torsion-free functor in $J(P)$ and $J(\Lambda\otimes_{kQ}P)$, respectively.
    
    Consider the case where $V\in\Gen M$ and therefore $\Lambda\otimes_{kQ}V\in\Gen(\Lambda\otimes_{kQ}M)$. Then, 
    \vfill
    \begin{align*}
        \Eps_{\Lambda\otimes_{kQ}U}(\Lambda\otimes_{kQ}V) &= \Eps_{(\Lambda\otimes_{kQ}P)[1]\oplus(\Lambda\otimes_{kQ}M)}(\Lambda\otimes_{kQ}V)\\[8pt]
        &= \Eps_{\Eps_{(\Lambda\otimes_{kQ}P)[1]}(\Lambda\otimes_{kQ}M)}^{J((\Lambda\otimes_{kQ}P)[1])}\Big (\Eps_{(\Lambda\otimes_{kQ}P)[1]}(\Lambda\otimes_{kQ}V)\Big ) &&\text{by Theorem \ref{Eps_{U+V}}}\\[8pt]
        &= \Eps_{\Lambda\otimes_{kQ}M}^{J(\Lambda\otimes_{kQ}P)}(\Lambda\otimes_{kQ}V) &&\text{by Theorem \ref{BM_bijection}(b)}\\[8pt]
        &= f_{\Lambda\otimes_{kQ}M}^{J(\Lambda\otimes_{kQ}P)}\Big(B_{\Lambda\otimes_{kQ}M}^{\Lambda\otimes_{kQ}V}\Big)[1] &&\text{by Equation \eqref{EpsEquation2}}\\[8pt]
        &\cong f_{\Lambda\otimes_{kQ}M}^{J(\Lambda\otimes_{kQ}P)}\Big(\Lambda\otimes_{kQ}B_M^{V}\Big)[1] &&\text{by Lemma \ref{direct_summand_of_co-Bongartz_commutes_with_induction}}\\[8pt]
        &\cong \Lambda\otimes_{kQ} f_M^{J(P)}(B_M^{V})[1] &&\text{by Prop. \ref{tau_rigid_in_J(M)_1:1_tau_rigid_in_J(Lambda_otimes_kQM)} and Cor. \ref{transport_of_structure}}\\[8pt]
        &= \Lambda\otimes_{kQ}\Eps_M^{J(P[1])}(V) &&\text{by Equation \eqref{EpsEquation1}}\\[8pt]
        &= \Lambda\otimes_{kQ}\Eps_{\Eps_{P[1]}(M)}^{J(P[1])}\Big(\Eps_{P[1]}(V)\Big) && \text{by Theorem \ref{BM_bijection}(b)}\\[8pt]
        &= \Lambda\otimes_{kQ}\Eps_{P[1]\oplus M}(V) &&\text{by Theorem \ref{Eps_{U+V}}}\\[8pt]
        &= \Lambda\otimes_{kQ}\Eps_{U}(V).
    \end{align*}
    \vspace{\fill}  
    
    Now consider the case where $V=Q[1]\in (\proj kQ)[1]$ and thus $\Lambda\otimes_{kQ}V = (\Lambda\otimes_{kQ}Q)[1]\in (\proj\Lambda)[1]$. 
    Then, 
    \begin{align*}
        \Eps_{\Lambda\otimes_{kQ}U}(\Lambda\otimes_{kQ}V) &= \Eps_{(\Lambda\otimes_{kQ}P)[1]\oplus(\Lambda\otimes_{kQ}M)}(\Lambda\otimes_{kQ}V)\\[8pt]
        &= \Eps_{\Eps_{(\Lambda\otimes_{kQ}P)[1]}(\Lambda\otimes_{kQ}M)}^{J((\Lambda\otimes_{kQ}P)[1])}\Big (\Eps_{(\Lambda\otimes_{kQ}P)[1]}\big((\Lambda\otimes_{kQ}Q)[1]\big)\Big ) &&\text{by Theorem \ref{Eps_{U+V}}}\\[8pt]
        &=  \Eps_{\Lambda\otimes_{kQ}M}^{J(\Lambda\otimes_{kQ}P)}\Big(f_{\Lambda\otimes_{kQ}P}(\Lambda\otimes_{kQ}Q)[1]\Big) &&\text{by \cite[Thm. 5.10a]{tauExcSeq_BM} and Thm. \ref{BM_bijection}(b)}\\[8pt]
        &\cong \Eps_{\Lambda\otimes_{kQ}M}^{J(\Lambda\otimes_{kQ}P)}\Big(\Lambda\otimes_{kQ}f_P(Q)[1]\Big) &&\text{ by Proposition \ref{tau_rigid_in_J(M)_1:1_tau_rigid_in_J(Lambda_otimes_kQM)}}\\[8pt]
        &= f_{\Lambda\otimes_{kQ}M}^{J(\Lambda\otimes_{kQ}P)}\Big(B_{\Lambda\otimes_{kQ}M}^{\Lambda\otimes_{kQ}f_PQ}\Big)[1] && \text{ by Equation \eqref{EpsEquation2}}\\[8pt]
        &\cong f_{\Lambda\otimes_{kQ}M}^{J(\Lambda\otimes_{kQ}P)}\Big(\Lambda\otimes_{kQ}B_M^{f_PQ}\Big)[1] &&\text{ by Lemma \ref{direct_summand_of_co-Bongartz_commutes_with_induction}}\\[8pt]
        &\cong \Lambda\otimes_{kQ} f_M^{J(P)}\Big(B_M^{f_PQ}\Big)[1] &&\text{ by Prop. \ref{tau_rigid_in_J(M)_1:1_tau_rigid_in_J(Lambda_otimes_kQM)} and Cor. \ref{transport_of_structure}}\\[8pt]
        &= \Lambda\otimes_{kQ}\Eps_{M}^{J(P[1])}\Big(f_P(Q)[1]\Big) &&\text{ by Equation \eqref{EpsEquation1}}\\[8pt]
        &= \Lambda\otimes_{kQ}\Eps_M^{J(P[1])}\Big(\Eps_{P[1]}(Q[1])\Big) &&\text{ by \cite[Prop. 5.10a]{tauExcSeq_BM}}\\[8pt]
        &= \Lambda\otimes_{kQ}\Eps_{\Eps_{P[1]}(M)}^{J(P[1])}\Big(\Eps_{P[1]}(V)\Big) &&\text{ by Theorem \ref{BM_bijection}(b)}\\[8pt]
        &= \Lambda\otimes_{kQ}\Eps_{P[1]\oplus M}(V) &&\text{ by Theorem \ref{Eps_{U+V}}}\\[8pt]
        &= \Lambda\otimes_{kQ}\Eps_U(V).
    \end{align*}
    Note that the map
    $${\left\{ V \in\ind\C(kQ)\;\middle|\; \begin{tabular}{@{}l@{}} $V\oplus U$ is support \\ $\tau_{kQ}$-rigid \end{tabular} \right\}}\xrightarrow{\Lambda\otimes_{kQ}-} {\left\{ Y \in \ind\C(\Lambda) \;\middle|\; \begin{tabular}{@{}l@{}} $Y\oplus(\Lambda\otimes_{kQ}U)$ is \\ support $\tau_{\Lambda}$-rigid \end{tabular} \right\}}$$
    is well-defined by Corollary \ref{stautitl(kQ)1:1stautilt(Lambda)}. Moreover, the isomorphism \eqref{E-map} ensures that the diagram above commutes. Hence, since both the horizontal maps and the right vertical map are bijections by Theorem \ref{BM_bijection} and Proposition \ref{supp_tau_J(U)-rigid1:1supp_tau_J(Lambda_x_U)-rigid}, respectively, then so is the above map. This concludes the proof. 
\end{proof}

The induction functor is compatible with the bijection $\varphi_t$ in the following way. 

\begin{proposition}\label{induction_commutes_with_varphi}
    Let $t\in \{1,\cdots,n\}$. Then, there is a commutative diagram of bijections
\[
\begin{tikzcd}[ampersand replacement=\&, cramped, row sep=large, column sep=3em, 
every matrix/.append style={font=\small}]
{\left\{ \begin{tabular}{@{}c@{}} 
    signed $\tau$-exceptional sequences \\ 
    in $\C(kQ)$ of length $t$  
\end{tabular} \right\}} 
\&\& 
{\left\{ \begin{tabular}{@{}c@{}} 
    support $\tau$-rigid objects in \\ 
    $\C(kQ)$ of length $t$ 
\end{tabular} \right\}} 
\\
{\left\{ \begin{tabular}{@{}c@{}} 
    signed $\tau$-exceptional sequences \\ 
    in $\C(\Lambda)$ of length $t$  
\end{tabular} \right\}} 
\&\& 
{\left\{ \begin{tabular}{@{}c@{}} 
    support $\tau$-rigid objects in \\ 
    $\C(\Lambda)$ of length $t$ 
\end{tabular} \right\}}. 
\arrow["{\varphi_t}", shift left, from=1-1, to=1-3]
	\arrow["{\Lambda\otimes_{kQ}-}"', from=1-1, to=2-1]
	\arrow["{\psi_t}", shift left, from=1-3, to=1-1]
	\arrow["{\Lambda\otimes_{kQ}-}", from=1-3, to=2-3]
	\arrow["{\varphi_t}", shift left, from=2-1, to=2-3]
	\arrow["{\psi_t}", shift left, from=2-3, to=2-1]
\end{tikzcd}
\]

\end{proposition}

\begin{proof}
    The horizontal and vertical maps are bijections by Theorem \ref{main_BM}, Remark \ref{BM_inverse_bijection}, and Corollary \ref{bijection_signed_tau_exceptional_sequences}, respectively. Next, we show that the diagram above commutes using the bijection $\psi_t$. So let $(\T_1,\cdots, \T_t)$ be an ordered support $\tau$-rigid object in $\C(kQ)$. We want to compute $\Lambda\otimes_{kQ}(\psi_t(\T_1,\cdots,\T_t))$. By definition, $\psi_t(\T_1,\cdots,\T_t) = (\U_1,\cdots,\U_t)$, where $(\U_1,\cdots,\U_t)$ is as in Remark \ref{BM_inverse_bijection}. Let
    \begin{alignat*}{2}
    \W_t      &= \modd kQ              &\qquad \widetilde{\W_t}      &= \modd \Lambda\\
    \W_{t-1}  &= J_{W_t}(\U_t)   &\qquad  \widetilde{W_{t-1}}  &= J_{\widetilde{W_t}}(\Lambda\otimes_{kQ}\U_t) \\
    \W_{t-2}  &= J_{W_{t-1}}(\U_{t-1})   &\qquad  \widetilde{W_{t-2}}  &= J_{\widetilde{W_{t-1}}}(\Lambda\otimes_{kQ}\U_{t-1}) \\
    &\vdots &\vdots \\
    \W_{1}  &= J_{W_{2}}(\U_{2})   &\qquad  \widetilde{W_{1}}  &= J_{\widetilde{W_{2}}}(\Lambda\otimes_{kQ}\U_{2}) \\
\end{alignat*}

Then, using Proposition \ref{E-version3.7} and Corollary \ref{transport_of_structure}
\begin{alignat*}{2}
    \Lambda\otimes_{kQ}\U_t      &= \Lambda\otimes_{kQ} \T_t \\[8pt]
    \Lambda\otimes_{kQ}\U_{t-1}  &= \Lambda\otimes_{kQ}\Eps_{\U_t}^{\W_t}(\T_{t-1})\\ &= \Eps_{\Lambda\otimes_{kQ}\U_t}^{\widetilde{\W_t}}(\Lambda\otimes_{kQ}\T_{t-1})  \\[8pt]
    \Lambda\otimes_{kQ}\U_{t-2}  &= \Lambda\otimes_{kQ}\Eps_{\U_{t-1}}^{\W_{t-1}}\Eps_{\U_{t}}^{\W_t}(\T_{t-2})\\ &= \Eps_{\Lambda\otimes_{kQ}\U_{t-1}}^{\widetilde{\W_{t-1}}}\left(\Lambda\otimes_{kQ}\Eps_{\U_{t}}^{\W_t}(\T_{t-2})\right)  \\ &=\Eps_{\Lambda\otimes_{kQ}\U_{t-1}}^{\widetilde{\W_{t-1}}}\left( \Eps_{\Lambda\otimes_{kQ}\U_{t}}^{\widetilde{\W_{t}}}(\Lambda\otimes_{kQ}\T_{t-2})\right) \\
    &\vdots\\
    \Lambda\otimes_{kQ}\U_1 &=  \Eps_{\Lambda\otimes_{kQ}\U_{2}}^{\widetilde{\W_{2}}}\cdots\Eps_{\Lambda\otimes_{kQ}\U_{t-1}}^{\widetilde{\W_{t-1}}}\Eps_{\Lambda\otimes_{kQ}\U_t}^{\widetilde{\W_t}}(\Lambda\otimes_{kQ}\T_1).
\end{alignat*}
In other words, $\Lambda\otimes_{kQ}(\psi_t(\T_1,\cdots,\T_t)) = \psi_t(\Lambda\otimes_{kQ}(\T_1,\cdots,\T_t))$. The claim follows. 
\end{proof}

\begin{proposition}\label{functor_between_tau_cluster_morphism_categories}
    There is a functor $F: \Mcluster(kQ) \to \Mcluster(\Lambda)$ given by $$J(U)\mapsto F(J(U)):=J(\Lambda\otimes_{kQ}U)$$ on objects, and $$g_V^{J(U)}\mapsto F(g_V^{J(U)}) := g_{\Lambda\otimes_{kQ}V}^{J(\Lambda\otimes_{kQ}U)}: J(\Lambda\otimes_{kQ}U)\to J_{J(\Lambda\otimes_{kQ}U)}(\Lambda\otimes_{kQ}V)$$ 
    for a morphism $g_V^{J(U)}: J(U)\to J_{J(U)}(V)$. 
\end{proposition}

\begin{proof}
     By Corollary \ref{stautitl(kQ)1:1stautilt(Lambda)}, $F$ is well-defined on objects. Let $g_V^{J(U)}: J(U)\to J_{J(U)}(V)$ be a morphism in $\Mcluster(kQ)$. By Proposition \ref{E-version3.7} and Theorem \ref{Buan-Hanson-Theorem-General}, we get that 
    \begin{align*}
        F(J_{J(U)}(V)) &= F(J(U\oplus \Eps_U^{-1}(V))\\
        &= J(\Lambda\otimes_{kQ}(U\oplus \Eps_U^{-1}(V)))\\
        &= J((\Lambda\otimes_{kQ}U)\oplus (\Lambda\otimes_{kQ}\Eps_U^{-1}(V))\\
        &= J((\Lambda\otimes_{kQ}U)\oplus \Eps^{-1}_{\Lambda\otimes_{kQ}U}(\Lambda\otimes_{kQ}V))\\
        &= J_{J(\Lambda\otimes_{kQ}U)}(\Lambda\otimes_{kQ}V).
    \end{align*}
    Hence, 
    $$F(g_V^{J(U)}) = g_{\Lambda\otimes_{kQ}V}^{J(\Lambda\otimes_{kQ}U)}: J(\Lambda\otimes_{kQ}U)\to J_{J(\Lambda\otimes_{kQ}U)}(\Lambda\otimes_{kQ}V),$$
    and therefore $F$ sends morphisms in $\Mcluster(kQ)$ to morphisms in $\Mcluster(\Lambda)$. Clearly, $$\Hom_{\Mcluster(\Lambda)}(F(J(U)),F(J(U))) = \Hom_{\Mcluster(\Lambda)}(J(\Lambda\otimes_{kQ}U),J(\Lambda\otimes_{kQ}U)) = g_0^{J(\Lambda\otimes_{kQ}U)}.$$
    Now, let $g_{V}^{J(U)}: J(U)\to J_{J(U)}(V)$ and $g_{W}^{J_{J(U)}(V)}: J_{J(U)}(V)\to J_{J(V,U)}(W)$ be two composable morphisms in $\Mcluster(kQ)$. By definition, 
    $$g_W^{J_{J(U)}(V)}\circ g_V^{J(U)} = g_{V\oplus \Eps^{-1}_V(W)}^{J(U)}$$ in $\Mcluster(kQ)$. Hence, 
    \begin{align*}
        F\left( g_W^{J_{J(U)}(V)}\circ g_V^{J(U)}\right ) &= F\left(g_{V\oplus \Eps^{-1}_V(W)}^{J(U)}\right)\\[6pt]
        &= g_{\Lambda\otimes_{kQ}(V\oplus \Eps_V^{-1}(W))}^{J(\Lambda\otimes_{kQ}U)}\\[6pt]
        &= g_{(\Lambda\otimes_{kQ}V)\oplus(\Lambda\otimes_{kQ}\Eps^{-1}_V(W))}^{J(\Lambda\otimes_{kQ}U)}\\[6pt]
        &= g_{(\Lambda\otimes_{kQ}V)\oplus \Eps^{-1}_{\Lambda\otimes_{kQ}V}(\Lambda\otimes_{kQ}W)}^{J(\Lambda\otimes_{kQ}U)}\\[6pt]
        &= g_{\Lambda\otimes_{kQ}W}^{J_{J(\Lambda\otimes_{kQ}U)}(\Lambda\otimes_{kQ}V)}\circ g_{\Lambda\otimes_{kQ}V}^{J(\Lambda\otimes_{kQ}U)}\\[6pt]
        &= F\left( g_{W}^{J_{J(U)}(V)}\right)\circ F\left(g _{V}^{J(U)}\right)
    \end{align*}
    in $\Mcluster(\Lambda)$. This shows that $F$ preserves composition. The claim follows.
\end{proof}

We are prepared to prove the main result of this section. 

\begin{theorem}\label{Mcluster(kQ)_equvalent_to_Mcluster(Lambda)}
    The functor $F: \Mcluster(kQ)\to \Mcluster(\Lambda)$ is an equivalence of categories.
\end{theorem}

\begin{proof}
    Let $J(X)\subseteq \modd \Lambda$ be a $\tau$-perpendicular subcategory of $\modd \Lambda$. By Corollary \ref{stautitl(kQ)1:1stautilt(Lambda)},  $X\cong \Lambda\otimes_{kQ}U$ for a unique support $\tau$-rigid object $U$ in $\C(kQ)$, and therefore $J(X) = J(\Lambda\otimes_{kQ}U)$. Hence, using the definition of $F$ (see Proposition \ref{functor_between_tau_cluster_morphism_categories}), we have that $J(X) = J(\Lambda\otimes_{kQ}U) = F(J(U))$. This shows that the functor $F$ is dense. 
    
    Now we show that $F$ induces a bijection on the $\Hom$-spaces. So let $$ \widetilde{g} = g_{\Lambda\otimes_{kQ}V}^{J(\Lambda\otimes_{kQ} U)}: J(\Lambda\otimes_{kQ} U)\to J_{J(\Lambda\otimes_{kQ}U)}(\Lambda\otimes_{kQ}V)$$ be a morphism in $\Hom_{\Mcluster(\Lambda)}(F(J(U)),F(J_{J(U)}(V)))$. Write $\widetilde{g} = g_{\widetilde{\U_1}}^{\widetilde{\W_1}}\cdots g_{\widetilde{\U_t}}^{\widetilde{\W_t}}$ as composition of irreducible morphisms in $\Mcluster(\Lambda)$, where $\widetilde{\W_t} = J(\Lambda\otimes_{kQ}U)$ and $\widetilde{\W_1} = J_{J(\Lambda\otimes_{kQ}U)}(\Lambda\otimes_{kQ}V)$. By Proposition \ref{irr_morph_in_Mcluster}, this decomposition corresponds to the signed $\tau$-exceptional sequence $\widetilde{\U} = (\widetilde{\U_1},\cdots, \widetilde{\U_t})$ in $J(\Lambda\otimes_{kQ}U)$ such that $\overline{\varphi_t}(\widetilde{\U}) = \Lambda\otimes_{kQ} V$. By Corollary \ref{bijection_signed_tau_exceptional_sequences} there exists a unique signed $\tau$-exceptional sequence $\V = (\V_1,\cdots, \V_t)$ such that $\Lambda\otimes_{kQ}\V = \widetilde{\U}$. In particular, $\overline{\varphi_t}(\V) = V$ by Proposition \ref{induction_commutes_with_varphi}. Applying Proposition \ref{irr_morph_in_Mcluster} again, $\V = (\V_1,\cdots, \V_t)$ yields a unique decomposition of $\tilde{g} = g_{V}^{J(U)}: J(U)\to J_{J(U)}(V)$ in $\Hom_{\Mcluster{(kQ)}}(J(U),J_{J(U)}(V))$ into irreducible morphisms $${g}  = g_{\V_1}^{{\W_1}}\cdots g_{\V_t}^{{\W_t}},$$  with ${\W_t} = J(U)$ and ${\W_1} = J_{J(U)}(V)$. Thus, for $\widetilde{g}\in \Hom_{\Mcluster(\Lambda)}(F(J(U)),F(J_{J(U)}(V)))$ there exists a unique ${g}\in \Hom_{\Mcluster(kQ)}(J(U),J_{J(U)}(V))$ such that $F({g}) = \widetilde{g}$. Indeed, suppose that ${g'}: J(U)\to J_{J(U)}(V)$ was another morphism in $\Mcluster(kQ)$ such that $F({g'}) = \widetilde{g}$. As before, write $${g'} = g_{{\V_1'}}^{{\W_1'}}\cdots g_{{\V_t'}}^{{\W_t'}}$$ as a composition of irreducible morphisms, where ${\W_t'} = J(U)$ and ${\W_1'} = J_{J(U)}(V)$. 
    By Proposition \ref{irr_morph_in_Mcluster}, this decomposition corresponds to the signed $\tau$-exceptional sequence ${\V'} = ({\V_1'}, \ldots, {\V_t'})$ such that $\overline{\varphi_t}({\V'}) = V$. Since $F$ preserves irreducible morphisms, we have 
    $$
    F({g'}) = F\left(g_{{\V_1'}}^{{\W_1'}}\right) \cdots F\left(g_{{\V_t'}}^{{\W_t'}}\right) = g_{\Lambda\otimes_{kQ}{\V_1'}}^{\widetilde{\W_1}} \cdots g_{\Lambda\otimes_{kQ}{\V_t'}}^{\widetilde{\W_t}}.
    $$
    By the assumption $F(g) = \tilde{g} = F(g')$, we obtain 
    $$ g_{\Lambda\otimes_{kQ}{\V_1}}^{\widetilde{\W_1}} \cdots g_{\Lambda\otimes_{kQ}{\V_t}}^{\widetilde{\W_t}} 
        = g_{\widetilde{\U_1}}^{\widetilde{\W_1}} \cdots g_{\widetilde{\U_t}}^{\widetilde{\W_t}} 
        = g_{\Lambda\otimes_{kQ}{\V_1'}}^{\widetilde{\W_1}} \cdots g_{\Lambda\otimes_{kQ}{\V_t'}}^{\widetilde{\W_t}}.$$

    and therefore, since 
    \begin{align*}
        \overline{\varphi_t}(\Lambda\otimes_{kQ}\V) = \overline{\varphi_t}(\U) = \Lambda\otimes_{kQ} V
        = \Lambda\otimes_{kQ}\overline{\varphi_t}(\V')
        = \overline{\varphi_t}(\Lambda\otimes_{kQ}\V')
    \end{align*}
    where the last equality follows from Proposition \ref{induction_commutes_with_varphi},  Proposition \ref{irr_morph_in_Mcluster} implies $\Lambda\otimes_{kQ}\V = \U = \Lambda\otimes_{kQ}\V'$. Hence, $\V = \V'$ by Corollary \ref{bijection_signed_tau_exceptional_sequences}. Since $\overline{\varphi_t}(\V) = V = \overline{\varphi_t}(\V')$, Proposition \ref{irr_morph_in_Mcluster} implies that 
    $$g_{{\V_1}}^{{\W_1}}\cdots g_{{\V_t}}^{{\W_t}} = g_{{\V_1'}}^{{\W_1'}}\cdots g_{{\V_t'}}^{{\W_t'}},$$
    and therefore $g = g'$.

    Hence, $F$ induces a bijection on the $\Hom$-spaces. In other words, $F$ is also fully faithful and therefore an equivalence. This concludes the proof. 
\end{proof}

\begin{remark}
    In recent work, Kaipel defined the $\tau$-cluster morphism category using a lattice theoretic description and showing this new notion is equivalent to the classical one \cite{kaipel2024tau}. Furthermore, he proved that two finite dimensional algebras $\Lambda$ and $\Gamma$ with isomorphic finite lattices of torsion classes have equivalent $\tau$-cluster morphism categories and there is a bijection between signed $\tau$-exceptional sequences in $\modd{\Lambda}$ and $\modd{\Gamma}$; see \cite[Cor. 4.9]{kaipel2024tau}. As pointed out in Remark \ref{central_element_rmk}, when $kQ$ is $\tau$-tilting finite (equivalently $\Lambda$ is $\tau$-tilting finite; see Corollary \ref{Lambda_tau-tilt_finite_iff_Q_Dynkin}), $\Lambda$ and $kQ$ have isomorphic lattices of torsion classes; see also \cite[Example 4.2(1)]{kaipel2024tau}. In this way, one can recover Theorem \ref{Mcluster(kQ)_equvalent_to_Mcluster(Lambda)} and Corollary \ref{bijection_signed_tau_exceptional_sequences} using a different approach from the theory developed in this paper. However, we emphasize that Theorem \ref{Mcluster(kQ)_equvalent_to_Mcluster(Lambda)} and Corollary \ref{bijection_signed_tau_exceptional_sequences} also hold in the more general case where $\Lambda$ is $\tau$-tilting infinite. 
\end{remark}

We want to conclude this section by proving that the following conjecture by Buan and Hanson holds for $\Lambda = R\otimes kQ$. 

\begin{conjecture}[{\cite[Conjecture 6.8]{tau-perpendicular_wide_subategories}}]\label{Buan-Hanson_Conjecture}
    Let $\Lambda$ be a finite dimensional algebra. Let $\W\subseteq\modd\Lambda$ be a $\tau$-perpendicular subcategory of $\modd\Lambda$ and let $\V\subseteq \W$ be a wide subcategory of $\W$. Then $\V$ is a $\tau$-perpendicular subcategory of $\modd\Lambda$ if and only if $\V$ is a $\tau$-perpendicular subcategory of $\W$. 
\end{conjecture}

We recall the following construction from \cite{Noncrossing_partitions_and_representations_of_quivers, Torsion-classes-wide-subcategories-and-localisations}. 

\begin{definition}
Let $\Lambda$ be a finite dimensional algebra.
    \begin{enumerate}
        \item[(a)] Let $\T\subseteq \modd\Lambda$ be a functorially finite torsion class. The \emph{left finite wide subcategory} of $\modd\Lambda$ corresponding to $\T$ is 
    $$\W_L(\T):= \{X\in \T \mid \forall (Y\in \T, g: Y\to X), \ker(g)\in \T\};$$
    \item[(b)]  Let $\F\subseteq \modd\Lambda$ be a functorially finite torsion-free class. The \emph{right finite wide subcategory} of $\modd\Lambda$ corresponding to $\F$ is 
    $$\W_R(\F):= \{X\in \T \mid \forall (Y\in \F, g:  X\to Y), \coker(g)\in \F\}.$$
    \end{enumerate}
\end{definition}

Let $\Lambda$ be a finite dimensional algebra and let $\mathcal{X}$ be a full subcategory of $\modd\Lambda$. We say that $M$ in $\mathcal{X}$ is \textit{split projective} in $\mathcal{X}$ if every epimorphism $N\to M$ in $\modd \Lambda$ with $N\in\mathcal{X}$ is split.

Recall there are mutually inverse bijections \cite[Thm. 2.7]{tau-tiling-theory} between $\text{s$\tau$-tilt }\Lambda$ and $\text{f-tors }\Lambda$ given by $\text{s$\tau$-tilt }\Lambda\ni T\mapsto \Gen T\in \text{f-tors }\Lambda$ and $\text{f-tors }\Lambda\ni\T\mapsto \P(\T)\in\text{s$\tau$-tilt }\Lambda$. Hence, given a functorially finite torsion class $\T$, there exists a unique support $\tau$-tilting object $M\oplus P[1]$ such that $\P(\T) = \add M$ and $\T = \Gen M$. 

\begin{lemma}[{\cite[Lemma 2.6, Prop. 2.7]{tau-perpendicular_wide_subategories}}]\label{left_finite_wide_subcat}
    Let $\Lambda$ be a finite dimensional algebra. Let $\T\subseteq \modd\Lambda$ be a functorially finite torsion class and let $M\oplus P[1]$ be the support $\tau$-tilting object in $\C(\Lambda)$ such that $\add M = \P(\T)$. Then, there is a decomposition $M = M_s\oplus M_{ns}$, where $M_s$ is split projective in $\T$ and no direct summand of $M_{ns}$ is split projective in $\T$. In particular, $\T=\Gen M_s$ and 
    $$\Gen M_s = \Gen M = {^\perp}\tau M\cap P^{\perp} = {^\perp}\tau M_{ns}\cap P^\perp.$$
\end{lemma}

It was shown in \cite[Lemma 4.3]{tau-perpendicular_wide_subategories} that every left finite wide subcategory of $\modd\Lambda$ is also a $\tau$-perpendicular subcategory. More precisely, every left finite wide subcategory is of the form $\W_L(\T) = J(M_{ns}\oplus P[1])$, for a functorially finite torsion class $\T$.
In particular, if $M$ is an indecomposable non-projective $\tau$-rigid $\Lambda$-module, $\W_L(\T) = \W_L(\Gen B_M) = J(M)$ by \cite[Prop. 6.15]{tau-perpendicular_wide_subategories}. We have the following observations. 

\begin{lemma}\label{induction_preserves_split_and_non-split}
    Let $\Lambda = R\otimes kQ$. Let $\Gen M \subseteq \modd kQ$ be a functorially finite torsion class and let $M\oplus P[1]$ be the support $\tau$-tilting object in $\C(kQ)$ such that $\P(\Gen M) = \add M$. Let $\Gen(\Lambda\otimes_{kQ}M)\subseteq \modd \Lambda$ be the corresponding functorially finite torsion class and let $(\Lambda\otimes_{kQ}M)\oplus (\Lambda\otimes_{kQ}P)[1]$ be the support $\tau$-tilting object in $\C(\Lambda)$ such that $\P(\Gen(\Lambda\otimes_{kQ}M)) = \add(\Lambda\otimes_{kQ}M)$. Write $M = M_s \oplus M_{ns}$ as in Lemma \ref{left_finite_wide_subcat}. Then $\Lambda\otimes_{kQ}M = (\Lambda\otimes_{kQ}M)_s\oplus (\Lambda\otimes_{kQ}M)_{ns}$, with $(\Lambda\otimes_{kQ}M)_s = \Lambda\otimes_{kQ}M_s$ and $(\Lambda\otimes_{kQ}M)_{ns} = \Lambda\otimes_{kQ}M_{ns}$. 
\end{lemma}

\begin{proof}
    Let $\Gen M \subseteq \modd kQ$ be a functorially finite torsion class and let $M\oplus P[1]$ be the support $\tau$-tilting object in $\C(kQ)$ such that $\P(\Gen M) = \add M$. Let $\Gen(\Lambda\otimes_{kQ}M)\subseteq \modd \Lambda$ be the functorially finite torsion class corresponding to $\Gen M$ under the bijection between $\ftors kQ$ and $\ftors \Lambda$, and let $(\Lambda\otimes_{kQ}M)\oplus (\Lambda\otimes_{kQ}P)[1]$ be the support $\tau$-tilting object in $\C(\Lambda)$ corresponding to $M\oplus P[1]$ under the bijection between $\stautilt kQ$ and $\stautilt \Lambda$ such that $\P(\Gen(\Lambda\otimes_{kQ}M)) = \add (\Lambda\otimes_{kQ}M)$; see Corollary \ref{stautitl(kQ)1:1stautilt(Lambda)}. Write $M = M_s\oplus M_{ns}$ as in Lemma \ref{left_finite_wide_subcat}, and similarly $\Lambda\otimes_{kQ}M = (\Lambda\otimes_{kQ}M)_s\oplus (\Lambda\otimes_{kQ}M)_{ns}$. 
    
    Let $X$ be an indecomposable direct summand of $M_{ns}$. Then, $\Lambda\otimes_{kQ}X$ is an indecomposable direct summand of $(\Lambda\otimes_{kQ}M)_{ns}$. Indeed, suppose that was not the case. Then, $\Lambda\otimes_{kQ}X$ must be a summand of $(\Lambda\otimes_{kQ}M)_s$. So let $f:N\to \Lambda\otimes_{kQ}X$ be a split epimorphism, with $N\in \Gen(\Lambda\otimes_{kQ}M)$, and let $f': \Lambda\otimes_{kQ}X\to N$ be such that $ff' = \id$. Applying the restriction of scalars, we get an epimorphism $\res f : \res N \to X^d$ with $\res(f)\res(f') = \res(ff')= \res(\id) = \id$. Hence, $\res f: \res N \to X^d$ is a split epimorphism with $\res N \in \Gen M$, a contradiction. 

    It remains to show that $\Lambda\otimes_{kQ}M_s$ is split projective in $\Gen(\Lambda\otimes_{kQ}M)$. Suppose there was a non-split projective epimorphism $g: V\to \Lambda\otimes_{kQ}M_s$, with $V\in \Gen(\Lambda\otimes_{kQ}M)$. Then for every map $g': \Lambda\otimes_{kQ}M_s\to V$, we have $gg' \neq \id$. Applying the restriction of scalar we get an epimorphism $\res g: \res V \to M_s^d$ with $\res V \in \Gen M$. Since the restriction of scalars is faithful, we have $\res(g)\res(g') = \res(gg') \neq \id$, which contradicts the fact that $M_s$ is split projective in $\Gen M$. The claim follows. 
\end{proof}

\begin{lemma}\label{induction_preserves_left_finite}
     The functor $F: \Mcluster(kQ)\to \Mcluster(\Lambda)$ preserves left finite wide subcategories. 
\end{lemma}

\begin{proof}
    Let $\W$ in $\Mcluster(kQ)$ be a left finite wide subcategory. Then, $\W = \W_L(\Gen M )$ for some functorially finite torsion class $\Gen M \subseteq \modd kQ$ where $M\oplus P[1]$ is the support $\tau$-tilting object  in $\C(kQ)$ such that $\P(\Gen M ) = \add M$. Let $\Gen(\Lambda\otimes_{kQ}M)\subseteq \modd\Lambda$ be the functorially finite torsion class corresponding to $\Gen M $ under the bijection between $\ftors kQ$ and $\ftors\Lambda$, and let $(\Lambda\otimes_{kQ}M)\oplus (\Lambda\otimes_{kQ}P)[1]$ be the support $\tau$-tilting module corresponding to $M\oplus P[1]$ under the bijection between $\stautilt kQ$ and $\stautilt \Lambda$ such that $\P(\Gen(\Lambda\otimes_{kQ}M)) = \add(\Lambda\otimes_{kQ}M)$; see Corollary \ref{stautitl(kQ)1:1stautilt(Lambda)}.  Write $M = M_s\oplus M_{ns}$ and $\Lambda\otimes_{kQ} M = (\Lambda\otimes_{kQ}M)_s\oplus (\Lambda\otimes_{kQ}M)_{ns}$ as in Lemma \ref{left_finite_wide_subcat}. Then
    \begin{align*}
        F(\W) = F(\W_L(\Gen M )) &= F(J(M_{ns}\oplus P[1])) &&\text{by \cite[Lemma 4.3]{tau-perpendicular_wide_subategories}}\\
        &= J(\Lambda\otimes_{kQ}(M_{ns}\oplus P[1])) &&\text{by the definition of $F$}\\
        &= J((\Lambda\otimes_{kQ}M_{ns})\oplus (\Lambda\otimes_{kQ}P)[1]) \\
        &= J((\Lambda\otimes_{kQ}M)_{ns}\oplus (\Lambda\otimes_{kQ}P)[1]) &&\text{by Lemma \ref{induction_preserves_split_and_non-split}}\\
        &= \W_L(\Gen(\Lambda\otimes_{kQ}M)) &&\text{by \cite[Lemma 4.3]{tau-perpendicular_wide_subategories}}.
    \end{align*}
    This finishes the proof. 
\end{proof}

\begin{lemma}{\cite[Page 977]{tau-perpendicular_wide_subategories}}\label{left_finite=tau_perp}
    Let $\Lambda$ be a finite dimensional algebra. Let $\V,\W\subseteq \modd\Lambda$ be $\tau$-perpendicular subcategories of $\modd\Lambda$ where $\V\subseteq \W$ is a wide subcategory of $\W$. Suppose that left finite wide and $\tau$-perpendicular subcategories of $\modd\Lambda$ coincide. Then, $\V$ is a $\tau$-perpendicular subcategory of $\W$. 
\end{lemma}

We need the following version of Theorem \ref{Mcluster(kQ)_equvalent_to_Mcluster(Lambda)}. 

\begin{proposition}\label{bijection_objects_M(kQ)_and_M(Lambda)}
    The functor $F: \Mcluster(kQ)\to \Mcluster(\Lambda)$ induces a bijection between the objects of $\Mcluster(kQ)$ and the objects of $\Mcluster(\Lambda)$.
\end{proposition}

\begin{proof}
    By the first part of the Proof of Theorem \ref{Mcluster(kQ)_equvalent_to_Mcluster(Lambda)}, $F$ is surjective on objects.
    
    We want to show that $F$ is also injective on objects. So, let $ F(J(U)) = \W = F(J(V))$. Then, $$\Hom_{\Mcluster(\Lambda)}(F(J(U)), F(J(V))) = \Hom_{\Mcluster(\Lambda)}(\W,\W) = g_0^{\W}$$ by Definition \ref{def_tau-cluster_morph_cat} (c)(ii). Since $F$ is fully faithful, $$ \Hom_{\Mcluster(kQ)}(J(U), J(V))\cong \Hom_{\Mcluster(\Lambda)}(F(J(U)), F(J(V))).$$ Let $g_X^{J(U)}: J(U)\to J_{J(U)}(X)$ be the unique morphism in $\Hom_{\Mcluster(kQ)}(J(U),J(V))$, such that $F(g_X^{J(U)}) = g_{0}^\W$. Notice that $J_{J(U)}(X) = J(V)$. By definition of $F$, we obtain that $$g_{0}^\W = F(g_X^{J(U)}) = g_{\Lambda\otimes_{kQ}X}^{J(\Lambda\otimes_{kQ}U)}.$$ It follows that $\Lambda\otimes_{kQ}X = 0$, which implies $X = 0$. Hence, $$J(V) = J_{J(U)}(X) = J_{J(U)}(0) = J(U).$$ We conclude that $F$ is also injective on objects. The claim follows. 
\end{proof}

As a consequence, we have the following result. 

\begin{corollary}\label{Buan-Hanson_conjecture_holds_for_Lambda}
    Let $\Lambda = R\otimes kQ$. The following statements hold.
    \begin{enumerate}
        \item[(i)] $\tau$-perpendicular categories, left finite wide subcategories, and right finite wide subcategories of $\modd \Lambda$ coincide;
        \item[(ii)] Conjecture \ref{Buan-Hanson_Conjecture} holds true for $\Lambda$. 
    \end{enumerate}
\end{corollary}

\begin{proof}
    (i) We show that left wide subcategories and $\tau$-perpendicular subcategories of  $\modd\Lambda$  coincide. The proof that right wide subcategories and $\tau$-perpendicular subcategories coincide is dual. We know that left finite wide subcategories are $\tau$-perpendicular subcategories; see \cite[Lemma 4.3]{tau-perpendicular_wide_subategories}. Moreover, left finite wide subcategories and $\tau$-perpendicular subcategories of $\modd kQ$ coincide by \cite[Cor. 2.17]{Noncrossing_partitions_and_representations_of_quivers}. By Proposition \ref{bijection_objects_M(kQ)_and_M(Lambda)}, $F$ induces a bijection between $\tau$-perpendicular subcategories of $\modd kQ$ and $\tau$-perpendicular subcategories of $\modd\Lambda$ and, since the functor $F$ preserves left finite wide subcategories by Lemma \ref{induction_preserves_left_finite}, it follows that $\tau$-perpendicular subcategories and left finite wide subcategories of $\modd\Lambda$ must coincide.  

    (ii) Let $\W\subseteq\modd\Lambda$ be a $\tau$-perpendicular subcategory of $\modd\Lambda$ and let $\V\subseteq \W$ be a wide subcategory of $\W$. If $\V$ is a $\tau$-perpendicular subcategory of $\W$, then $\V$ is also a $\tau$-perpendicular subcategory of $\modd\Lambda$ by \cite[Cor. 6.7]{tau-perpendicular_wide_subategories}. Conversely, assume that $\V$ is a $\tau$-perpendicular subcategory of $\modd\Lambda$. Then the result follows from (i) and Lemma \ref{left_finite=tau_perp}.
\end{proof}

\section{An example}\label{section6}

\begin{example}\label{example1}
    Let $R = kQ_R/I_R$ be the local commutative algebra where $Q_R$ is the quiver given by 
    
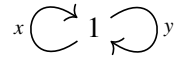
\[\begin{tikzcd}[ampersand replacement=\&,cramped]
	1
	\arrow["x", from=1-1, to=1-1, loop, in=145, out=215, distance=10mm]
	\arrow["y", from=1-1, to=1-1, loop, in=325, out=35, distance=10mm]
    \end{tikzcd}\] 
    and $I_R$ is the ideal generated by the relations  $x^2, y^2$, and $xy-yx$. Let $\Lambda = R\otimes kQ$, where $Q$ is the quiver given by $\begin{tikzcd} 1 \arrow[r, "a"] & 2 \end{tikzcd}$. Then $\Lambda$ is isomorphic to $kQ'/I$, where $Q'$ is the quiver 
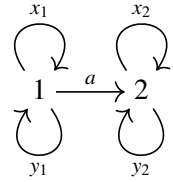
\[\begin{tikzcd}[ampersand replacement=\&,cramped]
	1 \& 2
	\arrow["{x_1}", from=1-1, to=1-1, loop, in=55, out=125, distance=10mm]
	\arrow["{y_1}", from=1-1, to=1-1, loop, in=235, out=305, distance=10mm]
	\arrow["a", from=1-1, to=1-2]
	\arrow["{x_2}", from=1-2, to=1-2, loop, in=55, out=125, distance=10mm]
	\arrow["{y_2}", from=1-2, to=1-2, loop, in=235, out=305, distance=10mm]
\end{tikzcd}\]
and $I = (x_i^2, y_i^2, x_i y_i-y_i x_i, ax_1-x_2a, ay_1-y_2a)$ for $i=1,2$. For each vertex $i$, we denote by $P^{kQ}_i, I^{kQ}_i, S^{kQ}_i$, the corresponding indecomposable projective (respectively, indecomposable injective, simple) module in $\modd kQ$, and by $P^{\Lambda}_i, I^{\Lambda}_i, S^{\Lambda}_i$, the corresponding indecomposable projective (respectively, indecomposable injective, simple) module in $\modd\Lambda$. The AR-quiver of the hereditary algebra $kQ$ can be depicted as follows 
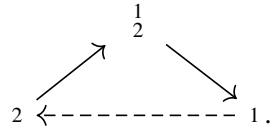
\[\begin{tikzcd}[ampersand replacement=\&,cramped]
	\&  \begin{array}{c} \begin{smallmatrix}1\\2\end{smallmatrix} \end{array} \\
	{ \begin{smallmatrix}2\end{smallmatrix}} \&\& {\begin{smallmatrix}1\end{smallmatrix}}.
	\arrow[from=1-2, to=2-3]
	\arrow[from=2-1, to=1-2]
	\arrow[dashed, from=2-3, to=2-1]
\end{tikzcd}\]
Note that every indecomposable $kQ$-module is $\tau$-rigid. Hence, using Proposition \ref{tau-rigid-Lambda1:1tau-rigid-kQ}, all the indecomposable $\tau$-rigid modules in $\modd\Lambda$ are given by 
$$ \Lambda\otimes_{kQ}S^{kQ}_1  = \begin{smallmatrix}&1&\\1&&1\\&1&\\&\end{smallmatrix} = I^\Lambda_1, \quad \Lambda\otimes_{kQ}P^{kQ}_1 = \begin{smallmatrix}&1&\\1&&1\\&1&\\&\\&2&\\2&&2\\&2&\\&\end{smallmatrix}= P^\Lambda_1, \quad \Lambda\otimes_{kQ}P^{kQ}_2  = \begin{smallmatrix}&2&\\2&&2\\&2&\\&\end{smallmatrix} = P^\Lambda_2.$$
Moreover, using Theorem \ref{bijection_of_tau_exc_seq} and Corollary \ref{bijection_signed_tau_exceptional_sequences}, signed $\tau$-exceptional sequences in $\modd \Lambda$ can be computed starting from ($\tau$-)exceptional sequences in $\modd kQ$ by applying the induction functor. We list them in the table below. 

\begin{center}
\setlength{\extrarowheight}{1.2mm}
\begin{tabular}{|c||c|}
\hline 
\begin{tabular}{c} Signed $(\tau)$-exc. \\ sequence in $\modd kQ$ \end{tabular} &
\begin{tabular}{c} Signed $\tau$-exc. \\ sequence in $\modd\Lambda$ \end{tabular}
\\ \hline 
$(P^{kQ}_1,S^{kQ}_1)$ $(P^{kQ}_1[1],S^{kQ}_1)$  & $(P^\Lambda_1,I^\Lambda_1)$ $(P^\Lambda_1[1],I^\Lambda_1)$ \\ \hline
\begin{tabular}{@{}c@{}}$(P^{kQ}_2, P^{kQ}_1)$ $(P^{kQ}_2, P^{kQ}_1[1])$ \\ $(P^{kQ}_2[1], P^{kQ}_1)$ $(P^{kQ}_2[1], P^{kQ}_1[1])$ \end{tabular} & \begin{tabular}{@{}c@{}}$(P^\Lambda_2,P^\Lambda_1)$ $(P^\Lambda_2,P^\Lambda_1[1])$ \\ $(P^\Lambda_2[1],P^\Lambda_1)$ $(P^\Lambda_2[1],P^\Lambda_1[1])$\end{tabular}  \\ \hline
\begin{tabular}{@{}c@{}}$(S^{kQ}_1, P^{kQ}_2)$ $(S^{kQ}_1, P^{kQ}_2[1])$ \\ $(S^{kQ}_1[1], P^{kQ}_2)$ $(S^{kQ}_1[1], P^{kQ}_2[1])$\end{tabular} & \begin{tabular}{@{}c@{}}$(I^\Lambda_1,P^\Lambda_2)$ $(I^\Lambda_1,P^\Lambda_2[1])$\\ $(I^\Lambda_1[1],P^\Lambda_2)$ $(I^\Lambda_1[1],P^\Lambda_2[1])$\end{tabular}  \\ \hline
\end{tabular}
\end{center}
\vspace{10pt}

Finally, combining Theorem \ref{Mcluster(kQ)_equvalent_to_Mcluster(Lambda)}, Proposition \ref{bijection_objects_M(kQ)_and_M(Lambda)},  and \cite[Thm. 6.16]{tau-perpendicular_wide_subategories}, the $\tau$-cluster morphism category of $\Lambda$ can be visualized in Figure \ref{fig1}.

\begin{figure}[h!]
    \centering
    \[\begin{tikzcd}[ampersand replacement=\&,cramped]
	\& {\boxed{\modd\Lambda}} \\
	\\
	{\boxed{J(P^\Lambda_1)\simeq \modd R}} \& {\boxed{J(I^\Lambda_1)\simeq \modd R}} \& {\boxed{J(P^\Lambda_2)\simeq \modd R}} \\
	\\
	\& {\boxed{0}}
	\arrow["{P^\Lambda_1[1]}"{description}, shift left=2, curve={height=-6pt}, shorten >=6pt, from=1-2, to=3-1]
	\arrow["{P^\Lambda_1}"{description}, shift right=3, curve={height=6pt}, from=1-2, to=3-1]
	\arrow["{I^\Lambda_1}"{description}, shorten <=3pt, from=1-2, to=3-2]
	\arrow["{P^\Lambda_2[1]}"{description}, shift right=2, curve={height=6pt}, shorten >=6pt, from=1-2, to=3-3]
	\arrow["{P^\Lambda_2}"{description}, shift left=3, curve={height=-6pt}, from=1-2, to=3-3]
	\arrow["{P^\Lambda_2[1]}"{description}, curve={height=-6pt}, shorten <=6pt, shorten >=6pt, from=3-1, to=5-2]
	\arrow["{P^\Lambda_2}"{description}, shift right=4, curve={height=6pt}, from=3-1, to=5-2]
	\arrow["{P^\Lambda_1[1]}"{description}, shift left=2, curve={height=-6pt}, shorten <=3pt, from=3-2, to=5-2]
	\arrow["{P^\Lambda_1}"{description}, shift right=2, curve={height=6pt}, shorten <=3pt, from=3-2, to=5-2]
	\arrow["{I^\Lambda_1[1]}"{description}, curve={height=6pt}, shorten <=6pt, shorten >=6pt, from=3-3, to=5-2]
	\arrow["{I^\Lambda_1}"{description}, shift left=4, curve={height=-6pt}, from=3-3, to=5-2]
    \end{tikzcd}\]
    \caption{The $\tau$-cluster morphism category of $\Lambda$.}
    \label{fig1}
\end{figure}
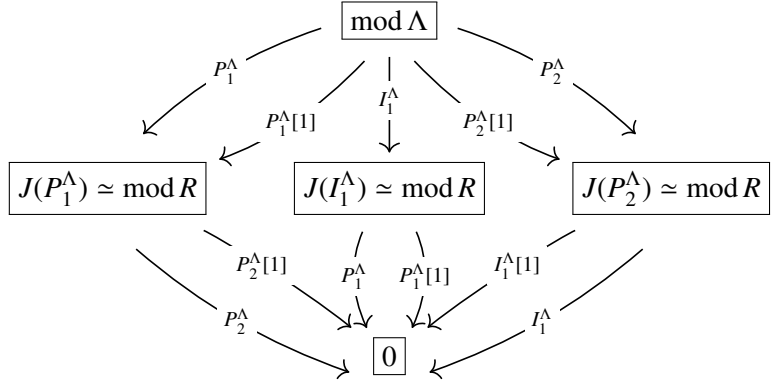

In particular, $J(M)\simeq \modd R$ for every $\tau$-rigid $\Lambda$-module $M$; see Theorem \ref{J(Lambda_otimes_kQ)=mod(R_otimes_kQ')}. 
\end{example}

\paragraph{\textbf{Acknowledgment:}} The author is grateful to Bethany R. Marsh for her supervision, for our countless inspirational discussions, and her meticulous proofreading. The author would also like to thank Aslak B. Buan and Francesca Fedele for helpful conversations. The author is also grateful for financial support from the University of Leeds and the EPSRC Programme Grant EP/W007509/1. \\

\maketitle

\nocite{*}
\bibliographystyle{alpha}
\bibliography{references}

\end{document}